\documentclass[12pt,twoside,reqno]{amsart}
\usepackage[utf8]{inputenc}
\usepackage[T1]{fontenc}
\usepackage[english]{babel}
\usepackage{tikz}
\usepackage{amsmath}
\usepackage{amsthm}
\usepackage{mathrsfs}
\usepackage{amsfonts,amssymb}
\usepackage{enumitem}
\usepackage[hypcap=false]{caption}
\usepackage[all]{xy}
\usepackage{url}
\usepackage{xspace}
\usepackage{subcaption}
\usepackage{textcomp}
 \usepackage{appendix}
\usepackage{pgf}
\usepackage{pdfpages}
\usepackage[left = 3cm, right = 3cm]{geometry}
\usepackage{hyperref}
\hypersetup{
colorlinks=true, 
breaklinks=true, 
urlcolor= blue, 
linkcolor= blue, 
citecolor=blue, 
} 
\usepackage{faktor}

\usetikzlibrary{arrows}
\usetikzlibrary[patterns]
\usetikzlibrary{matrix}

\theoremstyle{definition}
\newtheorem{definition}{Definition}[section]

\theoremstyle{plain}
\newtheorem{thm}{Theorem}

\newtheorem{prop}{Proposition}[section]
\newtheorem{lemma}{Lemma}[section]

\newtheorem{corollary}{Corollary}[section]

\theoremstyle{remark}
\newtheorem{remark}{Remark}

\newcommand{\Acal}{\mathcal{A}}
\newcommand{\A}{\textbf{A}}
\newcommand{\Ccal}{\mathcal{C}}
\newcommand{\C}{\mathbb{C}}

\newcommand{\D}{\mathcal{D}}
\newcommand{\Ecal}{\mathcal{E}}

\newcommand{\Fcal}{\mathcal{F}}

\newcommand{\Hil}{\mathcal{H}}
\newcommand{\Hsc}{\textsc{H}}
\newcommand{\Hg}{\textsc{\textbf{H}}}
\newcommand{\Ical}{\mathcal{I}}

\newcommand{\Kcal}{\mathcal{K}}

\newcommand{\Lie}{\mathcal{L}}
\newcommand{\Ncal}{\mathcal{N}}
\newcommand{\N}{\mathbb{N}}
\newcommand{\Ocal}{\mathcal{O}}

\newcommand{\R}{\mathbb{R}}

\newcommand{\Sn}{\mathbb{S}}

\newcommand{\Ucal}{\mathcal{U}}
\newcommand{\Vcal}{\mathcal{V}}
\newcommand{\Wcal}{\mathcal{W}}
\newcommand{\Z}{\mathbb{Z}}

\newcommand{\dt}{\frac{d}{dt}}

\newcommand{\w}[1]{\langle {#1} \rangle}
\newcommand{\ts}{\textsuperscript}

\DeclareMathOperator{\Card}{Card}
\DeclareMathOperator{\Ker}{Ker}

\DeclareMathOperator{\Op}{\textbf{Op}}
\DeclareMathOperator{\Opw}{Op}
\DeclareMathOperator{\Int}{Int}
\DeclareMathOperator{\ind}{ind}
\DeclareMathOperator{\Id}{Id}

\DeclareMathOperator{\re}{Re}
\DeclareMathOperator{\Res}{Res}

\DeclareMathOperator{\Ran}{Ran}

\DeclareMathOperator{\supp}{supp}

\def\rest#1#2{\mathchoice
              {\setbox1\hbox{${\displaystyle #1}_{\scriptstyle #2}$}
              \restrictionaux{#1}{#2}}
              {\setbox1\hbox{${\textstyle #1}_{\scriptstyle #2}$}
              \restrictionaux{#1}{#2}}
              {\setbox1\hbox{${\scriptstyle #1}_{\scriptscriptstyle #2}$}
              \restrictionaux{#1}{#2}}
              {\setbox1\hbox{${\scriptscriptstyle #1}_{\scriptscriptstyle #2}$}
              \restrictionaux{#1}{#2}}}
\def\restrictionaux#1#2{{#1\,\smash{\vrule height .8\ht1 depth .85\dp1}}_{\,#2}}

\newcommand{\executeiffilenewer}[3]{%
 \ifnum\pdfstrcmp{\pdffilemoddate{#1}}%
 {\pdffilemoddate{#2}}>0%
 {\immediate\write18{#3}}\fi%
}

\title{A Morse complex for Axiom A flows}
\author{Antoine Meddane}
\address{Laboratoire de mathématiques Jean Leray (UMR CNRS 6629), Université de Nantes, 2rue de la Houssinière, BP92208, 44322 Nantes Cédex 3, France}
\email{antoine.meddane@univ-nantes.fr}

\begin{document}

\begin{abstract}
 On a smooth compact Riemannian manifold without boundary, we construct a finite dimensional cohomological complex of currents that are invariant by an Axiom A flow verifying Smale's transversality assumptions. The cohomology of that complex is isomorphic to the De Rham cohomology via certain spectral projectors. This construction is achieved by defining anisotropic Sobolev spaces adapted to the global dynamics of Axiom A flows. In the particular case of Morse-Smale gradient flows, this complex coincides with the classical Morse complex. 
\end{abstract}

\maketitle
\begin{small}
\setcounter{tocdepth}{1}
\tableofcontents
\end{small}

\section{Introduction.}

Axiom A flows are a class of dynamical systems introduced by S. Smale \cite{smale} to describe chaotic dynamical systems. This type of flows arises in numerous physical problems and it contains two very interesting examples: Morse gradient flows and Anosov flows. On the one hand, the first one is well-known for its link with topology: notably through the Morse inequalities \cite{morse}, stated by Morse in $1920$, which relate the Betti numbers of the manifold and the number of critical points of a Morse function. Given a Riemannian metric on the manifold, Smale \cite{smale-morse} later gave another proof using dynamical arguments and ideas going back to Thom \cite{thom}. On the other hand, Anosov flows were defined first by Anosov in \cite{anosov} to describe the properties satisfied by geodesic flows on negatively curved manifolds and their links with the topology of the manifold are more sublte. 

From a purely dynamical point of view, Axiom A flows are interesting because, once they satisfy a generic condition called the \textit{strong transversality assumption}, they form an open subset of the set of vector fields. Precisely, for any vector field on a compact manifold without boundary, we have the equivalence: 
$$ \mbox{Axiom A + strong tranversality assumption } \Longleftrightarrow \: \mbox{$\Ccal^1$-structurally stable}.  $$
This equivalence is often refered to the $\Ccal^1$-structural stability conjecture which was solved by Robinson \cite{robinson} for the first implication and by Wen \cite{wen} and Hayashi \cite{hayashi} for the converse statement. The proof of the $\Ccal^1$-structural stability conjecture for diffeomorphisms was previously obtained by Robbin \cite{robbin} for the first implication and by Ma\~{n}e \cite{mane} for the converse statement. A great review on structural stability conjectures can be found in the book of Wen \cite{wen}.

In another direction, the concept of currents\footnote{In coordinates, differential forms with value in the set of distributions. We refer to the book of Schwartz \cite{schwartz} and the lecture notes of Laundenbach \cite{laudenbach-book} for a comprehensive introduction.} turns out to be very useful in the study of gradient flows. More precisely, Laudenbach \cite{laudenbach} and Harvey-Lawson \cite{harvey-lawson} gave a new interpretation of Morse homology in terms of currents by proving the following statement. Let us consider a smooth compact Riemannian manifold $(M,g)$ of dimension $n$ and a smooth Morse function $f$. If $x$ denotes a critical point of index $1 \leq k \leq n$, then the stable manifold $W^s (x)$ is an embedded submanifold of dimension $k$ and we have (in the sense of currents)
\begin{equation}\label{eq: intro derivative of De Rham currents}
 \partial [W^s (x)] = [\partial W^s (x)] = \sum_{\ind y = \ind x - 1 \atop y \mbox{ \Tiny{critical point}}} n(x,y)[W^s (y)]
\end{equation}
for some $n(x,y) \in \Z$ often called the \textit{instanton numbers}. For every $0 \leq \ell \leq n$, the space $\D'^{,n-\ell} (M)$ of current of degree $n-\ell$ is defined as the topological dual of the space of differential $\ell$-forms $\Omega^{\ell} (M)$, i.e. the space of smooth sections $\Gamma (M; \Lambda^{\ell} T^* M)$. An equivalent formulation for equation (\ref{eq: intro derivative of De Rham currents}) is given as follows
$$\forall \omega \in \Omega^{k-1} (M), \quad \int_{W^s(x)} d\omega = \sum_{\ind y = \ind x - 1 \atop y \mbox{ \Tiny{critical point}}} n(x,y) \int_{W^s (y)} \omega .$$
This relation is often presented in the following algebraic form. Consider the differential on the complex of critical points defined by
$$\partial x = \sum_{\ind y = \ind x - 1 \atop y \mbox{ \Tiny{critical point}}} n(x,y) . y \quad \mbox{for every }x \in C_k (f) = \{ \mbox{critical points of index }k \},$$
with the same numbers $n(x,y)$ as before. The cohomological complex $(C(f),\partial)$ is referred as the Morse complex and it is in fact quasi-isomorphic to the de Rham complex, in the sense that the cohomology groups are the same. A remarkable feature of (\ref{eq: intro derivative of De Rham currents}) is that it gives a representation of this algebraic complex in terms of currents that are invariant by the gradient flow. Morse inequalities have been generalized to more general dynamical systems as one can witness in the book of Franks \cite{franks_homology}. Yet, to the best of our knowledge, the previous algebraic procedure does not extend to Axiom A flows and there is no analogue of its analytical version as it was constructed by Laudenbach. In this direction, we can mention the article of Ruelle and Sullivan \cite{ruelle-sullivan} in which they constructed similar closed invariant currents for Axiom A diffeomorphisms. Nevertheless, their construction was only local (near a basic set) and was not enough to recover the whole De Rham cohomology. More recently, Dang and Rivière showed how to use the theory of Ruelle resonances \cite{baladi_2000} to define a natural cohomological complex of currents associated with Morse-Smale and Anosov flows which are two particular examples of Axiom A flows. In this article, we show how to extend this construction and we associate a natural cohomological complex to every Axiom A flow verifying Smale's transversality assumptions:
\vspace{0,3cm}

\fbox{
\begin{minipage}{0.9\linewidth}
\begin{thm}\label{thm intro}
 Let $V$ be an Axiom A vector field which satifies the Smale transversality assumption (\ref{Transversality assumption}). For every $0\leq k\leq n$, there exists a Hilbert space $\Omega^k (M;\C) \subset \Hil_k \subset \D^{\prime,k} (M;\C)$ (with continuous injection) and a positive integer $m_k (0)$ such that 
 $$ \cdots \overset{d}{\longrightarrow} \Ker \left( \rest{(\Lie_V^{(k)})^{m_k (0)}}{\Hil_k} \right) \overset{d}{\longrightarrow} \Ker \left( \rest{(\Lie_V^{(k+1)})^{m_{k+1} (0)}}{\Hil_{k+1}} \right) \overset{d}{\longrightarrow} \cdots ,$$
 defines a finite dimensional cohomological complex quasi-isomorphic to the De Rham complex.
\end{thm}
\end{minipage}
}
\vspace{0,3cm}

Equivalently, this Theorem shows the existence of a finite dimensional complex representing the De Rham cohomology and generated by dynamical currents that are invariant by the Axiom A flow. The Hilbert spaces $\Hil_k$ of the statement are anisotropic Sobolev spaces adapted to the spectral analysis of transfer operators,
$$ \Lie^t (u) =  u \circ \varphi^{-t} \qquad \forall u \in \Ccal^{\infty} (M), $$
as it was initiated by Ruelle \cite{ruelle, ruelle_book}.
Recall that this operator extends to the spaces of differential forms by setting
$$ \Lie_{(k)}^t (u):= (\varphi^{-t})^* (u), \qquad \forall u \in \Omega^k (M). $$
In order to prove theorem \ref{thm intro}, we can introduce the \textit{resolvent operator}  
$$ ( \Lie_V^{(k)} + z )^{-1} := \int_0^{+\infty} e^{-tz} (\varphi^{-t})^* dt : \Omega^k (M;\C) \longrightarrow \D'^{,k} (M;\C), $$
which is well-defined for $\re (z) \geq C_0 \gg 1$ and which defines a holomorphic function in that same region. Then, we deduce theorem \ref{thm intro} from the fact that the resolvent operator $( \Lie_V^{(k)} + z )^{-1}$ admit a meromorphic extension to a halfplane $\re (z) > -C_k$, with $C_k \gg 1$ arbitrary large so that $0$ belongs to that half plane. Precisely, we prove

\vspace{0,3cm}

\fbox{
\begin{minipage}{0.9\linewidth}
\begin{thm}\label{thm 2 intro}
 Let $V$ be an Axiom A vector field which satifies the Smale transversality assumptions (\ref{Transversality assumption}). For every $0\leq k \leq n$, the resolvent operator 
 $$z \longmapsto( \Lie_V^{(k)} + z )^{-1} : \Omega^k (M;\C) \longrightarrow \D'^{, k}(M;\C)$$
 extends meromorphically from $\text{Re}(z)\geq C_0$ to the whole complex plane $\C$.
\end{thm}
\end{minipage}
}
\vspace{0,3cm}

We call \textit{resonances} the poles of the resolvent operators $( \Lie_V^{(k)} + z )^{-1}$ seen as a meromorphic function on all $\C$. In order to prove that $(\Lie_V^{(k)} + z)^{-1}$ admits a meromorphic extension, one needs to find good Hilbert spaces $\Hil_k$ on which the Lie derivative operators $\Lie_V^{(k)}$ have discrete spectrum of resonances. More precisely, using \textit{analytic Fredholm theory} \cite[app. C]{DZ-book}, we will extend the resolvent operator on the half planes $\text{Re}(z)\geq -C_k$ (for arbitrarly large $C_k$) as a meromorphic family of Fredholm operators. In particular, we will find that the residue of $(\Lie_V^{(k)} + z)^{-1}$ at $z = 0$ is a finite rank projector whose explicit expression is given by
$$ \pi_0^{(k)} := \Res_{z = 0} ( - \Lie_V^{(k)}) = \frac{-1}{2 i \pi} \int_{\gamma_0} (\Lie_V^{(k)} + z)^{-1} dz : \Omega^k (M;\C) \longrightarrow \D'^{,k}(M;\C).$$
Here, $\gamma_0$ denotes a positively oriented closed curve which surrounds the resonance $0$ and no other resonances. The spaces $\Ker \left( \rest{(\Lie_V^{(k)})^{m_k (0)}}{\Hil_k} \right)$ which appear in Theorem \ref{thm intro} are thus defined as the range of the finite rank projectors $\pi_0^{(k)}$. The quasi-isomorphism with the De Rham complex is then given by the maps $\pi_0^{(k)}$. Therefore, Theorem \ref{thm intro} will be deduced from Theorem \ref{thm 2 intro} which relies on the construction of the Hilbert space $\Hil_k$ presented earlier and which is the main technical issue of this article.

\subsection{Earlier results on meromorphic continuations for Axiom A flows}

These functional constructions were originally made by Ruelle using Markov partitions in view of studying the mixing properties of these dynamical systems and the analytical properties of the zeta functions associated with their periodic orbits. Similar results have been obtained by Bowen, Fried, Rugh, Dolgopyat and others \cite{bowen2}, \cite{fried_meromorphic}, \cite{rugh}, \cite{dolgopyat} also relying on symbolic dynamic.

Building on some earlier work \cite{BKL} with Blank and Keller for Anosov diffeomorphisms, Liverani introduced in \cite{liverani} Banach spaces of distributions with anisotropic Hölder regularity adapted to contact Anosov flows. Among other things, these spaces allowed him to prove the meromorphic continuation of the resolvent in some half plane slightly beyond $\re(z)=0$. This approach was further developped in subsequent works with Butterley \cite{butterley-liverani} and Giuletti-Pollicott \cite{GLP} to show the meromorphic continuation of $(\Lie_V^{(k)}+z)^{-1}$ to the whole complex plane for every $0\leq k\leq n$ and for any smooth Anosov flow. We also refer to \cite{gouezel-liverani} for earlier results with Gouëzel in that spirit for Axiom A diffeomorphims. We can also mention \cite{baladi-tsujii_2007} for a different approach (still for diffeomorphisms) by Baladi and Tsujii using anisotropic Sobolev spaces and methods from Fourier/microlocal analysis.
 
 From another direction, the general theory of semiclassical resonances \cite{DZ-book}, \cite{helffer-sjostrand} was used to derive alternative approaches to construct Hilbert spaces adapted to the dynamics. First, for smooth Anosov diffeomorphisms, Faure, Roy and Sjöstrand recovered in \cite{FRS} the existence of a discrete spectrum for the tranfer operator. Then for general Anosov flows, Faure and Sjöstrand constructed in  \cite{FS} Hilbert spaces, referred as \textit{anisotropic Sobolev spaces}, on which the Lie derivative $\Lie_V^{(0)}$ has discrete spectrum on a large part of the complex plane. Their analysis used the machinery of microlocal analysis as a toolbox and it reduced in some sense the problem to a dynamical question, i.e. constructing an escape (or Lyapunov) function adapted to the dynamics on the cotangent space. In \cite{FS}, the meromorphic extension of the resolvent was obtained for $k = 0$ and this result has been extended for every $0 \leq k \leq n$ in \cite{dyatlov-zworski} by Dyatlov and Zworski in view of applications to the Ruelle zeta function. More informations on the spectrum (e.g. band structure) have also been obtained by Faure and Tsujii \cite{tsujii_2010,tsujii_2012, faure-tsujii2013,faure-tsujii2017,faure-tsujii2019,faure-tsujii2021} in the context of Anosov flows and contact Anosov flows using these kinds of methods.
 
 Subsequently, the meromorphic extension of the resolvent has been extended to Axiom A flows by Dyatlov-Guillarmou \cite{dyatlov-guillarmou_2016, dyatlov-guillarmou_2018} under the assumption that the resolvent acts on differential forms supported near a fixed basic set. Although this analysis was enough to prove the meromorphic continuation of the Ruelle zeta function for Axiom A flows, it does not seem to be sufficient to deduce Theorem \ref{thm 2 intro} (and thus Theorem \ref{thm intro}) which does not require any support restriction. In that direction, Dang and Rivière \cite{DR0,DRI,DRII,DRIII} proved the meromorphic continuation for globally supported test forms in the case of Morse-Smale gradient flows, a generic subset of Morse gradient flows which satisfies a so called \textit{transversality assumption}. In this series of articles, Dang and Rivière gave a complete description of Pollicott-Ruelle resonances, giving a band structure for the spectrum, computed the dimensions of eigenspaces by expliciting the eigenvectors in terms of De Rham's currents and gave a new proof of Morse-Smale inequalities. The link with the topology was made possible through their \textit{global} construction of Sobolev spaces adapted to the dynamics of Morse-Smale flows. In this article, we gathered the two approaches of Dyatlov-Guillarmou and Dang-Rivière so that we obtain the meromorphic continuation of the resolvent acting on globally supported forms for general Axiom A flows. Finally, we emphasize that, besides its applications to topology, Theorem \ref{thm 2 intro} also answers a question raised by Baladi in \cite[4.32, p. 144]{baladi}. 
 
 \subsection{Back to topology}
 
 Recently, the developments of these analytical tools lead to many progress on the link between the topology of the manifold and the spectrum of the Lie derivative, at least for the examples where the functional setup was globally defined, namely Morse-Smale and Anosov flows. We recall here some of these progress. 
 \begin{itemize}
  \item \textbf{Contact Anosov flows in dimension $3$.} In that geometric framework, Dyatlov and Zworski \cite{dyatlov-zworski_2017} computed the dimension of $\Ker (\Lie_V^{(k)})^{m_k(0)}$ for every $0\leq k\leq 3$ and expressed it in terms of the Betti numbers of the manifold. They used this to generalize earlier results of Fried \cite{fried_lefschetz} on the order of vanishing of the Ruelle zeta function. In particular, their computation holds true for any geodesic flow acting on the unitary cotangent bundle $S^*\Sigma =: M$ of a compact negatively curved surface $\Sigma$. Burns-Weil-Shen \cite{borns-weil-shen} extended \cite{dyatlov-zworski_2017} to the nonorientable case and Hadfield \cite{hadfield} showed a similar result for compact negatively curved surfaces with boundaries.
  \item \textbf{Anosov flows in high dimension.} Küster-Weich \cite{kuster-weich} computed the dimension of $\Ker (\Lie_V^{(1)})^{m_1(0)}$ in terms of the first Betti number for hyperbolic manifolds of dimension $\neq 3$.
  \item \textbf{Perturbation of Anosov flows.} Ceki\'{c}-Paternain \cite{cekic-paternain} gave the first examples of Anosov flows in dimension $3$ which preserves a volume form where the vanishing order of the Ruelle zeta function jumps under perturbation of the flow. Again this was achieved by computing explicitely the dimension of the spaces appearing in the cohomological complex of Theorem \ref{thm 2 intro}. In dimension $5$, Ceki\'{c}-Dyatlov-Küster-Paternain \cite{CDKP} found a similar result for geodesic flows on nearly hyperbolic $3$-manifold (the unitary cotangent bundle is $5$-dimensional).
  \item \textbf{Fried conjecture \cite{fried_lefschetz,fried_meromorphic}.} Dang-Guillarmou-Rivière-Shen \cite{DGRS} established, in the case of Anosov flows, a criterium in terms of the spaces appearing in Theorem \ref{thm intro} to ensure that the value at $0$ of the twisted Ruelle zeta function is locally constant. It allowed them to prove Fried conjecture on the Reidemeister torsion for nearly hyperbolic $3$-manifolds. This was further pursued by Chaubet and Dang \cite{chaubet-dang} who used the cohomological complex of Theorem \ref{thm intro} to define a dynamical torsion for contact Anosov flows in any dimension.
  \item \textbf{Morse-Smale flows.} Dang-Rivière \cite{DRIII} proved Theorem \ref{thm intro} in the case of Morse-Smale and Anosov flows. In the specific case of Morse-Smale gradient flows \cite{DR_witten}, they also considered the Lie derivative operator as a limit of the Witten Laplacian and they obtained the Ruelle spectrum as a limit of the Witten spectrum. It allowed them to recover the Witten-Helffer-Sjöstrand instanton formula and to prove the Fukaya conjecture on Witten deformation of the wedge product.
  \end{itemize}

 \subsection{Outline of the proof} We use the microlocal approach to Pollicott-Ruelle resonances of the Lie derivative operator $\Lie_V$ as it was developped by Faure and Sjöstrand. Recall that the proof of Theorem \ref{thm 2 intro} relies on the construction of Hilbert spaces adapted to the dynamics. Following \cite{FS}, defining such spaces can be reduced through some microlocal procedure to the construction of an escape function. More precisely, one has to exhibit a family of functions that are decreasing along the Hamiltonian flow of $H(x,\xi)= \xi(V(x))$ on the cotangent bundle $T^*M$ of $M$. The existence of such decreasing functions, called energy or Lyapunov functions, is already known for the flow on $M$ as soon as $V$ is an Axiom A flow verifying Smale's transversality assumption. We can cite for example the articles of Conley \cite{conley}, Wilson \cite{wilson} for flows and Pugh-Shub \cite{pugh-shub} for Axiom A diffeomorphisms verifying Smale's tranversality assumptions. One of the main novelty of this article is to do the same for the induced Hamiltonian flow on $T^*M$. It was already done by Faure-Sjöstrand \cite{FS} for Anosov flows, by Dyatlov-Guillarmou \cite{dyatlov-guillarmou_2016} near a basic set and by Dang-Rivière in \cite{DR0,DRI} for Morse-Smale flows. To construct a decreasing function along this Hamiltonian flow, Dang and Rivière highlighted in the case of gradient flows \cite{DR0} that one needs to prove the compactness of the conormal distribution 
 $$ \bigcup_{x \in M} \left\{ \xi \in S_x^* M \:: \: \xi (T_x W^u (x_{-})) = 0, \mbox{ for } x_- \mbox{ the critical point s.t } x \in W^u (x_{-}) \right\},$$
 where $W^u (x_-)$ denotes the unstable manifold of the critical point $x_-$. Nevertheless, to do so, they made a restriction on the class of Morse-Smale flows which is the existence of $\Ccal^1$\textit{-linearization charts} near critical points. Such a restriction is not avalaible for more general Axiom A flows and we need to proceed differently. In particular, we note that our proof allows to remove this linearization assumption in the specific case of Morse-Smale flows. To prove a similar result for Axiom A flows, we proceed in three steps.
 \begin{itemize}
  \item We define a transversality assumption for Axiom A flows which generalizes the one used for Morse-Smale gradient flows.
  \item Then, we generalize the compactness result for conormal distributions without using $\Ccal^1$-linearizing charts. This step will require a similar analysis than the local analysis near basic sets performed by Dyatlov-Guillarmou \cite{dyatlov-guillarmou_2016}.
  \item We deduce the existence of a \textbf{global} escape functions for Axiom A flows which satisfies the transversality assumption by adapting the construction of Faure-Sjöstrand \cite{FS}.
 \end{itemize}

Concerning the proof of Theorem \ref{thm intro}, we recall that that there is a strong analogy with the Hodge-De Rham Laplace operator\footnote{The derivatives $d^*$ denotes the formal adjoint of $d$ in $L^2 (M; \Lambda^k T^* M)$.} $\Delta = d \circ d^* + d^* \circ d = (d + d^*)^2$ acting on differential forms $\Omega^* (M)$ if we remark that
$$ \Lie_V = d\circ \iota_V + \iota_V \circ d = (d + \iota_V )^2 .$$
Note also that both operators $\Delta$ and $\Lie_V$ commutes with the exterior derivative $d$. These analogies are at the heart of the proof of Theorem \ref{thm intro}.

\subsection{Organization of the article}

\begin{itemize}
 \item In section \ref{section Dynamical preliminaries}, we recall the definition of an Axiom A flow and introduce the dynamical tools we will need. Furthermore, we  present in this part a few key notions for our analysis which turn out to be related: Smale's order relation on basic sets, transversality assumption, filtrations (with open sets) and unrevisited neighborhoods. We also explain how to bypass the $\Ccal^1$-linearizing charts used in Dang-Rivière's articles.
 \item In section \ref{section Escape function}, we present a possible construction of an escape function and we state a generalization of the compactness result for conormal distributions which takes into account the neutral direction given by the flow direction. The result stated in this part were in fact the most challenging ones to prove. 
 \item In section \ref{section Anisotropic Sobolev spaces}, we define anisotropic Sobolev spaces, in which the Lie derivative operators $\Lie_V^{(k)}$ have discrete spectrum (see Theorem \ref{thm: discrete sprectrum} from which Theorem \ref{thm 2 intro} derives).
 \item In section \ref{section topology}, we state and prove the main topological result of this article, namely Theorem \ref{thm intro}. 
 \item In section \ref{section: Dynamical proofs}, \ref{section : Compactness result and energy functions for the Hamiltonian flow} and \ref{section : Proof of escape function} we give the proof of the dynamical results such as the construction of energy functions for Axiom A flows, the proof of the compactness of conormal distributions and the construction of the global escape functions.
\end{itemize}

\subsection{Acknowlegments} The author would like to warmly thank Gabriel Rivière for many explanations about his work with Nguyen Viet Dang and for his careful reading and remarks which contributed a lot to improve this paper. This work was partly supported by the Institut Universitaire de France and by the Agence Nationale de la Recherche through the PRC project ADYCT (ANR-20-CE40-0017).

\section{Dynamical preliminaries.}\label{section Dynamical preliminaries}

In all this paper, we denote by $(M^n,g)$ a smooth compact Riemannian manifold without boundary of dimension $n\geq 1$ and associated to some smooth Riemannian metric $g$. We also denote by $d_g$ the geodesic distance associated to the metric $g$ and by $|.|_g = \sqrt{g(. \: ,.)}$ the norm induced on the fibers of the tangent bundle $TM$ or of the cotangent bundle $T^*M$. To a smooth vector field $V \in \Gamma (TM)$, we can associate a flow $(\varphi^t )_{t \in \R}$ which solves the Cauchy problem: 
\begin{equation}\label{eq: EDO flow}
\forall x \in M, \:  \forall t \in \R, \quad
\left\{
    \begin{array}{ll}
        \dt \varphi^t (x) = V (\varphi^t (x)) \\
        \varphi^0 (x) = x
    \end{array}
\right. .
\end{equation}
The system (\ref{eq: EDO flow}) is highly non-linear in general which makes difficult to predict the large-time behavior of trajectories, especially in the case of hyperbolic dynamics.

\begin{definition}\cite[p.796]{smale}
 A point $x \in M$ is said to be \textbf{nonwandering} if for every neighborhood $\Ucal$ of $x$ and every $T > 0$ there exists $t \in \R$ such that $|t| \geq T$ and $\varphi^t (\Ucal) \cap \Ucal \neq \emptyset$. The nonwandering points form a closed invariant subset of $M$, called the \textbf{nonwandering set}, and we will denote it by $\Omega:= \Omega (\varphi^t)$.
\end{definition}

We refer to appendix \ref{Appendix: hyperbolic set} or to the books \cite{palis-demelo}, \cite{katok-hasselblatt} for a definition of hyperbolic set. 

\begin{definition}[{Axiom A flow, \cite[p.803]{smale}, \cite{pugh-shub}}]
A flow $\varphi^t: M \rightarrow M$ is said to be \textbf{Axiom A} if its nonwandering set $\Omega$ is hyperbolic and can be written as $\Omega = \Fcal \sqcup \Kcal$ where
\begin{enumerate}[label = (\roman*)]
 \item $\Kcal = \overline{Per (\varphi^t)}$ is the closure of all periodic orbits. 
 \item $\Fcal$ is the set of fixed point for the flow $\varphi^t$, which is assumed to be \textbf{finite}.
\end{enumerate}
\end{definition}

It is known from the works of Smale and Bowen that an Axiom A flow has a nonwandering set which splits into a finite number of hyperbolic invariant compact sets, which are either reduced to a fixed point or invariant hyperbolic sets called basic sets: 

\begin{prop}[{Spectral decomposition , \cite[\S II.5]{smale}, \cite{bowen2}}]
 If $\varphi^t$ is an Axiom A flow, then its nonwandering set $\Omega$ decomposes into a finite union of \textbf{basic} sets $K_i$: 
 $$ \Omega = K_1 \sqcup K_2 \sqcup \cdots \sqcup K_N$$
 where $K$ basic means:
 \begin{itemize}
  \item $K$ is compact and hyperbolic.
  \item $K$ is locally maximal: there exists some open set $O\subseteq M$ such that
 $$ K = \bigcap_{t \in \R} \varphi^t (O) .$$
  \item $K$ est topologically transitive, i.e. there exists a point $x \in K$ such that $\overline{(\varphi^t(x))}_{t \in \R_+} = K$.
  \item $K$ is the closure of its periodic orbits which can potentially be reduced to a fixed point, i.e. have period $0$.
 \end{itemize}
\end{prop}

From now on, $\varphi^t$ will denote an Axiom A flow on $(M,g)$. We call \textit{attractor} for $\varphi^t$ a basic set $K$ which satifies
$$K = \bigcap_{t \in \R_+}  \varphi^t (O), $$
for some open set $O \supset K$. Similarly, we call \textit{repeller} for $\varphi^t$ a basic set $K$ which satifies
$$K = \bigcap_{t \in \R_-}  \varphi^t (O), $$
for some open set $O \supset K$. 

\begin{remark}\label{rk: fixed point isolated}
 The fixed points are isolated subsets of $\Omega$ thanks to the \textit{expansiveness} property of the flow on basic sets. See \cite[lemma 1 p.181]{bowen-walters} or \cite{bowen} for more details.
\end{remark}

\subsection{Stable and unstable manifolds} We begin by recalling some well-known facts concerning uniformly hyperbolic dynamics which can be found in \cite{katok-hasselblatt} or \cite{dyatlov}. Fix a basic set $K$. For all $\varepsilon > 0$ and all $z \in K$, the \textit{stable manifold}, \textit{local stable manifold} and \textit{local weak stable manifold} at the point $z$ are defined by 
\begin{align*}
 W^s (z) &:= \{ x \in M:  \: d_g(\varphi^t(x), \varphi^t(z)) \underset{t \to+\infty}{\longrightarrow} 0 \}\\
 W_{\varepsilon}^s (z) &:= \{ x \in  W^s (z): \: d_g(\varphi^t(x), \varphi^t(z)) < \varepsilon, \: \:  \forall t \in \R_+ \}\\
 W_{\varepsilon}^{so} (z) &:= \{ x \in  M: \: d_g(\varphi^t(x), \varphi^t(z)) < \varepsilon, \: \:  \forall t \in \R_+ \}.
\end{align*}

By replacing $s$ by $u$ and $\varphi^t$ by $\varphi^{-t}$ in the previous equalities, we could have defined similarly the $\emptyset$/local/local weak unstable manifolds. From this remark, let us only deal with stable manifolds by keeping in mind that everything can be adapted for unstable manifolds. Thanks to the Hadamard-Perron theorem, also called stable manifold theorem, there exists $\varepsilon_0 \ll 1$ such that, for all $z \in K$, the sets $W_{\varepsilon_0}^{s/so} (z)$ are smooth submanifolds of $M$ of dimension $d_{s/so}$, which is constant on each basic set. Precise statements and proof of this result can be found in \cite[thm 6.2.8, p. 242]{katok-hasselblatt} for the case of diffeomorphisms and in \cite[thm 5, p. 34]{dyatlov} for the case of flows. In general, stable manifolds are not embedded submanifolds but only immersed submanifolds, except in the case of Morse flows. Moreover, the stable manifold is related to the local stable manifold thanks to the following formula \cite[p. 24]{dyatlov}: 
$$
 W^s (z)  = \bigcup_{n \geq 0} \varphi^{-n} \left( W_{\varepsilon_0}^s (\varphi^{n} (z)) \right) , \quad W^{so} (z)  = \bigcup_{n \geq 0} \varphi^{-n} \left( W_{\varepsilon_0}^{so} (\varphi^{n} (z)) \right), \quad (n \in \N),
$$
which does not depend on $\varepsilon_0$ given by the stable manifold theorem. If $K$ denotes a basic set, then we define its \textit{stable manifold}\footnote{which is not a submanifold anymore, except in particular cases where for example $K$ is reduced to a fixed point or a single closed orbit.} by 
$$ W^{s} (K):= \{ x \in M: \: d(\varphi^{ t} (x), K) \underset{t \to +\infty}{\longrightarrow} 0 \}.$$
Thanks to the shadowing lemma \cite[thm 18.1.6 p. 569]{katok-hasselblatt} and to the local maximality of basic sets, this last set decomposes into the stable manifolds of elements of $K$, namely
  $$
W^s (K) = \bigcup_{z \in K} W^s (z).  
 $$
 A short proof can be found in \cite[prop 3.10, p. 53]{bowen3} or  \cite[thm 6.26, p. 131]{wen} in the case of diffeomorphisms. The scheme of the proof remains the same in the case of flows. For every $\varepsilon >0$, we can define its local version by setting
\begin{equation}\label{eq: local decomposition of the stable manifold of a basic set}
 W_{\varepsilon}^s (K) = \bigcup_{x \in K} W_{\varepsilon}^s (x).
\end{equation}

 Now, let us present a lemma which was originally given by Smale.

\begin{lemma}[{Partition by stable manifolds, \cite[lem 2.3, p.753]{smale} }]\label{lemma: spectral decomposition of the manifold} We have the following decomposition of $M$ in stable manifolds: 
	 \begin{equation} \label{partition of the manifold in stable manifolds}
	 M = \bigsqcup_{i = 1}^N W^s (K_i ).
	  \end{equation}
\end{lemma}

This lemma dictates the behavior of the trajectories outside the nonwandering set. Precisely, if we take an element $x \in M$, then there exists a unique couple $(i,j) \in [\![1,N]\!]^2$ such that $x \in W^u (K_i) \cap W^s (K_j)$ and the decomposition (\ref{eq: local decomposition of the stable manifold of a basic set}) provides elements $x_- \in K_i$ and $x_+ \in K_j$ such that $x \in W^u (x_-) \cap W^s (x_+)$. The point $x_+$ is unique modulo the equivalence relation on $K_j$ given by:
$$z_1 \sim_j z_2 \Longleftrightarrow z_1, z_2 \in K_j \mbox{ and } z_1 \in W^s (z_2) .$$
A similar remark holds for $x_-$. 
 
 \subsection{Lifting the dynamic on the cotangent}  Since we will use the analysis through escape functions developped in \cite[p. 329]{FS}, we begin by recalling how to lift the flow as an Hamiltonian flow acting on the unitary cotangent bundle $S^* M = \{(x,\xi) \in T^* M, \: |\xi|_g = 1 \}$. The cotangent bundle $T^* M$ can be endowed with a symplectic structure by considering the symplectic form $\omega = d\alpha$, where $d$ is the exterior derivative and $\alpha$ denotes the usual Liouville form on $M$. In a local trivialization chart $q_1, \cdots, q_n, p_1, \cdots, p_n$ of $TM$, the Liouville one form simply writes $\alpha = \sum_{i = 1}^n p_i d q_i$ and the symplectic form becomes $\omega = \sum_{i = 1}^n d p_i \wedge d q_i$. Let us introduce the following maps:
 \begin{equation}\label{eq : definition pi et kappa}
 \begin{split}
  \pi: T^* M \rightarrow M  &, \quad  (x,\xi) \mapsto x, \\
 \kappa: T^* M \setminus 0_M \rightarrow S^* M &, \quad  (x,\xi) \mapsto \left( x, \frac{\xi}{|\xi|_g}\right).
 \end{split}
 \end{equation}
 
 Also, we consider the Hamiltonian $H (x,\xi) = \xi(V(x)) \in \Ccal^{\infty} (T^* M; \R)$ and we denote by $X_H$ the corresponding Hamiltonian vector field, given by the formula $-dH = \iota_{X_H} \omega$. We also denote by $\Phi^t$ the corresponding flow on $T^*M$ which has an explicit formulation:
\begin{align*}
 \Phi^t(x,\xi) = (\varphi^t(x), \left( D\varphi^{t} (x)^{-1} \right)^{\top}(\xi)), \quad \forall (x,\xi) \in T^* M.
\end{align*}
Note that the Hamiltonian flow $\Phi^t$ extends the flow $\varphi^t$ to the cotangent in the sense that $\pi \circ \Phi^t = \varphi^t \circ \pi $ and that it is linear on each fiber. Moreover, the Hamiltonian flow sends $T^* M \setminus 0_M$ into $T^* M \setminus 0_M$ because the matrix $D\varphi^{t} (x)$ is invertible. Therefore, it induces a flow $\widetilde{\Phi}^t$ on the unitary cotangent bundle:
$$\widetilde{\Phi}^t (x,\xi) = ( \kappa \circ \Phi^t) (x,\xi) = \left(\varphi^t(x), \frac{(D\varphi^{t} (x)^{-1} )^{\top}(\xi)}{|(D\varphi^{t} (x)^{-1} )^{\top}(\xi)|_g} \right) , \quad  \forall (x,\xi) \in S^* M.$$
This flow on $S^* M$ is generated by a smooth vector field that will be denoted by $\widetilde{X}_H \in \Ccal^{\infty} (S^* M; T S^* M)$. To summarize, we have the following commutative diagrams:
\begin{center}
\begin{minipage}{5cm}
\begin{tikzpicture}
  \matrix (m) [matrix of math nodes,row sep=3em,column sep=4em,minimum width=2em]
  {
     T^* M & T^* M \\
     M & M \\};
  \path[-stealth]
    (m-1-1) edge node [left] {$ \pi $} (m-2-1)
    (m-1-1) edge node [below] {$ \Phi^t $} (m-1-2)
    (m-2-1) edge node [below] {$\varphi^t$} (m-2-2)
    (m-1-2) edge node [right] {$ \pi $} (m-2-2);
\end{tikzpicture}
\end{minipage} 
\begin{minipage}{5cm}
\begin{tikzpicture}
  \matrix (m) [matrix of math nodes,row sep=3em,column sep=4em,minimum width=2em]
  {
     T^* M \setminus 0_M & T^* M \setminus 0_M \\
     S^* M & S^* M \\};
  \path[-stealth]
    (m-1-1) edge node [left] {$ \kappa $} (m-2-1)
    (m-1-1) edge node [below] {$ \Phi^t $} (m-1-2)
    (m-2-1) edge node [below] {$\widetilde{\Phi}^t$} (m-2-2)
    (m-1-2) edge node [right] {$ \kappa $} (m-2-2);
\end{tikzpicture}
\end{minipage}
\end{center}

Since our analysis will take place in $T^* M$, we also define the dual distributions associated with the neutral $E_o$, stable $E_s$ and unstable $E_u$ distributions\footnote{On a basic set $K$, we also use the notations $E_{so}^*$ for $E_s^* + E_o^*$ and $E_{uo}^*$ for $E_u^* + E_o^*$.} which appear in the definition of hyperbolicity (see appendix \ref{Appendix: hyperbolic set}) at any point $z \in \Omega$ by 
\begin{align*}
  (E_o^*(z))(E_u(z) &\oplus E_s(z)) = 0, \\
  (E_u^*(z))(E_u(z) \oplus E_o(z)) = 0, \quad & \quad (E_s^*(z))(E_s(z) \oplus E_o(z)) = 0.
 \end{align*}
This gives us a hyperbolic splitting of the cotangent bundle:
\begin{align*}
 T_z^* M = E_s^* (z) \oplus E_u^* (z) \oplus E_o^*(z).
\end{align*}

 \subsection{Extension of the invariant distributions outside the nonwandering set}
 
 Thanks to the partition's lemma \ref{partition of the manifold in stable manifolds} by stable manifolds, we can extend the previous definitions outside the nonwandering set.
 
\begin{definition}\label{def distributions duales stables et instables sur Omega}
 For every $x \in W^s (x_+) \cap W^u (x_{-})$ with $x_- \in K_i$ and $x_+ \in K_j$, we define the following: 
	\begin{align*}
		E_s^*(x) (T_x W^{so} (x_+)) = 0,  &\qquad E_{so}^* (x)(T_x W^{s} (x_+)) = 0,\\
		E_u^*(x) (T_x W^{uo} (x_-)) = 0,  &\qquad E_{uo}^* (x)(T_x W^{u} (x_-)) = 0 \: .
	\end{align*}
	Moreover, the definition does not depend on the choice of $x_+$ and $x_-$ in $W^s (x_+) \cap K_j$ and $W^u (x_-) \cap K_i$ respectively.
\end{definition}

\begin{remark}\begin{itemize}[align=left, leftmargin=*, noitemsep]
      \item If $x_+$ is a fixed point, then $W^{so} (x_+) = W^{s} (x_+)$ and consequently $E_s^*(x) = E_{so}^*(x) = E_{so}^*(x) \cap \{\xi (V(x)) = 0 \}$.
      \item The distributions $E_s^*, E_u^*, E_{so}^*$ and $E_{uo}^*$ are defined on the whole manifold and are $\Phi^t$-invariant.
      \item Recall that fixed points are isolated (remark \ref{rk: fixed point isolated}). So if $x_+$ is not a fixed point then we get that $\dim E_{so}(x_+) = \dim E_{s}(x_+) + 1$ and $\dim W^{so} (x_+) = \dim W^{s} (x_+) + 1$. Therefore, we obtain $E_{s}^* (x) = E_{so}^* (x) \cap \{\xi (V(x)) = 0 \}$ and $\dim E_{so}^* (x) = \dim E_{s}^* (x) + 1$.
      \end{itemize} 
\end{remark}

  \subsection{Transversality assumption}

 \begin{figure}[!t]
\begin{subfigure}[b]{0.45\textwidth}
  \centering
  \definecolor{uququq}{rgb}{0.25,0.25,0.25}
\begin{tikzpicture}[line cap=round,line join=round,>=triangle 45,x=1.5cm,y=1.5cm]
\clip(-2.5,-2) rectangle (2.6,1.9);
\draw (-2,0)-- (2,0);
\draw (0,1.5)-- (0,-1.5);
\draw [->] (0,1.5) -- (-1,1.5);
\draw [->] (0,1.5) -- (1,1.5);
\draw [->] (0,-1.5) -- (1,-1.5);
\draw [->] (0,0) -- (-1,0);
\draw [->] (0,0) -- (1,0);
\draw [->] (0,-1.5) -- (0,-0.75);
\draw [->] (0,1.5) -- (0,0.75);
\draw [->] (0,-1.5) -- (-1,-1.5);
\draw (-2,1.5)-- (2,1.5);
\draw (2,1.5)-- (2,-1.5);
\draw (2,-1.5)-- (-2,-1.5);
\draw (-2,-1.5)-- (-2,1.5);
\draw (-0.02,1.84) node[anchor=north west] {$K_1$};
\draw (-0.1,-1.49) node[anchor=north west] {$K_1$};
\draw (2.12,1.28) node[anchor=north west] {\large $\mathbb{T}^2$};
\draw (-2.08,1.84) node[anchor=north west] {$K_4$};
\draw (1.95,1.8) node[anchor=north west] {$K_4$};
\draw (1.93,-1.49) node[anchor=north west] {$K_4$};
\draw (-2.15,-1.49) node[anchor=north west] {$K_4$};
\draw (-0.07,0.33) node[anchor=north west] {$K_2$};
\draw (-2.46,0.24) node[anchor=north west] {$K_3$};
\draw (1.96,0.24) node[anchor=north west] {$K_3$};
\draw [->] (-2,0) -- (-2,0.79);
\draw [->] (-2,0) -- (-2,-0.82);
\draw [->] (2,0) -- (2,0.79);
\draw [->] (2,0) -- (2,-0.83);
\draw [shift={(-1,1.14)}] plot[domain=3.18:6.21,variable=\t]({1*0.92*cos(\t r)+0*0.92*sin(\t r)},{0*0.92*cos(\t r)+1*0.92*sin(\t r)});
\draw [shift={(-0.99,1.45)}] plot[domain=3.65:5.76,variable=\t]({1*0.81*cos(\t r)+0*0.81*sin(\t r)},{0*0.81*cos(\t r)+1*0.81*sin(\t r)});
\draw [shift={(-0.97,2.76)}] plot[domain=4.32:5.06,variable=\t]({1*1.67*cos(\t r)+0*1.67*sin(\t r)},{0*1.67*cos(\t r)+1*1.67*sin(\t r)});
\draw [->] (-0.95,1.09) -- (-1.06,1.09);
\draw [->] (-0.99,0.64) -- (-1.1,0.64);
\draw [->] (-0.97,0.22) -- (-1.08,0.23);
\begin{scriptsize}
\fill [color=uququq] (-2,1.5) circle (1.5pt);
\fill [color=uququq] (2,1.5) circle (1.5pt);
\fill [color=uququq] (2,-1.5) circle (1.5pt);
\fill [color=uququq] (-2,-1.5) circle (1.5pt);
\fill [color=uququq] (-2,0) circle (1.5pt);
\fill [color=uququq] (2,0) circle (1.5pt);
\fill [color=uququq] (0,1.5) circle (1.5pt);
\fill [color=uququq] (0,0) circle (1.5pt);
\fill [color=uququq] (0,-1.5) circle (1.5pt);
\end{scriptsize}
\end{tikzpicture}
\caption{Axiom A flow which does not satisfy the transversality assumption: existence of a saddle-connection.}
 \label{fig: Axiom A flow on the torus - without transversality assumption}
\end{subfigure} \hfill
\begin{subfigure}[b]{0.45\textwidth}
 \centering
 \definecolor{uququq}{rgb}{0.25,0.25,0.25}
\begin{tikzpicture}[line cap=round,line join=round,>=triangle 45,x=1.5cm,y=1.5cm]
\clip(-2.4,-2) rectangle (2.7,1.8);
\draw (-2,0)-- (2,0);
\draw (0,1.5)-- (0,-1.5);
\draw [->] (0,1.5) -- (-1,1.5);
\draw [->] (0,1.5) -- (1,1.5);
\draw [->] (2,1.5) -- (2,0.75);
\draw [->] (2,-1.5) -- (2,-0.75);
\draw [->] (0,-1.5) -- (1,-1.5);
\draw [->] (-2,1.5) -- (-2,0.75);
\draw [->] (-2,-1.5) -- (-2,-0.75);
\draw [->] (0,0) -- (-1,0);
\draw [->] (0,0) -- (1,0);
\draw [->] (0,-1.5) -- (0,-0.75);
\draw [->] (0,1.5) -- (0,0.75);
\draw [->] (0,-1.5) -- (-1,-1.5);
\draw (-2,1.5)-- (2,1.5);
\draw (2,1.5)-- (2,-1.5);
\draw (2,-1.5)-- (-2,-1.5);
\draw (-2,-1.5)-- (-2,1.5);
\draw [shift={(-0.26,-0.32)}] plot[domain=1.66:2.79,variable=\t]({1*1.67*cos(\t r)+0*1.67*sin(\t r)},{0*1.67*cos(\t r)+1*1.67*sin(\t r)});
\draw [shift={(-1.87,2.03)}] plot[domain=4.83:5.8,variable=\t]({1*1.87*cos(\t r)+0*1.87*sin(\t r)},{0*1.87*cos(\t r)+1*1.87*sin(\t r)});
\draw [shift={(1.98,2.14)}] plot[domain=3.6:4.55,variable=\t]({1*2.03*cos(\t r)+0*2.03*sin(\t r)},{0*2.03*cos(\t r)+1*2.03*sin(\t r)});
\draw [shift={(-0.06,-0.76)}] plot[domain=0.53:1.37,variable=\t]({1*2.15*cos(\t r)+0*2.15*sin(\t r)},{0*2.15*cos(\t r)+1*2.15*sin(\t r)});
\draw [shift={(-1.8,-1.85)}] plot[domain=0.38:1.51,variable=\t]({1*1.72*cos(\t r)+0*1.72*sin(\t r)},{0*1.72*cos(\t r)+1*1.72*sin(\t r)});
\draw [shift={(-0.14,0.5)}] plot[domain=3.59:4.6,variable=\t]({1*1.91*cos(\t r)+0*1.91*sin(\t r)},{0*1.91*cos(\t r)+1*1.91*sin(\t r)});
\draw [shift={(1.78,-1.93)}] plot[domain=1.68:2.72,variable=\t]({1*1.78*cos(\t r)+0*1.78*sin(\t r)},{0*1.78*cos(\t r)+1*1.78*sin(\t r)});
\draw [shift={(0.28,0.24)}] plot[domain=4.77:5.87,variable=\t]({1*1.61*cos(\t r)+0*1.61*sin(\t r)},{0*1.61*cos(\t r)+1*1.61*sin(\t r)});
\draw [->] (-1.26,1.01) -- (-1.38,0.91);
\draw [->] (-0.76,0.53) -- (-0.83,0.48);
\draw [->] (0.66,0.59) -- (0.73,0.53);
\draw [->] (1.14,1.02) -- (1.24,0.95);
\draw [->] (0.69,-0.52) -- (0.8,-0.44);
\draw [->] (1.22,-1.06) -- (1.29,-1.01);
\draw [->] (-1.21,-1.09) -- (-1.29,-1.03);
\draw [->] (-0.72,-0.51) -- (-0.81,-0.44);
\draw (-0.02,1.84) node[anchor=north west] {$K_1$};
\draw (-0.1,-1.49) node[anchor=north west] {$K_1$};
\draw (2.12,1.28) node[anchor=north west] {\large $\mathbb{T}^2$};
\draw (-2.08,1.84) node[anchor=north west] {$K_2$};
\draw (1.95,1.8) node[anchor=north west] {$K_2$};
\draw (1.93,-1.49) node[anchor=north west] {$K_2$};
\draw (-2.15,-1.49) node[anchor=north west] {$K_2$};
\draw (-0.07,0.33) node[anchor=north west] {$K_3$};
\draw (-2.46,0.24) node[anchor=north west] {$K_4$};
\draw (1.96,0.24) node[anchor=north west] {$K_4$};
\begin{scriptsize}
\fill [color=uququq] (-2,1.5) circle (1.5pt);
\fill [color=uququq] (2,1.5) circle (1.5pt);
\fill [color=uququq] (2,-1.5) circle (1.5pt);
\fill [color=uququq] (-2,-1.5) circle (1.5pt);
\fill [color=uququq] (-2,0) circle (1.5pt);
\fill [color=uququq] (2,0) circle (1.5pt);
\fill [color=uququq] (0,1.5) circle (1.5pt);
\fill [color=uququq] (0,0) circle (1.5pt);
\fill [color=uququq] (0,-1.5) circle (1.5pt);
\end{scriptsize}
\end{tikzpicture}
\caption{Axiom A flow which satisfies the transversality assumption.\\} \label{fig: Axiom A flow on the torus - with transversality assumption}
\end{subfigure}
\caption{Some Axiom A flows on the 2-torus.}
\end{figure}
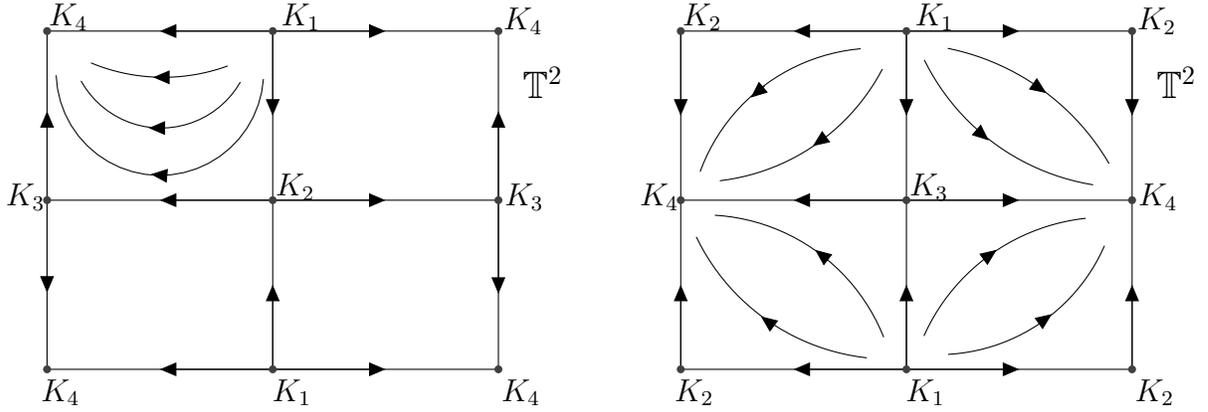
   
  For any $(x_-,x_+) \in K_i \times K_j$ and every $x \in  W^u (x_-) \cap W^s (x_+)$, we make the following \textit{strong tranversality assumption}:
   \begin{equation}\label{Transversality assumption - first}
	\boxed{ T_x W^{uo} (x_-) + T_x W^{so} (x_+)  = T_x M. }
  \end{equation}
  which is equivalent to
  \begin{equation}\label{Transversality assumption}
	\boxed{ T_x W^{u} (x_-) + T_x W^{so} (x_+)  = T_x W^{uo} (x_-) + T_x W^{s} (x_+) = T_x M. }
  \end{equation}
  since both spaces $T_x W^{uo} (x_-) $ and $T_x W^{so} (x_+)$ contain the flow direction $\R. V(x)$. Moreover, these transversality assumptions do not depend on the choice of $x_- \in K_i$ and $x_+ \in K_j$ such that $x \in W^u (x_-) \cap W^s (x_+)$.
  \begin{remark}If $x_-$ or $x_+$ is a fixed point, which is in particular the case for Morse-Smale gradient flows \cite{DR0} where both are fixed points, then the transversality assumption reads 
  $$T_x W^{u} (x_-) + T_x W^{s} (x_+)  = T_x M .$$
  \end{remark}
  
  From (\ref{Transversality assumption}) and directly from the definitions, we can deduce the following \textit{disjointness} properties:
  \begin{align*}
   E_s^* \cap E_{uo}^* = \emptyset \mbox{ and } E_{so}^* \cap E_u^* = \emptyset .
  \end{align*}
  
  \textbf{From now on}, the vector field $V \in \Gamma (TM)$ will be considered to be Axiom A and to satisfy the transversality assumption (\ref{Transversality assumption}). Let us extend the \textit{neutral} distribution outside the nonwandering set by fixing, for every $x \in  W^u (x_-) \cap W^s (x_+)$,
  $$
	E_o^* (x):= E_{uo}^* (x_-) \cap E_{so}^* (x_+) = \left\{ \xi \in T^* M , \: \xi \left(  T_x W^{u} (x_-) + T_x W^{s} (x_+)\right) = 0 \right\}.
  $$

 \subsection{Order relation}

 When Smale \cite{smale} defined Axiom A flows, he exhibited a partial order relation between basics sets of an Axiom A flow which satifies the strong tranversality assumption (\ref{Transversality assumption - first}).  Precisely, for two basic sets $K_i$ and $K_j$, he defined the relation $\leq$, illustrated in figures \ref{fig: order relation between two fixed points} and \ref{fig: order relation between two basic sets}, by
 \begin{equation}\label{order relation for basic sets}
    \begin{split}
	K_i \leq K_j &\Longleftrightarrow W^u (K_i) \cap W^s (K_j) \neq \emptyset\\
                &\Longleftrightarrow \exists x \in M, \: \exists (x_-, x_+) \in K_i \times K_j, \: \: x \in W^u (x_-) \cap W^s (x_+).
    \end{split}
 \end{equation}

\begin{figure}
   \begin{minipage}[c]{.46\linewidth}
\def\svgscale{0.45}
 \executeiffilenewer{deux_ens_basiques.svg}{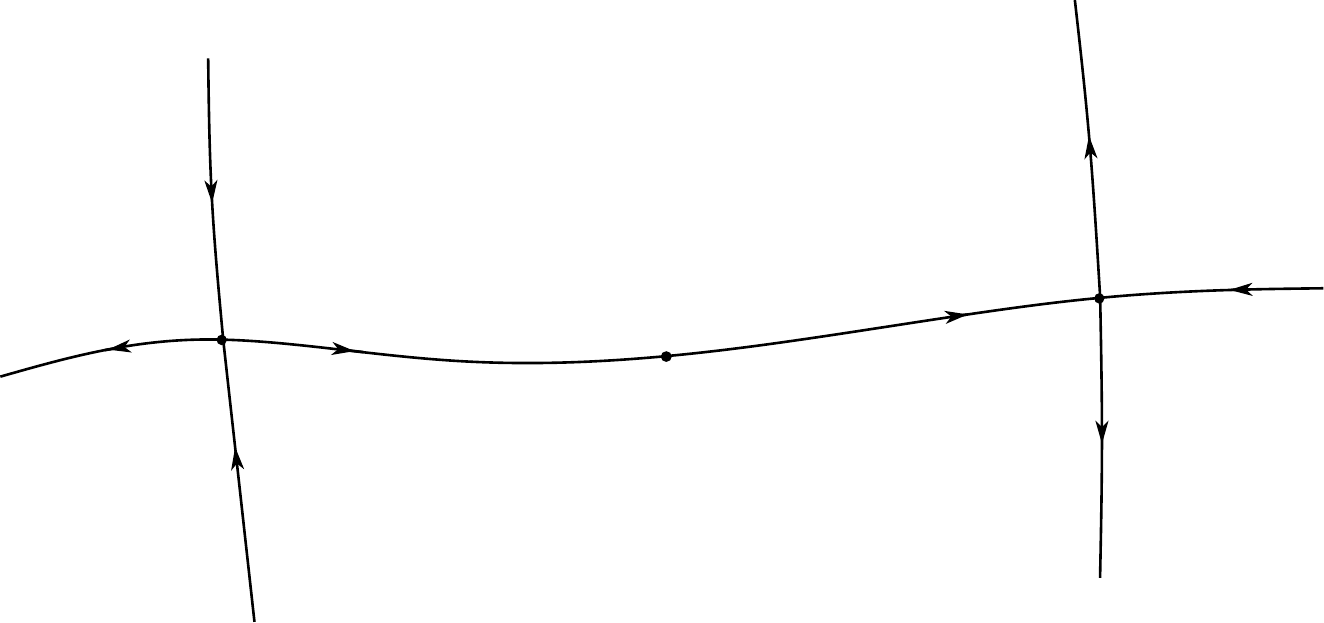}%
 {inkscape -z -D --file=deux_ens_basiques.svg %
 --export-pdf=deux_ens_basiques.pdf --export-latex}%
\begingroup%
  \makeatletter%
  \providecommand\color[2][]{%
    \errmessage{(Inkscape) Color is used for the text in Inkscape, but the package 'color.sty' is not loaded}%
    \renewcommand\color[2][]{}%
  }%
  \providecommand\transparent[1]{%
    \errmessage{(Inkscape) Transparency is used (non-zero) for the text in Inkscape, but the package 'transparent.sty' is not loaded}%
    \renewcommand\transparent[1]{}%
  }%
  \providecommand\rotatebox[2]{#2}%
  \newcommand*\fsize{\dimexpr\f@size pt\relax}%
  \newcommand*\lineheight[1]{\fontsize{\fsize}{#1\fsize}\selectfont}%
  \ifx\svgwidth\undefined%
    \setlength{\unitlength}{381.1395552bp}%
    \ifx\svgscale\undefined%
      \relax%
    \else%
      \setlength{\unitlength}{\unitlength * \real{\svgscale}}%
    \fi%
  \else%
    \setlength{\unitlength}{\svgwidth}%
  \fi%
  \global\let\svgwidth\undefined%
  \global\let\svgscale\undefined%
  \makeatother%
  \begin{picture}(1,0.47001723)%
    \lineheight{1}%
    \setlength\tabcolsep{0pt}%
    \put(0,0){\includegraphics[width=\unitlength,page=1]{deux_ens_basiques.pdf}}%
    \put(0.49549923,0.24457953){\makebox(0,0)[lt]{\lineheight{1.25}\smash{\begin{tabular}[t]{l}$x$\end{tabular}}}}%
    \put(0.1803473,0.2422438){\makebox(0,0)[lt]{\lineheight{1.25}\smash{\begin{tabular}[t]{l}$K_i$\end{tabular}}}}%
    \put(0.84249391,0.27532338){\makebox(0,0)[lt]{\lineheight{1.25}\smash{\begin{tabular}[t]{l}$K_j$\end{tabular}}}}%
  \end{picture}%
\endgroup%

  \caption{Order relation between two hyperbolic fixed points.}
  \label{fig: order relation between two fixed points} 
  \end{minipage}
   \begin{minipage}[c]{.46\linewidth}
\def\svgscale{0.5}
 \executeiffilenewer{two_basic_sets.svg}{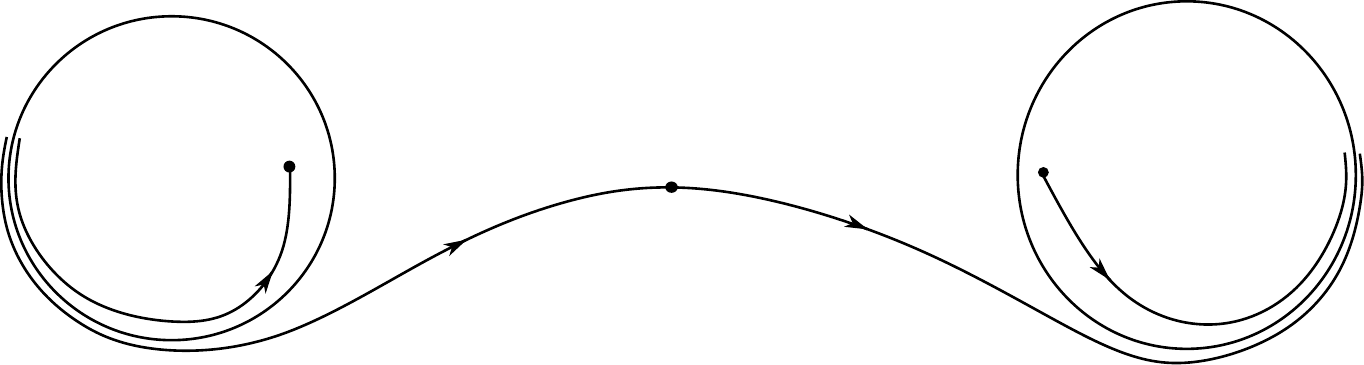}%
 {inkscape -z -D --file=two_basic_sets.svg %
 --export-pdf=two_basic_sets.pdf --export-latex}%
\begingroup%
  \makeatletter%
  \providecommand\color[2][]{%
    \errmessage{(Inkscape) Color is used for the text in Inkscape, but the package 'color.sty' is not loaded}%
    \renewcommand\color[2][]{}%
  }%
  \providecommand\transparent[1]{%
    \errmessage{(Inkscape) Transparency is used (non-zero) for the text in Inkscape, but the package 'transparent.sty' is not loaded}%
    \renewcommand\transparent[1]{}%
  }%
  \providecommand\rotatebox[2]{#2}%
  \newcommand*\fsize{\dimexpr\f@size pt\relax}%
  \newcommand*\lineheight[1]{\fontsize{\fsize}{#1\fsize}\selectfont}%
  \ifx\svgwidth\undefined%
    \setlength{\unitlength}{392.72001011bp}%
    \ifx\svgscale\undefined%
      \relax%
    \else%
      \setlength{\unitlength}{\unitlength * \real{\svgscale}}%
    \fi%
  \else%
    \setlength{\unitlength}{\svgwidth}%
  \fi%
  \global\let\svgwidth\undefined%
  \global\let\svgscale\undefined%
  \makeatother%
  \begin{picture}(1,0.26719698)%
    \lineheight{1}%
    \setlength\tabcolsep{0pt}%
    \put(0,0){\includegraphics[width=\unitlength,page=1]{two_basic_sets.pdf}}%
    \put(0.49406859,0.14984527){\makebox(0,0)[lt]{\lineheight{1.25}\smash{\begin{tabular}[t]{l}$x$\end{tabular}}}}%
    \put(0.20739237,0.16157944){\makebox(0,0)[lt]{\lineheight{1.25}\smash{\begin{tabular}[t]{l}$x_-$\end{tabular}}}}%
    \put(0.76512248,0.15415055){\makebox(0,0)[lt]{\lineheight{1.25}\smash{\begin{tabular}[t]{l}$x_+$\end{tabular}}}}%
    \put(0.86465905,0.14065571){\makebox(0,0)[lt]{\lineheight{1.25}\smash{\begin{tabular}[t]{l}$K_j$\end{tabular}}}}%
    \put(0.08853707,0.14585718){\makebox(0,0)[lt]{\lineheight{1.25}\smash{\begin{tabular}[t]{l}$K_i$\end{tabular}}}}%
  \end{picture}%
\endgroup%

  \caption{Order relation between two basic sets.}
  \label{fig: order relation between two basic sets} 
   \end{minipage}
\end{figure}
 
\begin{thm}[{Smale, \cite[prop. 8.5, p. 784]{smale}}]\label{thm Smale}
 If $\varphi^t$ satisfies the transversality assumption (\ref{Transversality assumption}), then the relation $\leq$ defines a partial order relation. Moreover, for every basic set $K$, we have
 $$ W^u (K) \cap W^s (K) = K .$$
\end{thm}

We refer to \cite[\S 3-4 p. 152]{pugh-shub} and \cite[app. B, p. 50]{DRI} for more details about the proof. As we can see in figure \ref{fig: Axiom A flow on the torus - without transversality assumption}, relation (\ref{order relation for basic sets}) may not be an order relation if the transversality assumption does not hold.

\begin{prop}[{Smale, \cite[p. 783]{smale}}]An equivalent definition is given by
	$$
		K_i \leq K_j \Longleftrightarrow W^s (K_i) \subseteq \overline{W^s (K_j) } \Longleftrightarrow W^u (K_j) \subseteq \overline{W^u (K_i) } .
	$$
	We also have
	$$
		\overline{W^u (K_i) } = \bigcup_{j, K_j \geq K_i} W^u (K_j) \: \mbox{ and } \: \overline{W^s (K_j) } = \bigcup_{i, K_i \leq K_i} W^s (K_i).
	$$
\end{prop}

\subsubsection{Graph structure}\label{subsubsection graph structure}

From the partial ordering on basic sets, one can define a graph structure as follows. The vertices $V$ of the graph $G$ are given by elements of $[\![1,N]\!]$ while the edges $E$ form a subset of $[\![1,N]\!] \times [\![1,N]\!]$ which satifies
$$i \neq j \mbox{ and } K_i \leq K_j \Longleftrightarrow (i,j) \in E.$$
To ensure that the graph does not contain too many edges, we add the following irreducibility assumption: for every integer $m \geq 3$ and for every $i_1,\cdots,i_m \in [\![1,N]\!]$,
$$
	(i_1,i_2), \cdots, (i_{m-1},i_m) \in E \Longrightarrow (i_1,i_m) \notin E.
$$
As a direct application of Theorem \ref{thm Smale} the oriented graph $G$ has no cycle.

\subsubsection{Total order relation}\label{subsubsection total order relation}

In what follows, we will frequently use mathematical induction. To do so, we define a total order relation from Smale's order relation $\leq$ on the basic sets as an order relation on $[\![1,N]\!]$  compatible with the partial order relation $\leq$ in the sense that
\begin{equation}\label{eq: total order relation, indices compatible}
	K_i \leq K_j \Longrightarrow i \leq j.
\end{equation}
\textbf{From now on}, we fix a total order relation. Note that $\cup_{j\geq i }W^u (K_j)$ is compact for every $1 \leq i \leq N$ but that $\overline{W^u (K_i)}$ and $\cup_{j\geq i }W^u (K_j)$ are not equal in general, see for example figure \ref{fig: graph structures on the 2 torus}.

\begin{figure}[t]
  \begin{minipage}[c]{.46\linewidth}
\centering
\def\svgscale{0.6}
 \executeiffilenewer{graphe_structure_tore.svg}{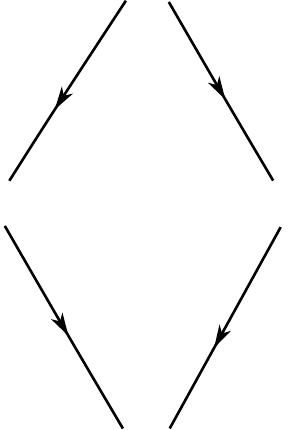}%
 {inkscape -z -D --file=graphe_structure_tore.svg %
 --export-pdf=graphe_structure_tore.pdf --export-latex}%
\begingroup%
  \makeatletter%
  \providecommand\color[2][]{%
    \errmessage{(Inkscape) Color is used for the text in Inkscape, but the package 'color.sty' is not loaded}%
    \renewcommand\color[2][]{}%
  }%
  \providecommand\transparent[1]{%
    \errmessage{(Inkscape) Transparency is used (non-zero) for the text in Inkscape, but the package 'transparent.sty' is not loaded}%
    \renewcommand\transparent[1]{}%
  }%
  \providecommand\rotatebox[2]{#2}%
  \newcommand*\fsize{\dimexpr\f@size pt\relax}%
  \newcommand*\lineheight[1]{\fontsize{\fsize}{#1\fsize}\selectfont}%
  \ifx\svgwidth\undefined%
    \setlength{\unitlength}{85.8115381bp}%
    \ifx\svgscale\undefined%
      \relax%
    \else%
      \setlength{\unitlength}{\unitlength * \real{\svgscale}}%
    \fi%
  \else%
    \setlength{\unitlength}{\svgwidth}%
  \fi%
  \global\let\svgwidth\undefined%
  \global\let\svgscale\undefined%
  \makeatother%
  \begin{picture}(1,1.44030322)%
    \lineheight{1}%
    \setlength\tabcolsep{0pt}%
    \put(0,0){\includegraphics[width=\unitlength,page=1]{graphe_structure_tore.pdf}}%
    \put(0.41246107,1.47885041){\makebox(0,0)[lt]{\lineheight{1.25}\smash{\begin{tabular}[t]{l}$K_1$\end{tabular}}}}%
    \put(-0.13611512,0.70184378){\makebox(0,0)[lt]{\lineheight{1.25}\smash{\begin{tabular}[t]{l}$K_2$\end{tabular}}}}%
    \put(0.91008139,0.7145733){\makebox(0,0)[lt]{\lineheight{1.25}\smash{\begin{tabular}[t]{l}$K_3$\end{tabular}}}}%
    \put(0.38260597,-0.12391249){\makebox(0,0)[lt]{\lineheight{1.25}\smash{\begin{tabular}[t]{l}$K_4$\end{tabular}}}}%
  \end{picture}%
\endgroup%

  \end{minipage}
  \begin{minipage}[c]{.46\linewidth}
   \centering
\def\svgscale{0.6}
 \executeiffilenewer{graphe_structure_tore2.svg}{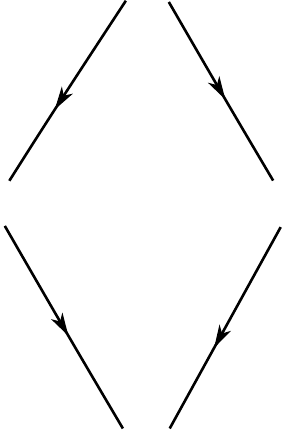}%
 {inkscape -z -D --file=graphe_structure_tore2.svg %
 --export-pdf=graphe_structure_tore2.pdf --export-latex}%
\begingroup%
  \makeatletter%
  \providecommand\color[2][]{%
    \errmessage{(Inkscape) Color is used for the text in Inkscape, but the package 'color.sty' is not loaded}%
    \renewcommand\color[2][]{}%
  }%
  \providecommand\transparent[1]{%
    \errmessage{(Inkscape) Transparency is used (non-zero) for the text in Inkscape, but the package 'transparent.sty' is not loaded}%
    \renewcommand\transparent[1]{}%
  }%
  \providecommand\rotatebox[2]{#2}%
  \newcommand*\fsize{\dimexpr\f@size pt\relax}%
  \newcommand*\lineheight[1]{\fontsize{\fsize}{#1\fsize}\selectfont}%
  \ifx\svgwidth\undefined%
    \setlength{\unitlength}{85.8115381bp}%
    \ifx\svgscale\undefined%
      \relax%
    \else%
      \setlength{\unitlength}{\unitlength * \real{\svgscale}}%
    \fi%
  \else%
    \setlength{\unitlength}{\svgwidth}%
  \fi%
  \global\let\svgwidth\undefined%
  \global\let\svgscale\undefined%
  \makeatother%
  \begin{picture}(1,1.44030322)%
    \lineheight{1}%
    \setlength\tabcolsep{0pt}%
    \put(0,0){\includegraphics[width=\unitlength,page=1]{graphe_structure_tore2.pdf}}%
    \put(0.41246107,1.47885041){\makebox(0,0)[lt]{\lineheight{1.25}\smash{\begin{tabular}[t]{l}$K_1$\end{tabular}}}}%
    \put(-0.13611512,0.70184378){\makebox(0,0)[lt]{\lineheight{1.25}\smash{\begin{tabular}[t]{l}$K_3$\end{tabular}}}}%
    \put(0.91008139,0.7145733){\makebox(0,0)[lt]{\lineheight{1.25}\smash{\begin{tabular}[t]{l}$K_4$\end{tabular}}}}%
    \put(0.38260597,-0.12391249){\makebox(0,0)[lt]{\lineheight{1.25}\smash{\begin{tabular}[t]{l}$K_2$\end{tabular}}}}%
  \end{picture}%
\endgroup%

  \end{minipage}
  \caption{Both graph correspond to the Axiom A flow on the 2-torus of fig. \ref{fig: Axiom A flow on the torus - with transversality assumption}. The left one has its indices compatible with the graph structure and not the right one.}
  \label{fig: graph structures on the 2 torus}
\end{figure}

 \subsection{Filtrations and unrevisited neighborhoods}

 An important concept in all our analysis is the concept of filtration. Eventhough the term of filtration usually refers to an increasing sequence of subcomplexes of a simplicial complex, we give here an open version deeply related to Morse homology where the subcomplexes are open sets, i.e. submanifolds of dimension $n = \dim M$ with boundary. To see the analogy, let us consider $f: M \rightarrow \R$ a Morse function which has $N$ critical points $x_1, \cdots, x_N$ that all satisfy $f(x_i) = i$ to simplify. Then, a filtration is given by the family of open sets $(f^{-1} (]-\infty,i +\frac{1}{2}[))_{0 \leq i \leq N}$. For a general Axiom A flow, we give a definition of a filtration which appeared in the works of Smale \cite{smale} and Robbin \cite[lem. 7.9, p. 471]{robbin}.

   \begin{definition}[Filtration] Let $\varphi^t$ be an Axiom A flow which satisfies the transversality assumption. Let us consider a total order relation on the basic sets in the sense of (\ref{eq: total order relation, indices compatible}). A sequence of open sets $(\Ocal_i^-)_{0 \leq i \leq N}$ is said to be a \textbf{filtration} for $\varphi^{-1}$ if the following conditions hold: 
\begin{enumerate}[label=(\roman*)]
 \item The sequence is increasing: $$\emptyset = \Ocal_0^- \subseteq \Ocal_{1}^- \subseteq \cdots \subseteq \Ocal_N^- = M$$
 \item For every $1 \leq i \leq N$, the open sets are $\varphi^{-1}$-stables: $\varphi^{-1} \left( \Ocal_{i}^- \right) \subseteq \Ocal_{i}^- $.
 \item For every $1 \leq i \leq N$, we have $K_i \subseteq \Ocal_i^- \setminus \overline{\Ocal_{i-1}^-}$. 
\end{enumerate}
\end{definition}

 \begin{figure}[t]
\centering
\def\svgscale{0.6}
 \executeiffilenewer{filtration.svg}{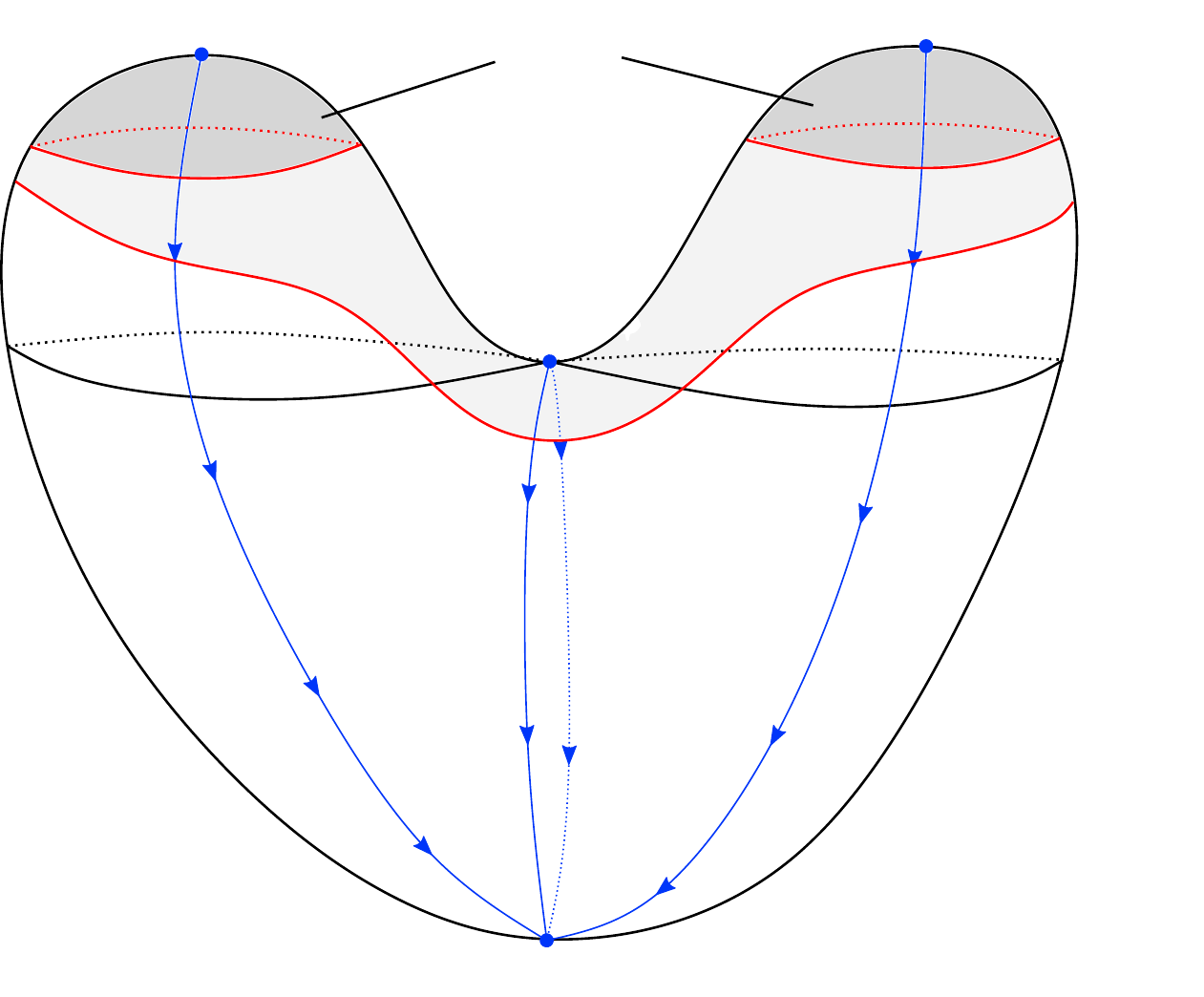}%
 {inkscape -z -D --file=filtration.svg %
 --export-pdf=filtration.pdf --export-latex}%
\begingroup%
  \makeatletter%
  \providecommand\color[2][]{%
    \errmessage{(Inkscape) Color is used for the text in Inkscape, but the package 'color.sty' is not loaded}%
    \renewcommand\color[2][]{}%
  }%
  \providecommand\transparent[1]{%
    \errmessage{(Inkscape) Transparency is used (non-zero) for the text in Inkscape, but the package 'transparent.sty' is not loaded}%
    \renewcommand\transparent[1]{}%
  }%
  \providecommand\rotatebox[2]{#2}%
  \newcommand*\fsize{\dimexpr\f@size pt\relax}%
  \newcommand*\lineheight[1]{\fontsize{\fsize}{#1\fsize}\selectfont}%
  \ifx\svgwidth\undefined%
    \setlength{\unitlength}{362.95523576bp}%
    \ifx\svgscale\undefined%
      \relax%
    \else%
      \setlength{\unitlength}{\unitlength * \real{\svgscale}}%
    \fi%
  \else%
    \setlength{\unitlength}{\svgwidth}%
  \fi%
  \global\let\svgwidth\undefined%
  \global\let\svgscale\undefined%
  \makeatother%
  \begin{picture}(1,0.82774237)%
    \lineheight{1}%
    \setlength\tabcolsep{0pt}%
    \put(0,0){\includegraphics[width=\unitlength,page=1]{filtration.pdf}}%
    \put(0.44995019,0.78196019){\makebox(0,0)[lt]{\lineheight{1.25}\smash{\begin{tabular}[t]{l}$\mathcal{O}_2^-$\end{tabular}}}}%
    \put(0.14260747,0.792147){\makebox(0,0)[lt]{\lineheight{1.25}\smash{\begin{tabular}[t]{l}$K_1$\end{tabular}}}}%
    \put(0.07166308,0.73133286){\makebox(0,0)[lt]{\lineheight{1.25}\smash{\begin{tabular}[t]{l}$\mathcal{O}_1^-$\end{tabular}}}}%
    \put(0.44007365,0.54324634){\makebox(0,0)[lt]{\lineheight{1.25}\smash{\begin{tabular}[t]{l}$K_3$\end{tabular}}}}%
    \put(0.75540003,0.79762737){\makebox(0,0)[lt]{\lineheight{1.25}\smash{\begin{tabular}[t]{l}$K_2$\end{tabular}}}}%
    \put(0.44045577,0.00518044){\makebox(0,0)[lt]{\lineheight{1.25}\smash{\begin{tabular}[t]{l}$K_4$\end{tabular}}}}%
    \put(0.610518,0.64518063){\makebox(0,0)[lt]{\lineheight{1.25}\smash{\begin{tabular}[t]{l}$\mathcal{O}_3^-$\end{tabular}}}}%
  \end{picture}%
\endgroup%

  \caption{Example of a filtration on the sphere $\Sn^2$.}
  \label{fig: filtration on the sphere} 
\end{figure}

\begin{remark}\label{rk : filtration}
\begin{itemize}[align=left, leftmargin=*, noitemsep]
 \item Any filtration for $\varphi^{-1}$ induces a filtration for $\varphi^{1}$ taking the interior of the complementary of each open set, i.e. setting $\Ocal_{i}^{+}:= \Int \left( \left( \Ocal_{N-i}^{-} \right)^c \right)$  for every $i$. Indeed, from $\varphi^{-1} \left( \Ocal_{N-i}^- \right) \subseteq \Ocal_{N-i}^-$, we deduce 
	$$\varphi^1 (\Ocal_{i}^+) = \varphi^1 \left( \Int \left( \left( \Ocal_{N-i}^- \right)^c \right) \right) \subseteq \varphi^1 \left( \left( \Ocal_{N-i}^- \right)^c \right)  \subseteq \left( \Ocal_{N-i}^- \right)^c. $$
	So, $\varphi^1 (\Ocal_{i}^+)$ is an open subset of $\left( \Ocal_{N-i}^- \right)^c$, and by definition of the interior we get $\varphi^1 (\Ocal_{i}^+) \subseteq \Int \left( \left( \Ocal_{N-i}^- \right)^c \right) = \Ocal_{i}^+$.
	Moreover, we still have $\emptyset = \Ocal_0^+ \subseteq \Ocal_{1}^+ \subseteq \cdots \subseteq \Ocal_N^+ = M$ and $K_{i} \subseteq \Ocal_{N-i+1}^+ \setminus \overline{\Ocal_{N-i}^+}$. 
 \item $\Ocal_i^- $ is a neighborhood of $\sqcup_{j\leq i} W^s (K_i)$. Indeed, $\Ocal_i^- $ contains every basic set $K_j$ for $j \leq i$ and if 
 $x \in W^s (K_j)$ then there exists $k \in \N$ such that $\varphi^k (x) \in \Ocal_i^-$. Thus, we deduce $x \in \varphi^{-k} (\Ocal_i^-) \subset \Ocal_i^-$.
 \item For every $1 \leq i \leq N$, let us define the set
 $$ \Vcal_i:= \Ocal_i^- \cap \Ocal_{N-i+1}^+. $$
 One can check that $\Vcal_i$ is a neighborhood of $K_i$ which satisfies $\overline{\Vcal_i}\cap \Omega = K_i$ and the following property: for all $m \in \N$ and for all $x \in \Vcal_i$, if we have $x \in \Vcal_i$ and $\varphi^m (x) \in \Vcal_i$ then we must have $\varphi^k (x) \in \Vcal_i$ for all $0 \leq k \leq m$. In the example presented in figure \ref{fig: filtration on the sphere}, the basic set $K_3$ belongs to $\Ocal_3^- \cap \Ocal_2^+$.
\end{itemize}
 \end{remark} 
 
This last remark brings us to the next definition.
 
\begin{definition}[{Unrevisited set, \cite[p.463]{robbin} and \cite{shub}}]Let $X$ be a smooth manifold. A set $W \subseteq X$ is called \textbf{unrevisited} for a diffeomorphism $f: X \rightarrow X$ if for any integer $m \in \N$,
 $$
  x, f^m (x) \in W \Longrightarrow \: \forall k \in \{0,\cdots,m \}, f^k (x) \in W .
 $$
\end{definition}

We say that a set is unrevisited for the flow $\varphi^t$ if it is unrevisited for the time-$1$ map $\varphi^1$. The existence of unrevisited neighborhoods goes back to the work of Robbin in 1971:
 
\begin{prop}[{Robbin - Hirsh, Palis, Pugh and Shub, \cite{robbin}, \cite[\S 6 and \S 7]{HPPS}}]\label{prop: Robbin existence of unrevisited neighborhoods}Let $K$ be a basic set and assume that $\varphi^t$ satisfies the transversality assumption. Then there exists arbitrarily small unrevisited neighborhoods $\Vcal$ of $K$.
\end{prop}

In Robbin's article, the proof was given for Axiom A diffeomorphisms and it follows from a construction of particular neighborhoods of hyperbolic sets given by Hirsh, Palis, Pugh and Shub \cite{HPPS} who explained how to generalize it for flows. This proposition can also be deduced from the work of Conley \cite{conley} up to a small perturbation of the vector field. Precisely, up to a small perturbation of the vector field and for every basic set $K$ there exists an open neighborhood $\Vcal \supset K$ such that for all $T \geq 0$:
$$ x, \varphi^T (x) \in \Vcal \Longrightarrow \forall t \in [0,T], \: \: \varphi^t (x) \in \Vcal.$$
In particular, $\Vcal$ is unrevisited. Even if this last definition seems to be more natural for flows, we chose the diffeo-like definition which will be more convenient for the analysis of the lifted Hamiltonian dynamic on the phase space as we will see later on. From now on, we assume that $\varphi^t$ satisfies the transversality assumption (\ref{Transversality assumption}). Unrevisited neighborhoods will be a very important tool of our analysis. They will be a purely dynamical alternative to the $\Ccal^1$ linearizing charts near critical points which were used in \cite{DR0} for Morse-Smale gradient flows, as we can witness in figure \ref{fig: unrevisited neighborhoods near a basic set}. But first, let us mention some of their properties. All the results mentionned below about unrevisited neighborhoods were not precisely stated in the litterature. They should be attributed to Robbin, Hirsh, Palis, Pugh, Shub, Conley and Easton to the best of our knowledge.

\subsubsection{Properties of unrevisited neighborhoods}
If $\Vcal$ is an unrevisited neighborhood of $K$, then

\begin{figure}[!t]
\centering
\def\svgscale{0.7}
 \executeiffilenewer{unrevisited_neighborhoods.svg}{unrevisited_neighborhoods.pdf}%
 {inkscape -z -D --file=unrevisited_neighborhoods.svg %
 --export-pdf=unrevisited_neighborhoods.pdf --export-latex}%
 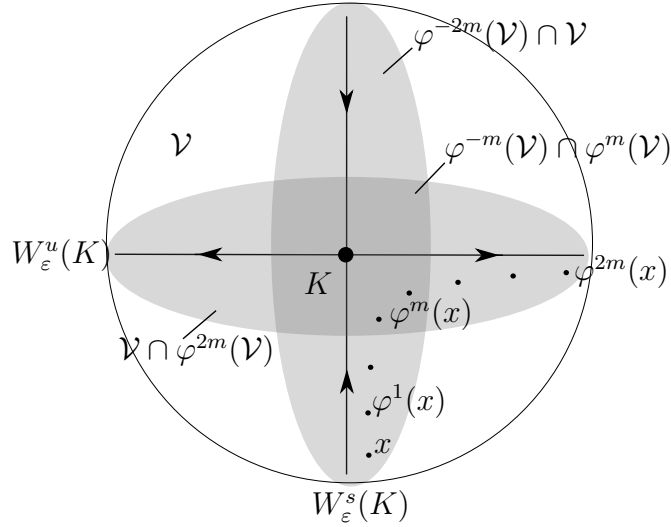%

  \caption{Illustration of some unrevisited neighborhoods near a basic set.}
  \label{fig: unrevisited neighborhoods near a basic set} 
\end{figure}

\begin{enumerate}[label={\bfseries (P\arabic*)}]
 \item The intersection of two unrevisited sets is also an unrevisited.
 \item we have a uniform approximation of the stable and unstable manifolds in the sense of the following lemma.
\end{enumerate}

 \begin{lemma}[Uniform convergence of unrevisited neighborhoods]\label{lemme: cv nonrevisite vers variete stable et instable}
    Let $\Vcal$ be an unrevisited neighborhood of a basic set $K$ which satisfies $\overline{\Vcal} \cap \Omega = K$. The sequence of unrevisited neighborhoods $\Vcal \cap \varphi^m (\Vcal)$  is decreasing with respect to $m \in \N$ and converges uniformly to $W^u (K) \cap \overline{\Vcal}$ as $m$ tends to $+ \infty$, in the sense that
    $$
	  \sup_{y \in \Vcal \cap \varphi^{m} (\Vcal)} d_g \left( y , W^{u} (K) \cap \overline{\Vcal} \right) \underset{m \to + \infty}{\longrightarrow} 0.
    $$
    A similar relation holds for the stable manifolds if we replace $\varphi^{m}$ by $\varphi^{-m}$.
\end{lemma}

\begin{proof}
   Let us prove the relation for the unstable manifold. The other one can be proved similarly if we replace $\varphi^t$ by $\varphi^{-t}$. By contradiction, let us assume we can find $\varepsilon > 0$ and a sequence of points 
  $$y_m \in \Vcal \cap \varphi^m (\Vcal) \: \mbox{ such that } \: d_g (y_m, W^u (K) \cap \Vcal) > \varepsilon$$
  for every $m \in \N$. By compactness of $M$, one can extract a subsequence $(y_{m_k})_k$ of $(y_m)_m$ which converges to a point $y_{\infty}$ as $k$ tends to $+\infty$. We have the implication
  $$ \forall k \in \N, \: d_g (y_{m_k}, W^u (K) \cap \Vcal) > \varepsilon   \Longrightarrow d_g (y_{\infty}, W^u (K) \cap \Vcal) \geq \varepsilon  .$$ Moreover, the sequence of unrevisited neighborhoods $\Vcal \cap \varphi^m (\Vcal)$ is decreasing with respect to $m \in \N$. So taking the closure, the sequence of compacts set $\overline{\Vcal} \cap \varphi^m (\overline{\Vcal})$ is also decreasing. The limit point $y_{\infty}$ belongs to $\overline{\Vcal} \cap \varphi^m (\overline{\Vcal})$ for every $m \in \N$ and it follows that
  $$
        y_{\infty} \in \bigcap_{m \in \N} \varphi^m (\overline{\Vcal}). 
  $$
  It remains to prove that $y_{\infty}$ lies in the unstable manifold of $K$. By contradiction, let us assume that $y_{\infty} \notin W^u (K)$. Thanks to Lemma \ref{partition of the manifold in stable manifolds} which decomposes $M$ into the unstable manifolds of the basic sets, there exists a basic set $K'$ distinct from $K$ such that 
  $$ y_{\infty} \in W^u (K').$$
  Now, let us fix $\varepsilon > 0$ sufficiently small so that $W_{\varepsilon}^u (K') \cap \overline{\Vcal} = \emptyset$. By definition, there exists $m \in \N$ such that $\varphi^{-m} (y_{\infty}) \in W_{\varepsilon}^u (K') \subset M \setminus \overline{\Vcal}$. Since $M \setminus \overline{\Vcal}$ is an open set, the point $\varphi^{-m} (y_p)$ does not belong to $\overline{\Vcal}$ for $p$ large enough. In particular, if we choose $p$ sufficiently large so that $p \geq m$ then we obtain that $\varphi^{-m} (y_p) \in \varphi^{-m} (\Vcal \cap \varphi^p (\Vcal)) \subset \Vcal$ and $\varphi^{-m} (y_p) \notin \overline{\Vcal}$. This gives the expected contradiction.
  \end{proof}

A direct consequence of this lemma is that 
\begin{equation}\label{eq: unrevisited and unstable manifold}
    \bigcap_{m \in \N} \varphi^m (\Vcal) = W^u (K) \cap \Vcal.
\end{equation}
Of course the previous lemma can be adapted for the stable manifold and it gives similarly
$$ \bigcap_{m \in \N} \varphi^{-m} (\Vcal) = W^s (K) \cap \Vcal.$$
Another way to look at these relations is to think about the \textit{exit time} of $\Vcal$. For every $x \in \Vcal$, we have
$$ \Card \{ m \in \N, \: \varphi^{m} (x) \in \Vcal \} = +\infty \Longleftrightarrow x \in W^u (K) \cap \Vcal $$
and similarly
$$ \Card \{ m \in \N, \: \varphi^{-m} (x) \in \Vcal \} = +\infty \Longleftrightarrow x \in W^s (K) \cap \Vcal .$$

\begin{enumerate}[resume*]
 \item Moreover if $\overline{\Vcal} \cap \Omega = K$ then \textbf{(P1) + (P2)} implies that the sequence $(\varphi^m (\Vcal) \cap \varphi^{-m} (\Vcal))_{m \in \N}$ converges uniformly to $$\left( W^s (K) \cap \Vcal \right) \cap \left( W^u (K) \cap \Vcal \right) = (W^s (K) \cap W^u (K) )\cap \Vcal = K $$ in the sense of Lemma \ref{lemme: cv nonrevisite vers variete stable et instable}.
\end{enumerate}

\begin{remark}
\begin{itemize}[align=left, leftmargin=*, noitemsep]
 \item Note that we used the transversality assumption twice for \textbf{(P3)}: one time for the existence of unrevisited neighborhoods and a second time for relation $W^s (K) \cap W^u (K) = K$. Also, \textbf{(P3)} is related with the local maximality of basic sets. Indeed, from \textbf{(P3)} we deduce
 $$ \bigcap_{m \in \N} \varphi^m (\Vcal) \cap \varphi^{-m} (\Vcal) = \bigcap_{m \in \Z} \varphi^m (\Vcal) = K,$$
 and thus we recover the local maximality of $K$:
$$ K \subseteq \bigcap_{t \in \R} \varphi^t (\Vcal) \subseteq \bigcap_{m \in \Z} \varphi^m (\Vcal) = K. $$ 
 \item We will see in Section \ref{section: Dynamical proofs} (Lemma \ref{lemma: equivalence filtration and unrevisited}) that the existence of unrevisited neighborhoods implies the existence of a filtration. 
\end{itemize}
\end{remark}

Let us present another result which describes the closure of the stable and unstable manifold of a basic set on an unrevisited neighborhood.

\begin{lemma}\label{lemma: clos. of stable man on unrevisited neigh.}
   Let $K$ be a basic set and $\Vcal$ be some unrevisited neighborhood of $K$ such that $\overline{\Vcal}\cap \Omega = K$. The following equalities hold:
   $$ \overline{W^s (K) \cap \Vcal} = W^s (K) \cap \overline{\Vcal} \quad \mbox{and} \quad \overline{W^u (K) \cap \Vcal} = W^u (K) \cap \overline{\Vcal} .$$
  \end{lemma}

 \begin{proof}
   Let us prove the equality for the unstable manifold, the other one is proven similarly if we replace $\varphi^{t}$ by $\varphi^{-t}$. The only nontrivial thing to show is the inclusion $\overline{W^u (K) \cap \Vcal} \subseteq W^u (K) \cap \overline{\Vcal}$. By contradiction, assume there exists a sequence $(x_p)_{p}$ of points in $W^u (K) \cap \Vcal$ which converges to an element $x_{\infty} \notin W^u (K) \cap \overline{\Vcal}$. Since $x_p \in \Vcal$ for all $p \in \N$, we must have $x_{\infty} \in \overline{\Vcal}$ and thus $x_{\infty} \notin W^u (K)$. Moreover, thanks to the decomposition of $M$ into the unstable manifold of the basic sets given by Lemma \ref{partition of the manifold in stable manifolds}, there exists a basic set $K'$ such that 
   $$ x_{\infty} \in W^u (K'). $$
   Now, let us choose $\varepsilon$ sufficiently small so that the following equality hold:
   $$ W_{\varepsilon}^u (K') \cap \overline{\Vcal} = \emptyset $$
   By definition of $W^u (K')$, there exists an integer $m \in \N$ such that $\varphi^{-m} (x_{\infty}) \in W_{\varepsilon}^u (K')$. Since $\varphi^{-m}
 (x_{p})$ converges to $\varphi^{-m} (x_{\infty}) \in M \setminus \overline{\Vcal}$ as $p \to + \infty$ and since $M \setminus \overline{\Vcal}$ is open, there exists $p_0 \in \N$ such that for every $p \geq p_0$,
 $$ \varphi^{-m} (x_p) \notin \Vcal. $$
 This gives the expected contradiction as $\Vcal$ is unrevisited and $x_p \in W^u (K) \cap \Vcal$.
 \end{proof} 
  
Eventhough we have $\overline{W^s (K)} = \cup_{j, K_j \leq K} W^s (K_j)$, we can deduce from the previous lemma that $\overline{W_{\varepsilon}^s (K)} \subset W^s (K)$ for $\varepsilon$ sufficiently small, see figure \ref{fig: closure of local stable manifold}.
     
\begin{lemma}\label{lemma: closure of local stable manifold on a basic set}
 For $\varepsilon \ll 1$, we have $\overline{W_{\varepsilon}^s (K)} \subseteq W_{2\varepsilon}^s (K)$.
\end{lemma}

\begin{proof}Let us consider an unrevisited neighborhood $\Vcal$ of $K$ such that $\overline{\Vcal} \cap \Omega = K$ and assume that $\varepsilon$ is chosen sufficiently small so that $W_{2\varepsilon}^s (K)$ is well defined and  
 $$ B_g (K,2 \varepsilon ) \subset \Vcal, $$
 where $B_g (K,2 \varepsilon )$ denotes the set of point at (geodesic) distance at most $2 \varepsilon$ of $K$. Therefore, we can deduce the result from Lemma \ref{lemma: clos. of stable man on unrevisited neigh.}.
\end{proof}

\begin{figure}[!t]
\centering
\def\svgscale{0.7}
 \executeiffilenewer{closure_local_stable_man.svg}{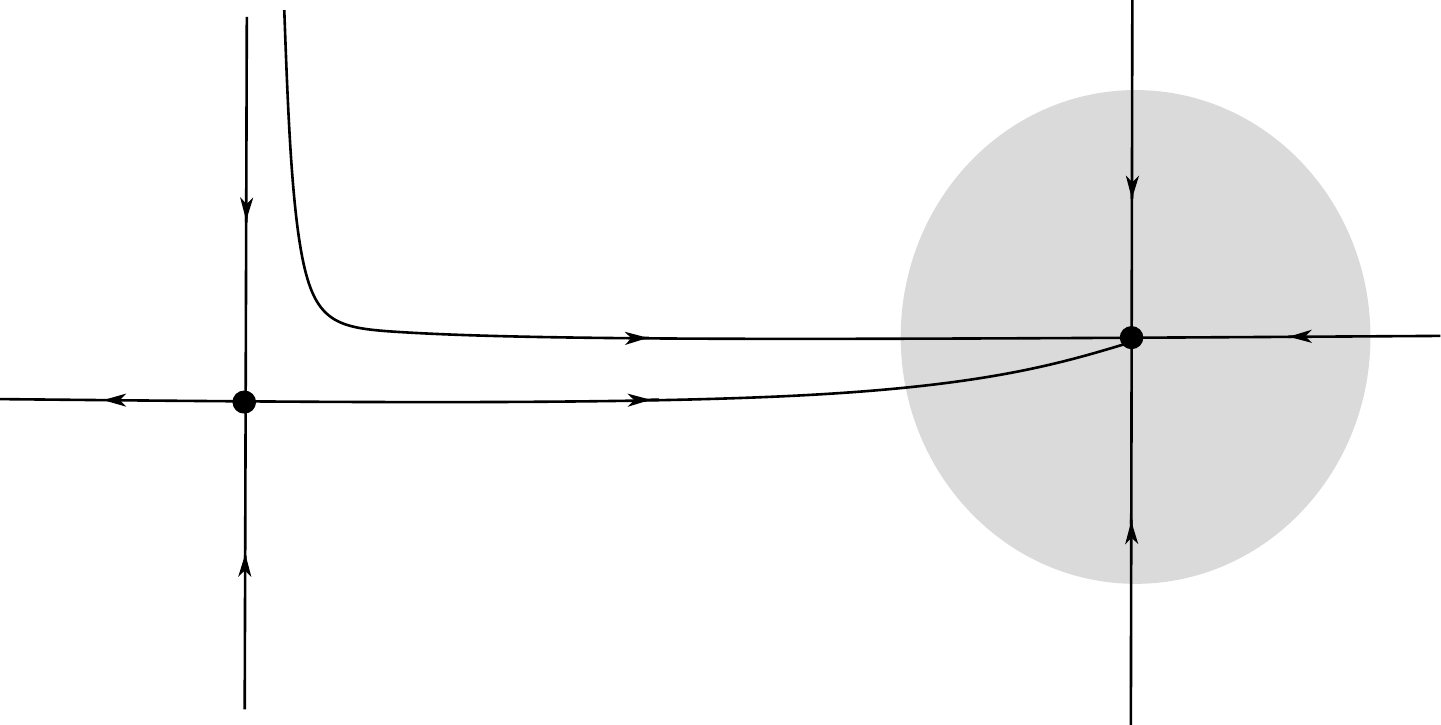}%
 {inkscape -z -D --file=closure_local_stable_man.svg %
 --export-pdf=closure_local_stable_man.pdf --export-latex}%
\begingroup%
  \makeatletter%
  \providecommand\color[2][]{%
    \errmessage{(Inkscape) Color is used for the text in Inkscape, but the package 'color.sty' is not loaded}%
    \renewcommand\color[2][]{}%
  }%
  \providecommand\transparent[1]{%
    \errmessage{(Inkscape) Transparency is used (non-zero) for the text in Inkscape, but the package 'transparent.sty' is not loaded}%
    \renewcommand\transparent[1]{}%
  }%
  \providecommand\rotatebox[2]{#2}%
  \newcommand*\fsize{\dimexpr\f@size pt\relax}%
  \newcommand*\lineheight[1]{\fontsize{\fsize}{#1\fsize}\selectfont}%
  \ifx\svgwidth\undefined%
    \setlength{\unitlength}{414.83168234bp}%
    \ifx\svgscale\undefined%
      \relax%
    \else%
      \setlength{\unitlength}{\unitlength * \real{\svgscale}}%
    \fi%
  \else%
    \setlength{\unitlength}{\svgwidth}%
  \fi%
  \global\let\svgwidth\undefined%
  \global\let\svgscale\undefined%
  \makeatother%
  \begin{picture}(1,0.50319354)%
    \lineheight{1}%
    \setlength\tabcolsep{0pt}%
    \put(0,0){\includegraphics[width=\unitlength,page=1]{closure_local_stable_man.pdf}}%
    \put(0.12371183,0.26355969){\makebox(0,0)[lt]{\lineheight{1.25}\smash{\begin{tabular}[t]{l}$K_i$\end{tabular}}}}%
    \put(0.80241624,0.30597322){\makebox(0,0)[lt]{\lineheight{1.25}\smash{\begin{tabular}[t]{l}$K_j$\end{tabular}}}}%
    \put(0.64955725,0.3266381){\makebox(0,0)[lt]{\lineheight{1.25}\smash{\begin{tabular}[t]{l}$W_{\varepsilon}^s(K_j)$\end{tabular}}}}%
  \end{picture}%
\endgroup%

  \caption{Illustration of Lemma \ref{lemma: closure of local stable manifold on a basic set}.}
  \label{fig: closure of local stable manifold} 
\end{figure}

Now, let us present another property of the local stable manifold of a basic set which also uses unrevisited neighborhoods. 

  \begin{lemma}[Uniform convergence]\label{convergence uniforme vers K}Let $K$ be a basic set. For every $0 < \varepsilon_2 < \varepsilon_1 \ll 1$, there exists a constant $T:= T(\varepsilon_1, \varepsilon_2) > 0$ such that 
 $$
  \varphi^T \left(  W_{\varepsilon_1}^s (K) \right) \subseteq W_{\varepsilon_2}^s (K), \quad  \varphi^{-T} \left(  W_{\varepsilon_1}^u (K) \right) \subseteq W_{\varepsilon_2}^u (K)
 $$
 \end{lemma}
 
 \begin{proof}
  This can be proved by contradiction using the Lemma \ref{lemma: closure of local stable manifold on a basic set} and the definition \ref{eq: local decomposition of the stable manifold of a basic set} of the local stable manifold of $K$.
 \end{proof}

   \section{Escape functions for Axiom A flows}\label{section Escape function}
   
   Following the strategy of Faure and Sjöstrand \cite{FS}, we will construct some function, called an escape function, which will allow us to define anisotropic Sobolev spaces on which the Lie derivative operator $-\Lie_V$ has nice spectral properties. This function is related to the construction of energy functions (also called Lyapunov functions) whose existence will be stated in paragraph \ref{subsection: energy function for the hamiltonian flow} and proved in Sections \ref{section: Dynamical proofs} and \ref{section : Compactness result and energy functions for the Hamiltonian flow}.
   
   According to the strategy of \cite{FS}, we need to construct an energy function on the unitary cotangent bundle $S^*M$ which is increasing along the (projected) Hamiltonian flow $\widetilde{\Phi}^t$. We choose here to split its construction into the one of two energy functions that are slightly easier to build independently: one on the base manifold $M$ and one on the fibers of $S^* M$. 
   
   \subsection{Energy functions on $M$}
   
    It is known from the work of Conley \cite{conley} that any continuous flow on a compact manifold behaves like a gradient flow outside an invariant set called the \textit{chain recurrent set} (see for instance \cite{palis-demelo} for a definition). For example, it is the case for gradient flows of Morse functions where the chain recurrent set equals the set of hyperbolic fixed points and an energy function is given by the Morse function itself. Axiom A flows are type of flows where the chain recurrent set equals the nonwandering set. The following proposition was originally proved by Conley \cite{conley} and Wilson \cite{wilson} (in a slightly better version). Yet, for the sake of completness and since its proof will be very instructive for the construction of energy fonctions on $S^* M$, we will provide another proof using filtrations and unrevisited neighborhoods in Section \ref{section: Dynamical proofs} for the next proposition and in Section \ref{section : Compactness result and energy functions for the Hamiltonian flow} for the analysis on $S^*M$.
   
    \begin{prop}[Energy function for Axiom A flows]\label{prop energy function for Axiom A flows}
	Let $\varphi^t$ be an Axiom A flow which satisfies the transversality assumption (\ref{Transversality assumption - first}). For every $\varepsilon > 0$ and for every family of pairwise distinct real numbers $(\lambda_i)_{1\leq i \leq N}$ compatible with the graph structure in the sense that 
	$$ K_i \leq K_j \Longrightarrow \lambda_i \leq \lambda_j,$$
	there exists an energy function $E \in \Ccal^{\infty}(M)$, $\varepsilon$-neighborhoods $ \Ncal_i$ of $K_i$ and a constant $\eta > 0$ such that:
 $$ \Lie_V E \geq 0 \mbox{ on } M, \: \mbox{and } \Lie_V E >  \eta \mbox{ on } M \setminus \left( \cup_{i = 1}^N \Ncal_i \right).$$
Moreover, for all $1 \leq i \leq N$, the map $E$ is close to $\lambda_i$ on each $\Ncal_i$ in the sense that 
$$
	E = \lambda_i \mbox{ on } K_i \mbox{ and } \sup_{x \in \Ncal_i} |E(x) - \lambda_i | < \varepsilon ( \lambda_N - \lambda_1).
$$
\end{prop}
  
  \subsection{Energy functions for the Hamiltonian flow}\label{subsection: energy function for the hamiltonian flow}
  
Let us define the following $\widetilde{\Phi}^t$-invariant subset of $S^*M$:
 
\begin{align*}
 \begin{split}
  \Sigma_{uo} &:= \bigcup_{x \in M} \kappa \left(E_{so}^* (x)\right), \qquad \Sigma_{s}:= \bigcup_{x \in M}  \kappa \left(E_u^* (x)\right),\\
  \Sigma_{u} &:= \bigcup_{x \in M} \kappa \left(E_s^* (x)\right), \qquad \Sigma_{so}:= \bigcup_{x \in M}  \kappa \left(E_{uo}^* (x)\right),
 \end{split}
 \end{align*}
 where $\kappa$ denotes the projection on the unitary cotangent bundle defined in (\ref{eq : definition pi et kappa}).
They will be our basic ingredients to construct energy functions on $S^*M$. Indeed, following the ideas of Faure-Sjöstrand \cite{FS} and Dang-Rivière \cite{DR0, DRI}, we will see that $(\Sigma_{uo},\Sigma_{s})$ and $(\Sigma_{u},\Sigma_{so})$ are both a couple of repelling and attracting compact invariant sets for the Hamiltonian flow $\widetilde{\Phi}^t$. It will be enough to construct an energy function on the fiber. First, let us recall that our \textit{transversality assumption} implies that 
$$\Sigma_{uo} \cap \Sigma_{s} = \emptyset = \Sigma_{u} \cap \Sigma_{so}.$$
The following lemma proved in \S \ref{subsection : proof Lemma attracting and repelling sets for the Hamiltonian flow} tells us that they are indeed attracting and repelling sets for the Hamiltonian flow:
 
 \begin{lemma}\label{lemma attractor and repeller on the phase space}
 For every $(x,\xi) \in S^* M \setminus \left( \Sigma_{s} \cup \Sigma_{uo} \right)$, we have
 $$ d_{S^* M} \left( \widetilde{\Phi}^{t} (x,\xi), \Sigma_{s} \right) \underset{t \to +\infty}{\longrightarrow } 0 \mbox{ and } 
 d_{S^* M} \left( \widetilde{\Phi}^{-t} (x,\xi), \Sigma_{uo} \right) \underset{t \to +\infty}{\longrightarrow } 0.$$
 Similarly, for every $(x,\xi) \in S^* M \setminus \left( \Sigma_{so} \cup \Sigma_{u} \right)$, we have
 $$ d_{S^* M} \left( \widetilde{\Phi}^{t} (x,\xi), \Sigma_{so} \right) \underset{t \to +\infty}{\longrightarrow } 0 \mbox{ and } 
 d_{S^* M} \left( \widetilde{\Phi}^{-t} (x,\xi), \Sigma_{u} \right) \underset{t \to +\infty}{\longrightarrow } 0 .$$
\end{lemma}

Moreover, contrary to Anosov flows for which it is rather immediate, we need to make sure that these sets are compact sets. The next proposition is similar to the compactness result of Dang and Rivière \cite[lem. 3.7, p.15]{DR0}. The proof is given in \S \ref{subsection : proof compactness result}.
 
 \begin{prop}[Compactness]\label{prop compactness}Let us assume that $\varphi^t$ satisfies the tranversality assumption (\ref{Transversality assumption}). Then, the subsets
 $\Sigma_{u}$, $\Sigma_{uo}$, $\Sigma_{so}$ and $\Sigma_{s}$ of $S^*M$ which are $\widetilde{\Phi}^t$-invariant compact sets.
 \end{prop}
 
 To construct energy functions, we also need the existence of arbitrarily small stable neighborhoods. The proof of next lemma is given in \S \ref{subsection : Lemma invariant neighborhoods for the Hamiltonian flow}.
 
\begin{lemma}[Invariant neighborhoods]\label{lemma invariant neighborhoods}
 For every $\varepsilon > 0$, there exist $\varepsilon$-neighborhoods $\Ucal^{s/so}$ (resp. $\Ucal^{u/uo}$) of $\Sigma_{s/so}$ (resp. $\Sigma_{u/uo}$) which are $\widetilde{\Phi}^{1}$-stable (resp. $\widetilde{\Phi}^{-1}$-stable).
 \end{lemma}
 
 As a consequence of these three results, we obtain energy functions on the fiber of $S^* M$. The proof of the following proposition is given in \S \ref{subsection : Proof energy functions for the Hamiltonian flow}.
 
\begin{prop}[Energy functions for the Hamiltonian flow]\label{prop energy function on the unitary cotangent bundle}
	Let $\varphi^t$ be an Axiom A flow which satisfies the tranversality assumption (\ref{Transversality assumption}). For every $\varepsilon > 0$, there exist energy functions $E_{\pm} \in \Ccal^{\infty}(S^* M; [0,1])$, $\varepsilon$-neighborhoods $\Wcal^{s/so}$ of $\Sigma_{s/so}$, $\Wcal^{uo/u}$ of $\Sigma_{uo/u}$ and a constant $\eta > 0$ such that:
	\begin{align*}
	 \begin{split}
		 \Lie_{X_H} E_{+} &\geq 0 \mbox{ on } S^* M \: \mbox{and } \Lie_{X_H} E_{+} > \eta \mbox{ on } S^* M \setminus  (\Wcal^{uo} \cup \Wcal^{s}) , \\ 
		  \Lie_{X_H} E_{-} &\geq 0 \mbox{ on } S^* M \: \mbox{and } \Lie_{X_H} E_{-} > \eta \mbox{ on } S^* M \setminus (\Wcal^{u} \cup \Wcal^{so}) .
	 \end{split}
	\end{align*}  
 Moreover, the map $E_{\pm}$ are constant on each $\Sigma_{*}$ and we have the estimate
 \begin{align*}
	\sup_{(x,\xi) \in \Wcal^{u/uo}} \left| E_{\pm}(x,\xi) - 0 \right| \leq \varepsilon \quad \mbox{and} \quad \sup_{(x,\xi) \in \Wcal^{s/so}} \left| E_{\pm}(x,\xi) - 1 \right| \leq \varepsilon.
 \end{align*}
\end{prop}

  \subsection{The escape functions}
  
Next, we give a global escape function which extends the construction of Dyatlov and Guillarmou \cite{dyatlov-guillarmou_2016} to the whole manifold and coincide with the one of Dang and Rivière \cite{DR0} for Morse-Smale gradient flows and the one of Faure-Sj\"ostrand \cite{FS} for Anosov flows. The proof of next proposition is given in Section \ref{section : Proof of escape function}.
 
\begin{prop}[Escape function]\label{prop escape function}
 Let $u,s,n_0 \in \R$ such that $u < 0 \leq n_0 < s$. There exists a smooth function $m(x,\xi) \in C^{\infty} (T^* M)$ called an \textbf{order function} and an \textbf{escape function} $G_m \in \Ccal^{\infty} (T^*M)$ defined by:
 $$
  G_{m} (x, \xi) = m(x, \xi) \log \sqrt{1 + f(x,\xi)^2}
 $$
where $f \in C^{\infty}(T^* M)$ is positive everywhere and homogeneous of degree $1$ in $\xi$ as soon as $|\xi| \geq 1$ and where $m$ is defined by $m(x,\xi) = \chi (|\xi|^2) E(x,\frac{\xi}{|\xi|})$ with $\chi$ being a smooth cut-off function such that $\chi = 0$ on $]-\infty, 1/2]$, $\chi = 1$ on $[1,+\infty[$, $\chi \geq 0$ everywhere as in figure \ref{figure chi} and $E$ being a linear combination of previous energy functions: 
$$E(x,\xi):= -E(x) + s + (u - n_0) E_+ (x,\xi) + (n_0 - s) E_- (x,\xi).$$
Moreover, we have the following estimates:
\begin{enumerate}
 \item There exist conical neighborhoods $\widetilde{\Ncal}^{s/o/u}$ of  $\bigcup_{x \in M} E_{s/o/u}^* (x) \setminus 0_M$ such that $f = |\xi(V)|$ on $\widetilde{\Ncal}^o$ and for $|\xi| \geq 1$,
 $$ \frac{1}{2} n_0 \leq m \leq  n_0  \mbox{ on } \widetilde{\Ncal}^{o} ,$$
 $$
 m \geq \frac{s}{4} \mbox{ on } \widetilde{\Ncal}^{s},  \quad \mbox{and} \quad m \leq \frac{u}{2} \mbox{ on } \widetilde{\Ncal}^{u} .
 $$
 Also, the open sets can be chosen arbitrarily close to the invariant distributions $E_s^*$, $E_o^*$ and $E_u^*$ as in Proposition \ref{prop energy function on the unitary cotangent bundle}.
 \item The map $G_m$ is strictly decreasing along the flow $\Phi^t$ except at points $(x,\xi)$ where $|\xi|$ is small or where $x$ is in a small neighborhood of the nonwandering set and $\xi$ is in a conical neighborhood of $E_o^*$: for all $1 \leq i \leq N$, there exist an open neighborhood $\Ncal_i$ of $K_i$ and a radius $R > 0$ such that for all $(x,\xi)\in \cup_{x\in \cup N_i} T_x^*M\setminus \widetilde{\Ncal}^o$ such that $ |\xi| \geq R$ \textbf{and} for all $(x,\xi) \in \cup_{x\notin \cup_i \Ncal_i} T_x^*M$ such that $ |\xi| \geq R$
 $$
   X(G_{m}) (x,\xi) <  -c \min (s,|u|)  =: - C_m  < 0
 $$
 with $c > 0$ which is independent of the constants $u, n_0, s$ and of the size of the conical neighborhoods. 
 \item More generally, for every $(x,\xi) \in T^* M$ such that $|\xi| \geq R$, we have
 $$
   X(G_{m}) (x, \xi) \leq 0.
 $$
\end{enumerate}
\end{prop}

  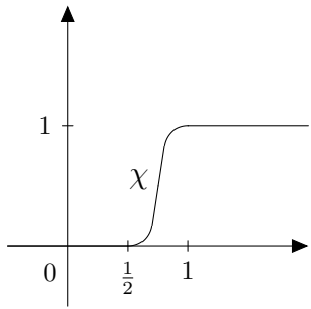
\begin{figure}[!t]
  \centering
  \begin{tikzpicture}[line cap=round,line join=round,>=triangle 45,x=1.6cm,y=1.6cm]
\draw[->,color=black] (-0.5,0) -- (2,0);
\draw[->,color=black] (0,-0.5) -- (0,2);
\draw[shift={(1,0)},color=black] (0pt,2pt) -- (0pt,-2pt) node[below] {\footnotesize $1$};
\draw[shift={(0.5,0)},color=black] (0pt,2pt) -- (0pt,-2pt) node[below] {\footnotesize $\frac{1}{2}$};
\draw[shift={(0,1)},color=black] (2pt,0pt) -- (-2pt,0pt) node[left] {\footnotesize $1$};
\draw[color=black] (0pt,-10pt) node[left] {\footnotesize $0$};
\clip(-0.5,-0.5) rectangle (2,2);
\draw [domain=-0.5:0.5] plot(\x,{(-0-0*\x)/-0.5});
\draw [domain=1.0:2.0] plot(\x,{(--0.5-0*\x)/0.5});
\draw [shift={(0.5,0.2)}] plot[domain=4.69:5.32,variable=\t]({1*0.2*cos(\t r)+0*0.2*sin(\t r)},{0*0.2*cos(\t r)+1*0.2*sin(\t r)});
\draw [shift={(0.49,0.2)}] plot[domain=5.38:6.11,variable=\t]({1*0.21*cos(\t r)+0*0.21*sin(\t r)},{0*0.21*cos(\t r)+1*0.21*sin(\t r)});
\draw [shift={(1.01,0.8)}] plot[domain=2.23:2.97,variable=\t]({1*0.21*cos(\t r)+0*0.21*sin(\t r)},{0*0.21*cos(\t r)+1*0.21*sin(\t r)});
\draw [shift={(1,0.8)}] plot[domain=1.55:2.18,variable=\t]({1*0.2*cos(\t r)+0*0.2*sin(\t r)},{0*0.2*cos(\t r)+1*0.2*sin(\t r)});
\draw (0.7,0.17)-- (0.75,0.5);
\draw (0.8,0.83)-- (0.75,0.5);
\draw (0.42,0.73) node[anchor=north west] {$\chi $};
\end{tikzpicture}
 \captionof{figure}{cut-off function $\chi$ \label{figure chi}}
 \end{figure}

 \section{Anisotropic Sobolev spaces}\label{section Anisotropic Sobolev spaces}
 
 The purpose of this section is to construct some Hilbert spaces in which the operator $-\Lie_V^{(k)}$ acting on sections of $\Lambda^k T^*M \otimes \C$ has good spectral properties. For $k = 0$ and for Anosov flows, Faure and Sjöstrand defined an anisotropic Sobolev space which can be roughly written as $\exp (G_m (x,-iD))^{-1} L^2 (M ;  \C)$ using the escape function $G_m$ of Proposition \ref{prop escape function}. In this part, we recall the definition of these Hilbert spaces starting from our escape function and we extend it for $0 \leq k \leq n$ as in \cite{dyatlov-zworski}. More precisely, we construct a pseudodifferential operator acting on sections of $ \Lambda^k T^*M \otimes \C$ using the map $e^{G_m}$. First, let us define 
 $$
  A_m (x,\xi) = \exp (G_m (x,\xi)) \in \Ccal^{\infty} (T^* M),
 $$
 with $G_m$ being the escape function obtained in Proposition \ref{prop escape function}.
 We consider for every $0 \leq k \leq n$ the vector bundle $\Ecal_k:= \Lambda^k T^*M \otimes \C$ of complex-valued differential $k$-forms\footnote{Note that a Hilbert structure on $L^2 (M; \Ecal_k) $ can be defined by setting
 $$
  \w{\alpha, \beta}_{L^2}^{(k)}:= \int_M \w{\overline{\alpha(x)}, \beta(x)}_{g^{\otimes k}} dvol_g (x),
 $$
 where $\w{.,.}_{g^{\otimes k}}$ denotes the scalar product induced by the Riemannian metric $g^{\otimes k}$ on the bundle $\Lambda^k (T^* M) \rightarrow M$ and $g^{\otimes k}$ denotes the metric on $\Lambda^k (T^* M)$ taking $k$-times the tensor product of $g$.}.

 Now, we briefly recall how to define pseudodifferential operators on vector bundles - see for instance \cite[app. C.1, p.29]{dyatlov-zworski} and \cite[\S 9.2 p.40]{DRI} for more details. Let us consider a finite cover $(U_i)_{i \in I}$ of $M$ by contractible open sets and local trivializations $\chi_i: \rest{\Ecal_k}{U_i} \rightarrow \widetilde{U}_i \times \R^2 \times \R^{\frac{n!}{k! (n-k)!}}$ where $\widetilde{U}_i$ is an open subset of $\R^n$. We define pseudo-differential operators on $\Ecal_k$ using the Weyl quantization in these local trivialization charts. In other words, we define the operator $\Op_i (A_m): \Omega^k (U_i ; \C) \rightarrow \Omega^k (M ; \C ) $ by the formula 
 $$
  \chi_i \circ \Op_i (A_m) \circ \chi_i^{-1} (f dx^{i_1} \wedge \cdots \wedge dx^{i_k}) = \Opw^w (A_m) (f) dx^{i_1} \wedge \cdots \wedge dx^{i_k},
 $$
where $\Opw^w$ denotes the usual \textbf{Weyl quantization} on $\R^n$ - see \cite[chap. 4, p.56]{zworski} for a definition. Now, if we consider a partition of unity $\zeta_i \in \Ccal_c^{\infty} (U_i;[0,1])$ associated to $(U_i)$, i.e. $\supp \zeta_i \subset U_i$ and $\sum_i \zeta_i =  1$ on $M$, and functions $\psi_i \in \Ccal_c^{\infty} (U_i)$ such that $\psi_i = 1$ on $\supp \zeta_i$, then we define the pseudodifferential operator $\A_m^{(k)}: \Omega^k (M ; \C) \rightarrow \Omega^k (M ; \C ) $ by setting:
$$
 \A_m^{(k)} = \frac{1}{2} \sum_{i \in I} \psi_i \Op_i (A_m) \zeta_i  + \frac{1}{2} \left( \sum_{i \in I} \psi_i \Op_i (A_m) \zeta_i \right)^* ,
$$
where $(\sum_{i \in I} \psi_i \Op_i (A_m) \zeta_i )^*$ denotes the formal adjoint $\sum_{i \in I} \psi_i \Op_i (A_m) \zeta_i $ given by the Hilbert structure $\w{.,.}_{L^2}^{(k)}$ of $L^2 (M ; \Ecal_k)$. Note that the symbols $A_m$ belong to a class of symbols with variable order whose properties are discussed in the appendix of \cite{FRS}.

The operator $\A_m^{(k)}$ has for principal symbol $A_m (x,\xi) \Id_{\Ecal_k (x)}$, where $\Id_{\Ecal_k (x)}$ denotes the identity map on each fiber $\Ecal_k (x) = \C_x \otimes \Lambda^k (T_x^* M)$, and it is \textbf{elliptic} in the sense of \cite[def. 8, p. 40]{FRS}. Therefore, there exists a smoothing operator $\hat{r}: \D'^{,k} (M;\C) \rightarrow \Omega^k (M; \C)$ such that $\widetilde{\A}_m^{(k)}:= \A_m^{(k)} + \hat{r}: \Omega^k (M; \C) \rightarrow \Omega^k (M; \C)$ is formally self-adjoint, elliptic and invertible - see \cite[lem. 12, p. 42]{FRS}. Thus, we choose $(\widetilde{\A}_m^{(k)})^{-1}$ to be a representant of the inverse of $\A_m^{(k)}$ modulo smoothing operators by setting
$$
 (\widetilde{\A}_k^{(k)})^{-1}:= (\A_m^{(k)} + \hat{r})^{-1} .
$$

Following \cite{FRS}, \cite{FS}, we define the anisotropic Sobolev space
$$
 \boxed{\Hil_k^{m}:= (\widetilde{\A}_m^{(k)})^{-1} L^2 (M; \Ecal_k)}
$$

The space $\Hil_k^m$ can be endowed with the following Hermitian structure which makes it isometric to the space $L^2 (M; \C \otimes \Lambda^k T^* M)$: for every $u,v \in \Hil_k^m$,
$$
 \w{u,v}_{\Hil_k^m}:= \w{\widetilde{\A}_m^{(k)} u, \widetilde{\A}_m^{(k)} v}_{L^2}^{(k)} .
$$
Moreover, the space $\Hil_k^m$ is isomorphic\footnote{The idea is that any current $u \in \D'^{,k}(M)$ writes in coordinates as a $k$-form with coefficients in $\D'(M)$, i.e. $u = \sum u_{i_1,\cdots, i_k} dx^{i_1} \wedge \cdots \wedge dx^{i_k}$ where $u_{i_1,\cdots, i_k} \in \D' (M)$. A partition of unity argument gives the result.} to the space $\Hil_0^m \otimes \Omega^k (M; \C)$ and the following inclusion
$$
 \Omega^k (M; \C) \subset \Hil_k^m \subset \D'^{,k} (M ;\C).
$$
are continuous, where $\D'^{,k} (M; \C)$ is endowed with the weak topology - see \cite{schwartz}.

\subsection{Spectral properties}

Adapting the proof of \cite[Theorem 1.4]{FS} to the case of vector bundles and using the properties\footnote{These are the only properties used in the proof of \cite{FS}.} of the escape function stated in Proposition (which is the exact analogue of Lemma $1.2$ in \cite{FS}), one can establish the existence of a discrete spectrum on these anisotropic Sobolev spaces:

\begin{center}
\fbox{
\begin{minipage}{0.9\linewidth}
\begin{thm}[Discrete spectrum]\label{thm: discrete sprectrum} Let $\varphi^t$ be an Axiom A flow verifying the transversality assumption (\ref{Transversality assumption - first}). Let $G_m$ be an escape function. For every $0 \leq k \leq n$, the operator $-\Lie_V^{(k)}$ defines a maximal closed unbounded operator on $\Hil_k^m$ with domain 
$$
 \D (-\Lie_V^{(k)}) = \{u \in \Hil_k^m: \: -\Lie_V^{(k)} u \in \Hil_k^m \} \:  \mbox{ and }  \: \Lie_V^{(k)}: \D (-\Lie_V^{(k)}) \subset \Hil_k^m \rightarrow \Hil_k^m .  
$$
Moreover, there exists a constant $C_0 \in \R$ (which depends on the choice of the escape function $G_m$) such that 
$$ -\Lie_V^{(k)} \mbox{ has empty spectrum on } \re (z) \geq C_0 $$
and there exists constants $C_m >0$ (given by Proposition \ref{prop escape function}) and $C_1 > 0$ (which only depends on the vector field $V$ and the metric $g$) such that
$$ -\Lie_V^{(k)} \mbox{ has discrete spectrum on } \re (z) \geq -C_m + C_1 ,$$
where $C_m>0$ is the constant given by Proposition \ref{prop escape function}.
\end{thm}
\end{minipage}
}
\end{center}

\begin{figure}[!t]
\begin{center}
\definecolor{uququq}{rgb}{0.25,0.25,0.25}
\definecolor{zzzzzz}{rgb}{0.6,0.6,0.6}
\definecolor{qqqqff}{rgb}{0,0,1}
\definecolor{ttttff}{rgb}{0.2,0.2,1}
\begin{tikzpicture}[line cap=round,line join=round,>=triangle 45,x=0.5cm,y=0.4cm]
\draw[->,color=black] (-14,0) -- (10,0);
\draw[->,color=black] (0,-8) -- (0,8);
\clip(-14,-8) rectangle (10,8);
\fill[line width=0pt,dash pattern=on 2pt off 2pt,color=zzzzzz,fill=zzzzzz,fill opacity=0.1] (-14,-8) -- (-14,8) -- (-8,8) -- (-8,-8) -- cycle;
\draw [dash pattern=on 1pt off 1pt] (6,-8) -- (6,8);
\draw [dash pattern=on 1pt off 1pt] (-8,-8) -- (-8,8);
\draw (-8.63,-0.08) node[anchor=north west] {$-C_m + C_1$};
\draw (5.74,0) node[anchor=north west] {$C_0$};
\draw (-13.85,-6.22) node[anchor=north west] {$\scriptsize{\mbox{Essential spectrum}}$};
\draw (-2.72,-6.33) node[anchor=north west] {$\scriptsize{\mbox{Discrete spectrum}}$};
\draw (6.1,-6.35) node[anchor=north west] {$\scriptsize{\mbox{No spectrum}}$};
\draw (0,0) node[anchor=north west] {$0$};
\begin{scriptsize}
\draw [color=black] (6,0)-- ++(-2.0pt,0 pt) -- ++(4.0pt,0 pt) ++(-2.0pt,-2.0pt) -- ++(0 pt,4.0pt);
\fill [color=qqqqff] (0,0) circle (1.5pt);
\fill [color=qqqqff] (1.48,-1.27) circle (1.5pt);
\fill [color=qqqqff] (4,-2) circle (1.5pt);
\fill [color=qqqqff] (5.04,-3.45) circle (1.5pt);
\fill [color=qqqqff] (3.4,-7.12) circle (1.5pt);
\fill [color=qqqqff] (-1.36,-3.13) circle (1.5pt);
\fill [color=qqqqff] (-1.99,-5.16) circle (1.5pt);
\fill [color=qqqqff] (-4.11,-2.97) circle (1.5pt);
\fill [color=qqqqff] (-6.56,-3.98) circle (1.5pt);
\fill [color=qqqqff] (-4.88,-6.32) circle (1.5pt);
\fill [color=qqqqff] (-6.56,3.98) circle (1.5pt);
\fill [color=qqqqff] (-4.11,2.97) circle (1.5pt);
\fill [color=qqqqff] (-1.36,3.13) circle (1.5pt);
\fill [color=qqqqff] (-4.88,6.32) circle (1.5pt);
\fill [color=qqqqff] (-1.99,5.16) circle (1.5pt);
\fill [color=qqqqff] (1.48,1.27) circle (1.5pt);
\fill [color=qqqqff] (4,2) circle (1.5pt);
\fill [color=qqqqff] (5.04,3.45) circle (1.5pt);
\fill [color=qqqqff] (3.4,7.12) circle (1.5pt);
\draw [color=uququq] (-8,0)-- ++(-2.0pt,0 pt) -- ++(4.0pt,0 pt) ++(-2.0pt,-2.0pt) -- ++(0 pt,4.0pt);
\end{scriptsize}
\end{tikzpicture}
\captionof{figure}{Illustration of Pollicott-Ruelle resonances of $-\Lie_V^{(k)}$ on the anisotropic Sobolev space $\Hil_k^m$. The fact that $0$ is a resonance or not depends on $k$.}
\end{center}
\end{figure}
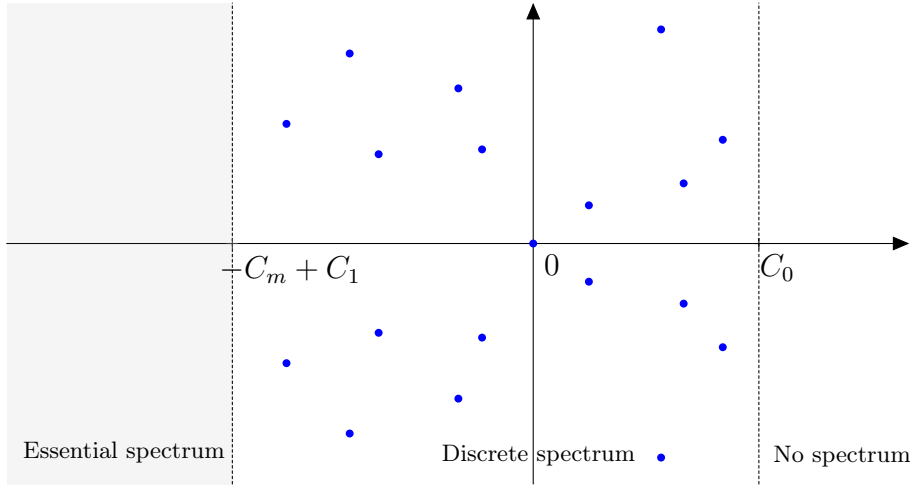

The eigenvalues of $-\Lie_V^{(k)}$ on the anisotropic Sobolev space are called the \textit{Pollicott-Ruelle resonances} of $-\Lie_V^{(k)}$. There are many (equivalent) definitions of resonances. In particular, they can be viewed as the poles of the resolvent operator $(-\Lie_V^{(k)} - z)^{-1}: \Omega^k (M; \C) \rightarrow \D'^{,k} (M; \C)$. Let us make a few remarks about them:

\begin{remark}
 \begin{itemize}[align=left, leftmargin=*, noitemsep]
  \item The discrete spectrum of $-\Lie_V^{(k)}$ is intrinsic in the sense that it does not depend on the escape function and the essential spectrum can be chosen as far as we want to the origin by taking $m$ such that $C_m \gg 1$. We refer to \cite[Thm. 1.5, p. 134]{FS} for a proof in the case of Anosov vector fields.
  \item The set of resonances is symmetric along the real axis since the vector field $V$ is real. 
  \item When $k = 0$, the resonances are included in the set $\{ \re (z) \leq 0 \}$ and the point $z = 0$ is a resonance since the constants are solutions of $\Lie_V u = 0$. This fact is not true in general for $k > 0$ and the optimal constant $h \in \R$ such that there is no spectrum in the set $\re (z) > h$ is related to the \textit{topological entropy} of  the basic sets.
  \item From the previous remark, we can see that the resonance $0$ is somehow related to Morse inequalities:
  $$ \boxed{\dim (\Ker(\Lie_V)) \geq b_0 = \dim \Hsc_0  (M),}$$
  where $b_0(M)$ is the number of connected components of $M$ and where $\Ker \Lie_V$ denotes the kernel of the Lie derivative viewed as an unbounded operator on $\Hil_0^m$.
  \item From the first point, we can deduce that the space $\Ker ((\Lie_V^{(k)})^{\ell})$ does not depend on the espace function $m$ for any $\ell \in \N$ (provided $m$ is chosen such that $-C_m+C_1<0$).
 \end{itemize}
\end{remark}

Before going deeper into topological considerations, let us recall some useful properties of the operators $-\Lie_V^{(k)}$. 

\begin{remark}
 \begin{itemize}[align=left, leftmargin=*, noitemsep]
    \item When proving the discrete spectrum theorem for Anosov vector fields, precisely in \cite[lem. 3.3, p. 343]{FS}, Faure and Sjöstrand obtained a bound on the resolvent operator which remains true in our context. For every $z$ such that $\re (z) > C_0$, we have
    $$
     \| (\Lie_V^{(k)} + z)^{-1} \|_{\Hil_k^m \rightarrow \Hil_k^m} \leq \frac{1}{\re (z) - C_0}.
    $$
    An application of Hille-Yosida theorem \cite[cor. 3.6, p. 76]{engel-nagel} gives that 
    $$ (\varphi^{-t})^*: \Hil_k^m \rightarrow \Hil_k^m, \quad \forall t \geq 0,$$
    generates a \textbf{strongly continuous semi-group} whose norm is bounded by $e^{C_0 t}$. Therefore, for every $z$ such that $\re (z) > C_0$ we can write the resolvent as follows:
    $$
     (\Lie_V^{(k)} + z)^{-1} = \int_0^{+\infty} e^{-zt } (\varphi^{-t})^* dt: \Hil_k^m \rightarrow \Hil_k^m,
    $$
    where the integral converges absolutely. Note that it is convenient to use the convention $(\varphi^{-t})^* = e^{-t \Lie_V^{(k)}}: \Hil_k^m \rightarrow \Hil_k^m$.
  \item If $z_0 $ is any resonance of $-\Lie_V^{(k)}$ such that $\re (z_0) > -C_m + C_1$, then we can define the Riesz projector
  $$
   \pi_{z_0}^{(k)}:= \frac{1}{2 i \pi} \int_{\gamma_{z_0}} (\Lie_V^{(k)} + z)^{-1} dz: \Hil_k^m \rightarrow \Hil_k^m ,
  $$
  where the integral is over a positively oriented closed curved $\gamma_{z_0}$ which surrounds the resonance $z_0$ and no other resonances. Moreover, it commutes with $\Lie_V^{(k)}$ and it has finite rank. Note that this definition still makes sense when $z_0$ is not a resonance and in that case, $\pi_{z_0}^{(k)}$ is identically $0$. Note also that $\pi_{z_0}^{(k)}$ commutes with the exterior derivatives $d$ because $\Lie_V^{(k)}$ commutes with $d$ thanks to the Cartan formula.
  \item The resolvent operator writes as a Laurent series near $z_0$:
  $$
      (\Lie_V^{(k)} + z)^{-1} = \sum_{l = 1}^{m_k (z_0)}  (-1)^{l-1} \frac{(\Lie_V^{(k)} + z)^{l-1} \pi_{z_0}^{(k)}}{(z - z_0)^{l}} + R_{z_0,k} (z)
  $$
  where $R_{z_0,k}$ is holomorphic near $z_0$.
 \end{itemize}
\end{remark}

\section{Construction of the Morse-De Rham complex} \label{section topology}

Let us define $\Res_k (V)$ as the set of resonances $z \in \C$ of the operator $-\Lie_V^{(k)}$, i.e. the set of points $z_0\in \C$ such that we can find an escape function $G_m$ with $\re (z_0) > -C_m+C_1$ and such that the algebraic multiplicity of $z_0$, denoted by $m_k (z_0)$, verifies $m_k (z_0) \neq 0$. We can then define $C_V^k(z_0)$ as the range of the projector $\pi_{z_0}^{(k)}$ defined on the space of $k$-forms $\Omega^k (M;\C)$. Equivalently, we have
$$
 C_V^k (z_0) = \Ker \left( (\Lie_V^{(k)} + z_0)^{m_k (z_0)} \right),
$$
and since $\pi_{z_0}^{(k)}$ has finite rank, the vectorial spaces $C_V^k (z_0)$ are \textbf{finite dimensional}. Recall that this space is independent of the choice of the escape function used to define Hilbert spaces.

Recall now that the sequence of De Rham complex is defined as
$$ 0 \overset{d}{\longrightarrow} \Omega^0 (M) \overset{d}{\longrightarrow} \Omega^1 (M) \overset{d}{\longrightarrow} \cdots \overset{d}{\longrightarrow} \Omega^n (M) \overset{d}{\longrightarrow}  0 $$

We will denote by $\Hg (C_V^* (0),d)$ the cohomology of the spectral complex associated with the eigenvalue $0$ and by $\Hg^k (M; \C)$ the complex $k$-th De Rham cohomology: 
$$ \Hg^k (C_V^* (0),d)  = \faktor{ \Ker (\rest{d}{C_V^k (0)}) }{\Ran (\rest{d}{C_V^{k-1} (0)}) } \mbox{ and } \Hg^k (M; \C)  = \faktor{\Ker (\rest{d}{\Omega^k (M;\C)}) }{ \Ran (\rest{d}{\Omega^{k-1} (M;\C)})}.$$

Theorem \ref{thm intro} of the introduction is a consequence of the next Theorem.
\vspace{0.2cm}

\fbox{
\begin{minipage}{0.9\linewidth}
\begin{thm}\label{thm isomorphism}
 Suppose that the vector field V generates an Axiom A flow which satisfies the transversality assumption (\ref{Transversality assumption - first}). For every $0 \leq k \leq n$, the map
 $$
  \pi_0^{(k)}: \Omega^k (M ;\C) \rightarrow C_V^k (0)
 $$
induces an isomorphism between $\Hg^k (M; \C)$ and $\Hg^k (C_V^* (0),d)$. 
\end{thm}
\end{minipage}}

\vspace{0.2cm}
Its proof is based on the following De Rham theorem:
\vspace{0.2cm}

\fbox{
\begin{minipage}{0.9\linewidth}
\begin{thm}[De Rham]\label{thm: De Rham} Let $u$ be an element in $\Hil_k^m$ satisfying $d u = 0$.
\begin{enumerate}
 \item There exists $\omega \in \Omega^k (M; \C)$ such that $u - \omega \in d\left( \Hil_{k-1}^{m+1} (M;\C) \right)$.
 \item If $u = d v$ with $u \in \Omega^k (M; \C)$ and $v \in \D'^{,k-1} (M; \C)$, then there exists $\omega \in \Omega^{k-1} (M; \C)$ such that 
 $$ u = d \omega .$$
\end{enumerate}
\end{thm}
\end{minipage}}

\begin{remark}
\begin{itemize}[align=left, leftmargin=*, noitemsep]
 \item We give a short proof of De Rham Theorem in appendix \ref{Appendix: De Rham} for the sake of completeness. We follow the lines of the proof presented in \cite[p. 16]{DRIII} and \cite[p. 355]{schwartz}.
\end{itemize}
\end{remark}

\label{section: proof of isomorphism theorem}

\begin{proof}[Proof of theorem \ref{thm isomorphism}]
Our goal is to prove that the following diagram is well defined, commutes and induces a map $\widetilde{\pi_0^{(k)}}$ which is an isomorphism from $\Hg^k (M; \C)$ to $\Hg^k (C_V^* (0),d)$:
\begin{center}
\begin{tikzpicture}
  \matrix (m) [matrix of math nodes,row sep=3em,column sep=4em,minimum width=2em]
  {
     \Ker (d) \cap \Omega^k (M; \C) & \Ker (d) \cap C_V^k (0) \\
     \Hg^k (M; \C) & \Hg^k (C_V^* (0),d) \\};
  \path[-stealth]
    (m-1-1) edge node [left] {$ p_1 $} (m-2-1)
    (m-1-1) edge node [above] {$ \pi_0^{(k)} $} (m-1-2)
    (m-2-1) edge node [above] {$\widetilde{\pi_0^{(k)}}$} (m-2-2)
    (m-1-2) edge node [left] {$ p_2 $} (m-2-2);
\end{tikzpicture}
\end{center}
Let us recall that the projector $\pi_0^{(k)}$ is defined on the anisotropic Sobolev space $\Hil_k^m$ which contains $\Omega^k (M; \C)$. Thus, it induces a map by restriction that we will still denoted by $\pi_0^{(k)}$. Also, since $d$ commutes with the projectors $\pi_0^{(*)}$, in the sense that $ d \circ \pi_0^{(k-1)} = \pi_0^{(k)} \circ d$, we get that $\pi_0^{(k)}$ sends the space $\Ker (d) \cap \Omega^k (M; \C)$ into the space $\Ker (d) \cap C_V^k (0)$. To prove the existence of the map $\widetilde{\pi_0^{(k)}}$ and in the same way its injectivity we will check the following equality:
\begin{equation}\label{proof thm isomorphism: equation kernel}
    \Ker (p_2 \circ \pi_0^{(k)} ) \cap \Ker (d) \cap \Omega^k (M; \C) = \Ran (d) \cap \Omega^k (M; \C).
\end{equation}
  Since we have $ d \circ \pi_0^{(k-1)} = \pi_0^{(k)} \circ d$, we can deduce the inclusion $\Ker (p_2 \circ \pi_0^{(k)} ) \supseteq \Ran (d) \cap \Omega^k (M; \C)$. Let us prove the other one. Consider $u \in \Omega^k (M; \C)$ such that $d u = 0$ and $p_2 \circ \pi_0^{(k)} (u) = 0$. By definition of $p_2$, there exists $v \in C_V^{k-1} (0)$ such that 
$$ \pi_0^{(k)} (u) = dv \quad \mbox{in } \: \D^{',k} (M; \C) .$$
In order to apply part $(2)$ of De Rham theorem (Thm \ref{thm: De Rham}), we will show that 
$$ \pi_0^{k} (u) \mbox{ and } u \mbox{ are in the same cohomology class (in } \D'^{,k}(M;\C) ) ,$$
and apply the theorem to $u$. But first, let us note that $\Lie_V^{(k)}$ is invertible in the Hilbert space 
$$ \Hil_k^m  \cap \Ker \pi_0^{(k)} = (\Id_{\Lambda^k T^* M \otimes \C} - \pi_0^{(k)}) \Hil_k^m $$
and for our $u \in \Omega^k (M; \C) \subset \Hil_k^m$ we have
\begin{equation}\label{eq: preuve thm - calcul representant lisse classe homologie}
 \begin{split}
   u &= \pi_0^{(k)} (u) + (\Id_{\Lambda^k T^* M \otimes \C} - \pi_0^{(k)}) (u) \\
 &= \pi_0^{(k)} (u) + \Lie_V^{(k)} \circ (\Lie_V^{(k)})^{-1}(\Id_{\Lambda^k T^* M \otimes \C} - \pi_0^{(k)}) (u) \\
 &= \pi_0^{(k)} (u) + (d \circ \iota_V + \iota_V \circ d) \circ (\Lie_V^{(k)})^{-1}(\Id_{\Lambda^k T^* M \otimes \C} - \pi_0^{(k)}) (u) \\
 &= \underbrace{\pi_0^{(k)} (u)}_{dv} + d \circ R^{(k)} (u) +  R^{(k)} \circ \underbrace {d u}_{0} \\
 &= d (v + R^{(k)} (u))  \in d \left( \D'^{,k-1}(M; \C) \right)
 \end{split}
\end{equation}
where
\begin{equation}\label{eq: proof isomorphism theorem - homotopic operator}
    R^{(k)} (u) = \iota_V \circ (\Lie_V^{(k)})^{-1}(\Id_{\Lambda^k T^* M \otimes \C} - \pi_0^{(k)}) (u).
\end{equation}
  We also used the Cartan formula $\Lie_V^{(k)} = d \circ \iota_V + \iota_V \circ d$ and the fact that $d$ commutes with $\Lie_V^{(k)}$ and therefore with $(\Lie_V^{(k)})^{-1}$ and $\pi_0^{(k)}$. Finally, we can apply the point $(2)$ of the De Rham theorem and we obtain the existence of $\omega \in \Omega^{k-1} (M; \C)$ such that 
$$ u = d \omega \in \Ran (d) \cap \Omega^{k-1} (M; \C).$$
So, equation (\ref{proof thm isomorphism: equation kernel}) holds and the universal property of the quotient gives us the existence of the linear map $\widetilde{\pi_0^{(k)}}$. In order to prove that it is an isomorphism, we need to show that it is onto.

\textbf{The map $\widetilde{\pi_0^{(k)}}$ is onto.} By definition of $C_V^k (0)$ the projector $\pi_0^{(k)}: \Hil_k^m \rightarrow C_V^k (0)$ is onto. Let us take $\widetilde{v} \in \Hg^k (C_V^* (0),d)$ and consider $v \in \Ker (d) \cap C_V^k (0)$ such that $\widetilde{v} = p_2 (v)$. Applying De Rham theorem to $v \in \Hil_k^m$, we obtain the existence of $\omega \in \Omega^k (M; \C) \cap \Ker (d)$ and $v' \in \Hil_{k-1}^{m +1}$ such that 
$$ v = \omega + dv' \mbox{ and } \widetilde{v} = p_2 (v) = p_2 (\omega) + \underbrace{p_2 (d v')}_{0} = p_2 (\omega).$$
Note that the expression $p_2 (d v')$ only makes sense for $dv' \in \Hil_k^m$.
Using a similar computation that in (\ref{eq: preuve thm - calcul representant lisse classe homologie}) on $\omega$, we deduce that
$$ v = \pi_0^{(k)} (\omega) + d( R^{(k)}(\omega) + v'). $$
Now, using the explicit expression of the operator $R^{(k)}$ given in (\ref{eq: proof isomorphism theorem - homotopic operator}) together with the fact that $(\Lie_V^{(k)})^{-1}(\Id_{\Lambda^k T^* M \otimes \C} - \pi_0^{(k)})$ is a pseudodifferential operator of order $-1$, we get that
$$ R^{(k)} (\omega) \in \Hil_{k-1}^{m +1}$$
and therefore 
$$ p_2 \circ \pi_0^{(k)} (\omega) = \widetilde{v}.$$
Finally, if we define $\widetilde{u} = p_1 (\omega) \in \Hg^k (M; \C)$ then we obtain
$$ \widetilde{\pi_0^{(k)}} (\widetilde{u} ) = \widetilde{v} $$
and thus the surjectivity of $\widetilde{\pi_0^{(k)}}$.
\end{proof}
  
 \section{Energy functions and application to Axiom A flows}\label{section: Dynamical proofs}

  \subsubsection{A useful lemma}\label{subsubsection: a useful lemma} In this part, we recall a general analysis introduced in \cite{FS} which will be applied to both flows $\varphi^t: M \rightarrow M$ and $\widetilde{\Phi}^t: S^* M \rightarrow S^* M$.
  
  Let us consider $v \in \Gamma (TX)$ some smooth vector field on a compact manifold $X$ and let us denote by $\exp (t.v)$ the flow generated by $v$. A couple of compact sets $(K_+,K_-)$ is said to be \textbf{attractor-repeller} for the flow $\exp (t.v)$ on $X$ if it satisfies the two following conditions:
  	\begin{enumerate}[label = \roman*)]
  	 \item $\forall x \in X \setminus (K_- \cup K_+)$,  $d( \exp (\pm t. v) (x), K_{\pm}) \underset{t \to \pm \infty}{\longrightarrow} 0$.
  	 \item There exists open neighborhoods $\Vcal_{\pm}$ of $K_{\pm}$ stable by $\exp (\pm 1. v)$ such that  $\overline{\Vcal_{-}} \cap \overline{\Vcal_{+}} = \emptyset$.
  	\end{enumerate}
  	
  	\begin{figure}[!t]
\centering
\def\svgscale{0.7}
 \executeiffilenewer{attractor-repeller.svg}{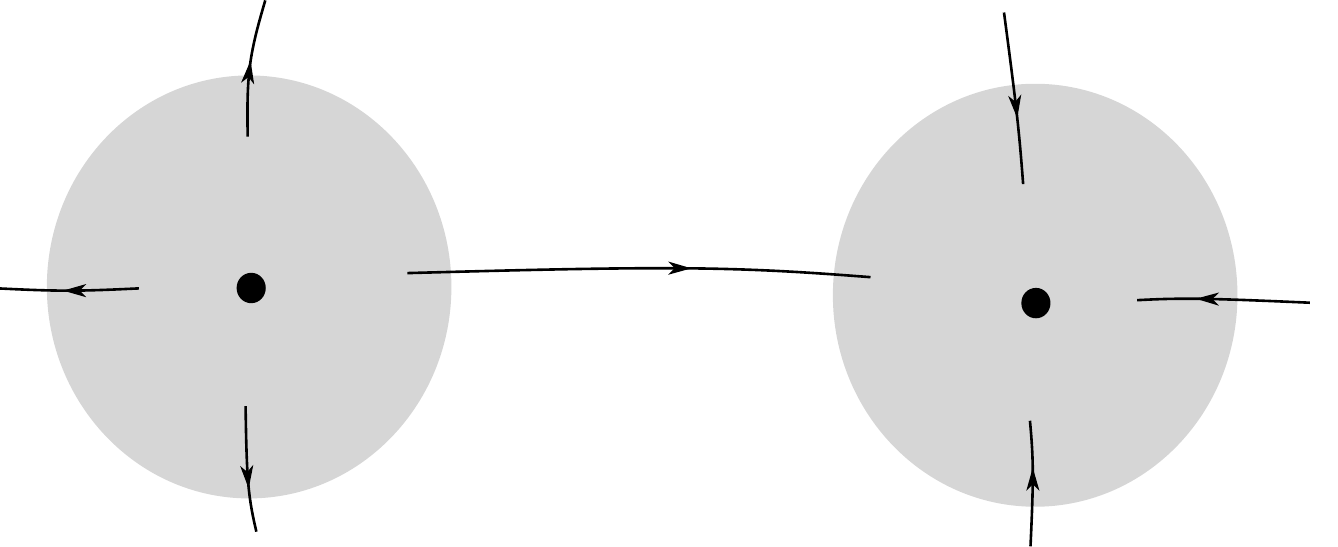}%
 {inkscape -z -D --file=attractor-repeller.svg %
 --export-pdf=attractor-repeller.pdf --export-latex}%
\begingroup%
  \makeatletter%
  \providecommand\color[2][]{%
    \errmessage{(Inkscape) Color is used for the text in Inkscape, but the package 'color.sty' is not loaded}%
    \renewcommand\color[2][]{}%
  }%
  \providecommand\transparent[1]{%
    \errmessage{(Inkscape) Transparency is used (non-zero) for the text in Inkscape, but the package 'transparent.sty' is not loaded}%
    \renewcommand\transparent[1]{}%
  }%
  \providecommand\rotatebox[2]{#2}%
  \newcommand*\fsize{\dimexpr\f@size pt\relax}%
  \newcommand*\lineheight[1]{\fontsize{\fsize}{#1\fsize}\selectfont}%
  \ifx\svgwidth\undefined%
    \setlength{\unitlength}{385.2276371bp}%
    \ifx\svgscale\undefined%
      \relax%
    \else%
      \setlength{\unitlength}{\unitlength * \real{\svgscale}}%
    \fi%
  \else%
    \setlength{\unitlength}{\svgwidth}%
  \fi%
  \global\let\svgwidth\undefined%
  \global\let\svgscale\undefined%
  \makeatother%
  \begin{picture}(1,0.40845389)%
    \lineheight{1}%
    \setlength\tabcolsep{0pt}%
    \put(0,0){\includegraphics[width=\unitlength,page=1]{attractor-repeller.pdf}}%
    \put(0.19249616,0.22371105){\makebox(0,0)[lt]{\lineheight{1.25}\smash{\begin{tabular}[t]{l}$K_-$\end{tabular}}}}%
    \put(0.77991422,0.20903732){\makebox(0,0)[lt]{\lineheight{1.25}\smash{\begin{tabular}[t]{l}$K_+$\end{tabular}}}}%
    \put(0.07065287,0.24345724){\makebox(0,0)[lt]{\lineheight{1.25}\smash{\begin{tabular}[t]{l}$\mathcal{V}_-$\end{tabular}}}}%
    \put(0.65425376,0.23798236){\makebox(0,0)[lt]{\lineheight{1.25}\smash{\begin{tabular}[t]{l}$\mathcal{V}_+$\end{tabular}}}}%
  \end{picture}%
\endgroup%

  \caption{Illustration of an attractor-repeller system.}
  \label{fig: attractor-repeller model} 
\end{figure}

 \begin{remark}\label{remark: uniform bound on time spent outside unrevisited neighborhoods}
  From conditions i) and ii) we can deduce that the time spent outside $\Vcal_- \cup \Vcal_+$ is uniformly bounded: if we define the map
  $$ \tau (x) = \Card \{ m \in \Z: \: \varphi^m (x) \notin \Vcal_- \cup \Vcal_+ \}, $$
  then we get the existence of an integer $T > 0$ such that $\tau (x) \leq T$ for every $x \in M$. Indeed, by contradiction, if the map was not bounded then we could find (by compactness) a point $x \in X \notin (\Vcal_- \cup \Vcal_+)$ which satisfies $\varphi^m (x) \in X \notin (\Vcal_- \cup \Vcal_+)$ for all $m \in \N$. It is in contradiction with condition i).
 \end{remark}

\begin{lemma}[{Faure-Sjöstrand, \cite{FS}}]\label{lemma: energy function for attractor-repeller model}
 Let $(K_-,K_+)$ be an attractor-repeller couple for the flow $\exp (t.v)$ on the compact manifold $X$ and fix $\varepsilon > 0$. There exist $\varepsilon$-neighborhoods $\Wcal_{\pm} \subset \Vcal_{\pm}$ of $K_{\pm}$, an energy function $m \in \Ccal^{\infty} (X; [0,1])$ and a constant $\eta > 0$, which depends on $\varepsilon$, such that $v (m) \geq 0$ everywhere and $v (m) > \eta$ outside $\Wcal_{-} \cup \Wcal_{+}$. Moreover, we have $m > 1 - \varepsilon$ on $\Wcal_{+}$ and $m = 1$ on $K_+$. Similarly, we have $m < \varepsilon$ on $\Wcal_{-}$ and $m = 0$ on $K_-$. 
\end{lemma}

This lemma proves similarly to Lemma $2.1$ of Faure-Sjöstrand \cite[p. 336]{FS}. For the sake of completeness, let us recall the main lines of the proof:
\begin{itemize}
 \item First, we consider $T \geq 0$ large enough so that $\exp (T.v) \left( X \setminus \Vcal_{-} \right) \subseteq \Vcal_+$, \\
 $\exp (-T.v) \left( X \setminus \Vcal_{+} \right) \subseteq \Vcal_-$ and we define
  $$\Wcal_- = X \setminus \exp (-T.v)(\Vcal_+) \quad \mbox{and} \quad \Wcal_+ = X \setminus \exp (T.v)(\Vcal_-). $$
 \item Then, we fix a map $\phi \in \Ccal^{\infty} (X)$ which has value $1$ on $\Vcal_{-}$ and $0$ on $\Vcal_{+}$ and we define the \textit{energy function} $m \in \Ccal^{\infty} (X)$ by 
  $$ m(x):= \frac{1}{2T} \int_{-T}^T \phi (\exp (t.v) (x)) dt,$$
 \item Finally, by noticing that $\Ical (x):= \{ t \in \R: \: \exp (t.v). x \in X \setminus (\Vcal_- \cup \Vcal_+) \}$
 is uniformly bounded in the sense that 
 $$ \exists \tau > 0, \:  \forall x \in X, \: \: |\max (\Ical (x)) - \min (\Ical (x))| < \tau, $$ one can deduce the estimates on $m$ for $T$ sufficiently large (which depends on $\varepsilon$, $\tau$ and the size of $\Vcal_{\pm}$). We refer to \cite{FS} for complementary details. 
\end{itemize}

 \subsection{Energy functions for Axiom A flows} 
  
In this part, we explain how to construct an energy function for the flow $\varphi^t$ from the previous lemma. More precisely, we will prove Proposition \ref{prop energy function for Axiom A flows}. We will use strongly the order relation property, thus the transversality assumption. \textbf{From now on}, we consider a total ordering of the basic set in the sense of \S\ref{subsubsection total order relation}. The construction of an energy function presented here splits in two parts:
\begin{itemize}
 \item First, we prove that for every $1 < j \leq N$ the couple $(\cup_{k \geq j}W^u (K_k), \cup_{i < j}W^s (K_i))$ is an attractor-repeller couple. 
 \item Then, as a consequence of Lemma \ref{lemma: energy function for attractor-repeller model} we obtain a family of energy functions $E_i$ and a linear combination of them gives the global energy function for $\varphi^t$.
\end{itemize}

\begin{lemma}[Invariant neighborhoods on the base]\label{lemma: attractor and repeller for the flow on the basis}
 For every $1 < j \leq N$, $\cup_{i < j}W^s (K_i)$ and $\cup_{k \geq j}W^u (K_k)$ are disjoint invariant compact sets such that:
 \begin{equation}\label{eq : lemma invariant neigh. on the base - eq 1}
  \forall x \notin \bigcup_{i < j}W^s (K_i), \quad d (\varphi^t (x), \cup_{k \geq j}W^u (K_k)) \underset{t \to+\infty}{\longrightarrow} 0
 \end{equation}
 and
 \begin{equation}\label{eq : lemma invariant neigh. on the base - eq 2}
   \forall x \notin \bigcup_{k \geq j}W^u (K_k), \quad d (\varphi^{-t} (x),\cup_{i < j}W^s (K_i)) \underset{t \to+\infty}{\longrightarrow} 0 .
 \end{equation}
\end{lemma}

\begin{proof}
 The fact that these sets are disjoint and compact is a direct consequence of the order relation's properties. Now, consider $x \notin \bigcup_{i < j}W^s (K_i)$. From the decomposition (\ref{partition of the manifold in stable manifolds}) of $M$ into unstable manifolds of the basic sets, there exists $1 \leq k \leq N$ such that $x \in W^u( K_k)$. Thanks to our choice of a total order relation in the sence of \S \ref{subsubsection total order relation}, we have necessarily $k \geq j$. It proves the convergence (\ref{eq : lemma invariant neigh. on the base - eq 1}). Up to replacing the flow $\varphi^t$ by $\varphi^{-t}$, we also get the convergence (\ref{eq : lemma invariant neigh. on the base - eq 2}).  
\end{proof}

Now the key point is to prove the second property of an attractor-repeller. This is given by the next lemma which is slightly more precise.

\begin{lemma}\label{lemma: equivalence filtration and unrevisited}
	For every $\varepsilon >0$, there exists a filtration $(\Ocal_{\ell}^-)_{1\leq \ell \leq N}$ on $M$ for $\varphi^{-1}$ which is within a distance $\varepsilon$ of the stable manifolds, i.e. $\forall  j \in [\![1,N]\!]$,
	 $$
		\sup_{y \in \Ocal_j^-} d_g \left( y, \bigcup_{i \leq j} W^s (K_i) \right) < \varepsilon .
	 $$
	 Similarly, there exists a filtration $(\Ocal_{\ell}^+)_{1\leq \ell \leq N}$ on $M$ for $\varphi^{1}$ which is within a distance $\varepsilon$ of the unstable manifolds, i.e. $\forall  j \in [\![1,N]\!]$,
	 $$
		\sup_{y \in \Ocal_{N-j+1}^+} d_g \left(y, \bigcup_{k \geq j} W^u (K_k)\right) < \varepsilon. 
	 $$
\end{lemma}

Note that the filtration $(\Ocal_{\ell}^+)_{1 \leq \ell \leq N}$ is not obtained by taking the complementary of $(\Ocal_{\ell}^-)_{1 \leq \ell \leq N}$ but by exchanging the sense of times.

\begin{proof}
Let us proceed by induction to construct a filtration on $M$ stable by $\varphi^{1}$. The following arguments adapt easily to construct a filtration $\varphi^1$-stable if we change $\varphi^{-1}$ by $\varphi^1$.

\textbf{Base case (construction of $\Ocal_1^+$).} Fix $\varepsilon > 0$. Since the indices of $(K_i)_{1 \leq i \leq N}$ have been chosen compatible with the relation $\leq$, the basic set $K_{N}$ must be an attractor. So if we consider some unrevisited neighborhood $\Vcal_{N}$ of $K_N$ such that $\overline{\Vcal_N} \cap \Omega = K_N$, then $\Vcal_N$ is $\varphi^1$-stable and we will simply define $$\Ocal_{1}^+:= \Vcal_N \cap \varphi^{n}(\Vcal_{N})$$
for a large enough value of $n$. Indeed, thanks to Lemma \ref{lemme: cv nonrevisite vers variete stable et instable} we can choose $n$ so that $\sup_{y \in \Vcal_N \cap \varphi^{n}(\Vcal_{N})} d_g (y, W^u (K_{N}) \cap \Vcal_N) < \varepsilon $. Since $K_N$ is an attractor, the equality $W^u (K_N) = K_N$ holds and it implies 
$$\sup_{y \in \Ocal_{1}^+} d_g (y,  K_N) < \varepsilon .$$

\textbf{Induction step.} Fix $\varepsilon > 0$ and assume that the open sets $\Ocal_1^+ \subseteq \Ocal_2^+ \subseteq \cdots \subseteq \Ocal_{j-1}^+ \subseteq \Ocal_j^+$ to be constructed for some $ 1 \leq j \leq N$ so that it defines a filtration for the family of basic sets $(K_{N-i})_{0 \leq i \leq j-1}$ which is $\varepsilon$-close to the unstable manifolds, in the sense that, for every $1 \leq i \leq j$,
$$
 \sup_{y \in \Ocal_i^+} d_g \left( y, \bigcup_{k \geq N- i +1} W^u (K_{k}) \right) < \varepsilon .
$$
	We want to construct a $\varphi^1$-stable neighborhood $\Ocal_{j+1}^+$ of $\bigcup_{k \geq N-j} W^u (K_{k})$. To lighten the proof, let us denote by $K$ the basic set $K_{N-j}$. If $K$ is an attractor, then we can proceed exactly as in the base case and take the union of this open set with $\mathcal{O}_{j}^+$. So let us assume that $K$ is not a attractor, i.e. $W^u (K) \neq K$, and take some small unrevisited neighborhood $\Vcal$ of $K$ in the sense that $\overline{\Vcal} \cap  \Omega = K$. Let us define the following ``annulus'' of $K$ for every $m \in \N$:
	$$\Acal (m):= (\varphi^{m}(\Vcal) \cap \Vcal) \setminus (\varphi^{m}(\Vcal) \cap \varphi^{-1} (\Vcal)). $$
	 We refer the reader to the figure \ref{fig: unrevisited neighborhoods near a basic set}. Thanks to Lemma \ref{lemme: cv nonrevisite vers variete stable et instable}, the decreasing sequence  $\varphi^{m} (\Vcal) \cap \Vcal$ of unrevisited neighborhoods converges uniformly to $W^u (K) \cap \overline{\Vcal}$ and therefore 
	 \begin{equation}\label{eq : lemma invariant neighborhoods on the base - convergence of annulus}
	  \Acal (m)  \mbox{ converges uniformly to } W^u (K) \cap \left( \overline{\Vcal} \setminus \varphi^{-1} (\Vcal) \right)
	 \end{equation}
    as $m$ tends to $+\infty$. 
	
	Before proceeding to the construction of $\Ocal_{j+1}^+$, we need to prove a few properties:  (1) show that for every $m \in \N$ we have $\Acal(m) \neq \emptyset$ and $W^u (K) \cap \Acal(m) = W^u (K) \cap \Acal(0)$, (2) find for every  $x$ in $W^u (K) \cap \Acal(0)$ an integer $m_0 > 0$ such that we have $\varphi^{m_0}(x) \in \Ocal_j^{+}$, (3) prove that $m_0$ can be chosen uniformly in $x \in W^u (K) \cap \Acal(0)$ and (4) prove that $\varphi^{m_0} (x)$ belongs to $\Ocal_j^+$ for all $x$ in $\Acal (m)$ with $m$ large enough.
	\begin{enumerate}[label=(\arabic*)]
	 \item Since the relations $W^u (K) \cap \varphi^{m} (\Vcal) \cap \Vcal = W^u (K) \cap \Vcal$ and $W^u (K) \cap \varphi^{m} (\Vcal) \cap \varphi^{-1} (\Vcal) = W^u (K) \cap \Vcal \cap  \varphi^{-1} (\Vcal)$ hold from (\ref{eq: unrevisited and unstable manifold}), we must have
	$$W^u (K)\cap \Acal (m) = W^u (K) \cap \Acal (0) = \{ x \in W^u (K) \cap \Vcal, \: \varphi^{1} (x) \notin \Vcal \}$$
	for all $m \in \N$. Let us prove that the last set is non-empty. If $x$ belongs to $W^u (K) \cap \Vcal \setminus K \neq \emptyset$ as $K$ is not an attractor, then one can find an integer $k_0 \geq 0$ such that $\varphi^{k_0}(x) \in W^u (K) \cap \Vcal$ and $\varphi^{k_0 + 1}(x) \notin \Vcal$ by fixing $k_0 = \sup \{ k \in \N, \: \varphi^{k} (x) \in \Vcal \}$ which is finite thanks to the definition of $x$.
	 \item A straightforward application of Lemma \ref{lemma: attractor and repeller for the flow on the basis} gives us the result:
	\begin{equation}\label{eq : lemma filtration on the base - annulus on unstable manifold}
            \forall x \in W^u (K) \cap \overline{\Acal (0)}, \: \exists m_0 \in \N, \mbox{ s.t } \varphi^{m_0}(x) \in \Ocal_j^+ .
	\end{equation}
	 \item Let us prove that there exists $m_0 \in \N$ such that for every $x \in W^u (K) \cap \Acal (0)$ we have $\varphi^{m_0} (x) \in \Ocal_j^+$. By contradiction, let us assume that for every integer $m \in \N^*$ there exists an element $x_m \in W^u (K) \cap \Acal (0)$ such that $\varphi^{m} (x_m) \notin \Ocal_j^+$. By compactness of $\overline{\Vcal}$, we can extract a subsequence $(x_{m_k})_{k \in \N}$ which converges to an element $x_{\infty} \in \overline{\Vcal}$. According to Lemma \ref{lemma: clos. of stable man on unrevisited neigh.}, we must have $x_{\infty} \in W^u (K) \cap \overline{\Vcal}$. Also, by definition of $\Acal (0)$, the elements $\varphi^1 (x_{m_k})$ belong to $M \setminus \Vcal$. So, letting $k$ tends to $+\infty$, we deduce that $\varphi^1 (x_{\infty}) \in M \setminus \Vcal $. Therefore, we have $x_{\infty} \in W^u (K) \cap \overline{\Acal (0)}$ and the step (2) implies that 
	 $$\varphi^{m_0} (x_{\infty}) \in \Ocal_j^+ .$$ However, by definition of our sequence $(x_{m_k})_k$ and due to the stability of $\Ocal_j^+$, we must have
	$  \varphi^{m_0} (x_{n_k}) \notin \Ocal_j^+$ for $k$ sufficiently large.
	Letting $k$ tends to $+\infty$ in the last relation gives a contradiction with the fact that $\varphi^{m_0} (x_{\infty}) \in \Ocal_j^+$. Therefore, the integer $m_0$ can be chosen uniformly with respect to $x$ in $W^u (K) \cap \Acal (0)$.
	 \item Since $\Acal (m)$ converges uniformly to $W^u (K) \cap \left( \overline{\Vcal} \setminus \varphi^{-1} (\Vcal) \right)$ as $m$ tends to $+\infty$ in the sense of Lemma \ref{lemme: cv nonrevisite vers variete stable et instable}, we deduce by continuity the following statement: there exists $m_1 \in \N$ such that 
	\begin{equation}\label{eq: lemme filtration cv uniform R(n)}
	\forall m \geq m_1, \: \forall x \in \Acal (m), \: \varphi^{ m_0} (x) \in \Ocal_j^+.
	\end{equation}
	\end{enumerate}
	Finally, for every $m \geq m_1$, we define
	$$
		\Ocal_{j+1}^+ (m):= \Ocal_j^+ \cup \bigcup_{k = 0}^{m_0 - 1} \varphi^{k} (\varphi^{m} (\Vcal) \cap \Vcal).
	$$
	It remains to prove that $\Ocal_{j+1}^+ (m)$ is $\varphi^1$-stable for every choice of $m \geq m_1$. Let us consider $y \in  \varphi^{k} (\varphi^{m} (\Vcal) \cap \Vcal)$ for some $0 \leq k \leq m_0 - 1$. By definition, there exists $x \in \varphi^{m} (\Vcal) \cap \Vcal$ such that $y = \varphi^k (x)$. We have two cases to deal with.
	\begin{itemize}
	 \item \textbf{1\ts{st} case:} $x \in \varphi^{m} (\Vcal) \cap \varphi^{-1} (\Vcal)$. Then, we have
	 $$\varphi^1 (x) \in  \varphi^{m+ 1} (\Vcal) \cap \Vcal \subseteq  \varphi^{m} (\Vcal) \cap \Vcal.$$
	 \item \textbf{2\ts{nd} case:} $x \in  (\varphi^{m} (\Vcal) \cap \Vcal ) \setminus (\varphi^{m} (\Vcal) \cap \varphi^{-1} (\Vcal))  = \Acal (m)$. Since $x \in \varphi^{m} (\Vcal) \cap \Vcal$ we get $\varphi^{k}(x) \in \varphi^{k}(\varphi^{m} (\Vcal) \cap \Vcal)$ for every $k \in [\![0, m_0 -1 ]\!]$ and the statement (\ref{eq: lemme filtration cv uniform R(n)}) gives us $\varphi^{m_0}(x) \in \Ocal_i^+$.
	 \end{itemize}
	 Finally, we get that for every $m \geq m_1$ and every $\ell \geq 0$, we have $\varphi^{\ell} (x) \in \Ocal_{j+1}^+ (m)$. In particular, we deduce that for every $\ell \geq 0$, we have $\varphi^{\ell} (y) \in \Ocal_{j+1}^+ (m)$. So, $\Ocal_{j+1}^+ (m)$ is $\varphi^1$-stable. Now, if we choose $m$ sufficiently large so that $ d_g (\cup_{k = 0}^{m_0 - 1} \varphi^{k} (\varphi^{m} (\Vcal) \cap \Vcal) ,  W^u (K)) < \varepsilon$, then it ensures that 
	 $$\sup_{y \in \Ocal_{j+1}^+} d_g \left( y, \bigcup_{k \geq N-j} W^u (K_k) \right) < \varepsilon.$$
	 It ends the induction and the proof.
\end{proof}

\begin{remark}
 Thanks to the total order relation (\ref{eq: total order relation, indices compatible}), the compact sets $\cup_{i \leq j} W^s (K_i)$ and $\cup_{k \geq j}W^u (K_k)$ intersect on $K_j$, i.e.
$$\bigcup_{i \leq j}W^s (K_i) \cap \bigcup_{k \geq j}W^u (K_k) = K_j. $$
Recall from Remark \ref{rk : filtration} that the set $\Vcal_{j}:= \Ocal_j^- \cap \Ocal_{N-j+1}^+$ is an unrevisited neighborhood at distance (at most) $\varepsilon$ of $K_j$.
\end{remark}

A direct application of these lemmas gives what we were looking for. 

\begin{prop}\label{prop: attractor reppeler couple on the base}
 For every $1 < j \leq N$, $(\cup_{k \geq j} W^u (K_k), \cup_{i < j}W^s (K_i))$ defines an attractor-repeller couple.
\end{prop}

\begin{proof}It is a direct application of Lemmas \ref{lemma: attractor and repeller for the flow on the basis} and \ref{lemma: equivalence filtration and unrevisited} once we have chosen $\varepsilon\ll 1$ small enough to ensure that
$$ B_g \left( \cup_{k \geq j} W^u (K_k), 2\varepsilon \right) \bigcap B_g \left( \cup_{i < j} W^s (K_i), 2\varepsilon \right) = \emptyset, \quad \forall 1 < j \leq N,$$
where $B_g(K, 2\varepsilon)$ denotes the geodesic ball at distance $2\varepsilon$ to the compact set $K$.
\end{proof}

Now, we are ready to construct an energy function for $\varphi^t$. 

\begin{proof}[Proof of proposition \ref{prop energy function for Axiom A flows}]
 Let $\varepsilon > 0$ as in the previous proof and fix a sequence of pairwise distinct real numbers $(\lambda_i)_{1 \leq i \leq N}$ compatible with the graph structure in the sense that $\lambda_i \leq \lambda_j \Longleftrightarrow K_i \leq K_j$. Up to a permutation of the indices of the basic sets, we can assume that $\lambda_1 < \lambda_2 < \cdots < \lambda_N$. Indeed, there exists a one to one map $\sigma: [\![1, N ]\!] \rightarrow [\![1, N ]\!]$ such that $\lambda_{\sigma^{-1} (1)} < \lambda_{\sigma^{-1} (2)} < \cdots < \lambda_{\sigma^{-1} (N)}$. The map $\sigma$ is given by
 $$\sigma (j) = \Card \left\{ i \in [\![1, N ]\!], \: \: \lambda_i \leq \lambda_j \right\}.$$
 If we rename $K_{\sigma^{-1} (i)}$ by $K_i$ and $\lambda_{\sigma^{-1}(i)}$ by $\lambda_i$, then we obtain a total order relation on the basic sets given by the usual order relation on $[\![1, N ]\!]$. 
 
  In order to find an energy function $E$ such that $E = \lambda_j$ on $K_j$, we will apply Lemma \ref{lemma: energy function for attractor-repeller model} for each attractor-repeller given in Proposition \ref{prop: attractor reppeler couple on the base}. Thanks to Lemma \ref{lemma: equivalence filtration and unrevisited}, there exists filtrations (which depend on $\varepsilon$) $\Ocal_0^- \subset \Ocal_1^- \subset \cdots \subset \Ocal_{N}^-$ and $\Ocal_N^+ \supset \cdots \supset \Ocal_{1}^+ \supset \Ocal_0^+$ which are respectively $\varphi^{-1}$-stable and $\varphi^{1}$-stable. Thanks to Lemma \ref{lemma: energy function for attractor-repeller model}, we obtain for every $1 < j \leq N$ a smooth energy function $E_{j} \in \Ccal^{\infty} (M; [0,1])$, $\varepsilon$-neighborhoods $\Wcal_j^- \subset \Ocal_j^-$ and $\Wcal_j^+ \subset \Ocal_j^+$ of $ \bigcup_{i < j} W^s (K_i)$ and $ \bigcup_{k \geq j} W^u (K_k)$ respectively, a constant $\eta_0 > 0$ (which only depends on $\varepsilon$) such that $\Lie_V (E_j) \geq 0$ on $M$ and
  \begin{itemize}
   \item $\Lie_V (E_j) > \eta_0$ on $M\setminus (\Wcal_j^- \cup \Wcal_j^+)$,
   \item $E_j < \varepsilon$ on $\Wcal_j^-$ and $E_j = 0$ on $\bigcup_{i < j} W^s (K_i)$. In particular, we have $E_j = 0$ on $\bigcup_{i < j} K_i$.
   \item $E_j > 1 - \varepsilon$ on $\Wcal_j^+$ and $E_j = 1$ on $\bigcup_{k \geq j} W^u (K_k)$. In particular, we have $E_j = 1$ $\bigcup_{k \geq j} K_k$.
  \end{itemize}
 We define a global energy function $E \in \Ccal^{\infty} (M)$ as a linear combinaison of previous energy functions:
 $$
	\boxed{E = \lambda_1  + \sum_{j = 2}^N (\lambda_{j} - \lambda_{j-1}) E_{j}.}
 $$
 Thanks to the analysis of Lemma \ref{lemma: equivalence filtration and unrevisited}, we deduce that 
 $$ \Lie_V (E)  = \sum_{j = 1}^N (\lambda_j - \lambda_{j-1}) \Lie_V (E_j) > \min_{1 < j \leq N}(\lambda_j - \lambda_{j-1}) \eta_0 =: \eta \:  \mbox{ on } \: \left( \bigcap_{j =2}^N (\Wcal_j^- \cup \Wcal_j^+) \right)^c .$$ 
 It remains to proof that $\bigcap_{j = 2}^N (\Wcal_j^- \cup \Wcal_j^+)$ is an $\varepsilon$-neighborhood of the nonwandering set and that $E$ is close to $\lambda_i$ near $K_i$ with equality on $K_i$. One has: 
 $$
	\bigcap_{j = 2}^N (\Wcal_j^- \cup \Wcal_j^+) = \bigcup_{\tau: [\![1,N]\!] \rightarrow \{\pm\}} \bigcap_{j = 2}^N \Wcal_j^{\tau (j)} = \bigcup_{j = 1}^N \Ncal_j:= \bigcup_{j = 1}^N \left( \Wcal_{2}^+ \cap \cdots \cap \Wcal_{j}^+ \cap \Wcal_{j+1}^- \cap \cdots \cap \Wcal_{N}^- \right) ,
 $$
 using the convention $\Ncal_1 = \Wcal_{2}^- \cap \cdots \cap \Wcal_{N}^-$.
 Note that the second equality is a direct consequence of the next fact which holds for $\varepsilon \ll 1$:
 $$
	\mbox{if } 2 \leq i < j \mbox{ and } (\tau (i), \tau (j)) = (-,+) \mbox{ then }  \bigcap_{j = 2}^N \Wcal_j^{\tau (j)} = \emptyset.
 $$
Moreover, we have on each $K_i$:
$$ E = \lambda_1 + \sum_{j = 2}^N (\lambda_j - \lambda_{j-1}) E_j = \lambda_1 + \sum_{j = 2}^N (\lambda_j - \lambda_{j-1}) \delta_{j \leq i} = \lambda_i .$$
It remains to prove that $E$ is close to $\lambda_i$ on the neighborhood $\Ncal_i$ of $K_i$. To that aim, let us note that for every $x \in M$ and for every $1 \leq i \leq N$, we have 
 $$
	|E(x) - \lambda_i | \leq  \sum_{j = 2}^N (\lambda_j - \lambda_{j-1}) |E_j (x) - \delta_{j \leq i} |.
 $$
 For any $1 < j \leq N$ and $x \in \Ncal_i$ we will bound $|E_j (x) - \delta_{j \leq i}|$ by a small quantity independent of $x$ in $\Ncal_i$. We have two cases to deal with: 
\begin{itemize}
 \item If $i < j$, then $\delta_{j \leq i} = 0$ and $|E_j (x) - 0| = E_j (x) < \varepsilon$ by definition of $E_j$.
 \item If $i \geq j$, then $\delta_{j \leq i} = 1$ and $|E_j (x) - 1| = 1 - E_j (x) < \varepsilon$ again by definition of $E_j$. 
\end{itemize}
Finally, we obtain the upper bound
$$
      \sup_{1 \leq i \leq N} \sup_{x \in \Ncal_i} |E (x) - \lambda_i| \leq \varepsilon \sum_{j = 2}^N (\lambda_j - \lambda_{j-1}) \leq  (\lambda_N - \lambda_1) \varepsilon.
$$
 \end{proof}
  
  Now, the idea will be to perform the same analysis for the Hamiltonian flow $\widetilde{\Phi}^t$ acting on $S^* M$. However, in that case, proving that one has an attractor-reppeler structure for ($\Sigma_{\cdots}$) reveals to be more challenging because we need to understand what happens to the fiber part of $\widetilde{\Phi}^t (x,\xi)$ when the orbit of $(x,\xi)$ comes close to a basic set. The next part is devoted to the local analysis of the Hamiltonian near basic sets in view of applications to the proofs of Proposition \ref{prop energy function for Axiom A flows}, Lemmas \ref{lemma: attractor and repeller for the flow on the basis} and \ref{lemma: equivalence filtration and unrevisited} and finally the existence of the energy function on $S^*M$.
  
  \section{Compactness result and energy functions for the Hamiltonian flow}\label{section : Compactness result and energy functions for the Hamiltonian flow}
  
  On a basic set $K$, one can define conical neighborhoods of the unstable distributions $E_u^*$ and $E_{uo}^*$ which are stable under the Hamiltonian flow $\Phi^t$ as soon as $t > 0$, under some assumptions on the conical neighborhood. In particular, they are $\Phi^1$ stable and this property should extend ``by continuity'' to a small neighborhood of $K$. In this part, we will make sense of the term ``by continuity''.  
  
  \subsection{Adapted metric on a basic set}\label{subsubsection adapted metric}

From the fixed Riemannian metric $g$ on $M$, we can define on any basic set $K$ the following new metric called \textbf{adapted metric} for the flow $\varphi^t$. More precisely, for every $x \in K$ and every $v = v_s + v_u + v_o \in E_s (x) \oplus E_u (x) \oplus E_o (x) = T_x M$, we set
\begin{equation}\label{eq : def extended metric near basic set}
	\begin{split}
	\hat{g} (v,v) &=  \hat{g}(v_s, v_s) + \hat{g}(v_u, v_u) + \hat{g}(v_o, v_o) \\
		&:= \int_0^{+\infty} e^{\lambda t /2} (\varphi^{t*} g) (v_s,v_s) dt + \int_0^{+\infty} e^{\lambda t /2} (\varphi^{-t*} g) (v_u,v_u) dt + g(v_o,v_o),
	\end{split}
\end{equation}
where $\lambda$ denotes the hyperbolic exponent on the basic set $K$, see Appendix \ref{Appendix: hyperbolic set}.
This new metric is well defined thanks to hyperbolicity and to the invariance properties of the vector bundles on $K$. It also depends continuously on the point $x \in K$, even if the vector bundles depend smoothly on the point $x$. Recall also that the distributions $E_s, E_u$ are only Hölder continuous in general. Precisely, $\hat{g}$ will be seen as a continuous section of the vector bundle of metrics (i.e. symmetric $(2,0)$-tensors) $\otimes_{sym}^2 T M \rightarrow M$ defined on the compact set $K$. Let us denote by $|.|_{\hat{g}}$ the norm induced by $\hat{g}$, i.e. $|v|_{\hat{g}}:= \sqrt{\hat{g}(v,v)}$. The metric is said to be \textit{adapted} due to the following hyperbolic estimates: for every $x \in K$,
\begin{equation}\label{hyperbolic estimates for the adapted metric}
    \begin{split}
	 |D\varphi^t (x) v_s |_{\hat{g}} &\leq e^{-\lambda t /2} |v_s |_{\hat{g}}, \qquad \forall t  \geq 0, \: \forall v_s \in E_s (x) \\
	 |D\varphi^{-t} (x) v_u |_{\hat{g}} &\leq e^{-\lambda t /2} |v_u |_{\hat{g}}, \qquad \forall t  \geq 0, \: \forall v_u \in E_u (x) \\
	 |D\varphi^t (x) v_o |_{\hat{g}} &= |v_o |_{\hat{g}}, \qquad \forall t \in \R, \:\forall v_o \in E_o (x).
    \end{split}
\end{equation}

\subsection{Extension of the invariant distributions near basic sets.}

In order to analyse the dynamics near basic sets, it will be convenient to extend the previous hyperbolic estimates in some neighborhood of $K$. To that aim, we need to extend the distributions $E_s, E_u$, $E_o$ and $|.|_{\hat{g}}$ near each basic set. This can be achieved thanks to the following lemma:

\begin{lemma}[{\textbf{Extension lemma}, \cite[lem. 4.4 p. 128]{HPPS}}]\label{Lemme geometrique prolongement des sections} Let $X$ be a smooth manifold and let $\pi: E \rightarrow X$ be some vector bundle over $X$. If $s: K \rightarrow E$ denotes a continuous section defined on a compact set $K \subseteq M$, then $s$ extends as a continuous section $\overline{s}: \Ncal \rightarrow E$ on a neighborhood $\Ncal$ of $K$.  
 \end{lemma}
 
 \subsubsection{Extension of the adapted metric near a basic set} Let us apply this Lemma with
 $$ X:= M \quad \mbox{and} \quad E = T^* M \otimes_{sym} T^* M.$$
 The Riemannian metric $\hat{g}$ can be seen as a continuous section $\hat{g}:  K \rightarrow E$. Since the basic set $K$ is a compact subset of $M$, the lemma applies and it allows to extend $\hat{g}$ continuously on an open neighborhood $\Ncal_0$ of $K$. Up to considering smaller $\Ncal_0$, we can assume that the extended metric remains Riemannian on $\Ncal_0$ as positivity and semidefiniteness are open conditions. Since we can do this extension near each basic and since the metric $g$ is Riemannian, a partition of unity argument allows to prove
 
 \begin{lemma}
  For any Axiom A flow on a compact manifold, there exists a continuous Riemannian metric (globally defined) which is adapted to the dynamics on each basic set, in the sense that (\ref{hyperbolic estimates for the adapted metric}) holds on each basic set.
 \end{lemma}

 In what follows, we will always assume that $g$ is a continuous Riemannian metric adapted to the dynamics on each basic set and we will denote by $|.|$ its norm on the fibers of $TM$ and $T^* M$ to lighten notations. 
 \vspace{0.2cm}
 
 \subsubsection{Extension of distributions} We now apply Lemma \ref{Lemme geometrique prolongement des sections} in order to extend the distributions $E_{\cdots}^*$ defined on $K$ on a neighborhood of $K$. Note that we already explained that $E_s^*$, $E_u^*$, ... are well defined all over $M$ using the partition into unstable manifolds. The point of this new extension based on the bundles on $K$ (and not on the global dynamics) is that we expect that these new bundles have good hyperbolic properties in the sense of (\ref{hyperbolic estimates for the adapted metric}). We will also need to make sure that our local analysis is related to the invariant distributions $E_{u/uo}^*$ and $E_{s/so}^*$ defined all over $M$.
  
  Fix a basic set $K$ and recall that the dimension of the distributions $E_s, E_u, E_o$ are constant on $K$. We denote by $d_{s} = \dim E_s$ and $d_{u} = \dim E_u$ their dimension on $K$. According to Lemma \ref{lemma: closure of local stable manifold on a basic set}, the inclusions $\overline{W_{\varepsilon}^s (K)} \subseteq W_{2\varepsilon}^s (K)$ and $\overline{W_{\varepsilon}^u (K)}  \subseteq W_{2\varepsilon}^u (K)$ hold for $\varepsilon$ small enough. Now, if we apply Lemma \ref{Lemme geometrique prolongement des sections} for the continuous\footnote{The continuity follows from an adaptation of the proof of Lemma 4.2 of \cite{dyatlov} using the exponential estimates on $T_x W^s (z)$ and $T_x W^u (z)$ for every $z \in K$ which can be found in \cite[Lemma 2.10, p. 13]{dyatlov-guillarmou_2016}.} section 
  $$s: \overline{W_{\varepsilon}^{s} (K)} \rightarrow G_{s}(M)$$ with value in the Grassmann vector bundle of subspaces of dimension $d_{s}$, then we obtain an open set $\Ncal_s \supset \overline{W_{\varepsilon}^s (K)} $ and an extension $\widetilde{E}_s$ of the distribution $\bigcup_{x \in \overline{W_{\varepsilon}^s (K)}}E_{s} (x)$ on $\Ncal_s$. If we replace $s$ by $u$ in the previous construction then we obtain similarly a continuous extension $\widetilde{E}_u$ of $ \bigcup_{x \in \overline{W_{\varepsilon}^u (K)}} E_{u} (x)$ on a neighborhood $\Ncal_u$ of $\overline{W_{\varepsilon}^u (K)}$.  
  Next, we define $ \widetilde{E}_{so}^*$  (resp. $\widetilde{E}_{uo}^*$) by taking the dual orthogonal of $\widetilde{E}_s$ (resp. $\widetilde{E}_u$). The notation $\widetilde{E}_{so}^*$ can seem a little ambiguous at first, because we extend first and then take the dual orthogonal. Yet, everything is consistent here since $\widetilde{E}_{so}^*$ also extends continuously the distribution  $\bigcup_{x \in \overline{W_{\varepsilon}^s (K)}} E_{so}^* (x)$. A similar remark holds for $\widetilde{E}_{uo}^*$. Moreover, by setting $\widetilde{E}_s^*:=  \widetilde{E}_{so}^* \cap \{\xi \in T_x^*M, \:  \xi(V(x)) = 0\}$ we obtain a continuous extension of $\bigcup_{x \in \overline{W_{\varepsilon}^s (K)}}E_{s}^* (x)$. The different steps can be summarized in the next diagram (which of course also hold if we replace s by u):
 
 $$ \bigcup_{x \in \overline{W_{\varepsilon}^s (K)}} E_s (x) \dashrightarrow \widetilde{E}_s \dashrightarrow \widetilde{E}_{so}^* \dashrightarrow  \widetilde{E}_s^* .$$
 Moreover, the distributions $\widetilde{E}_{so/s}^* $ extend $E_{so/s}^*$ on a neighborhood of the local stable manifold of $K$:
 $$
  \boxed{\forall x \in W_{\varepsilon}^s (K), \quad \widetilde{E}_{so}^* (x) =  E_{so}^* (x) \mbox{ and } \widetilde{E}_s^* (x) = E_s^* (x).}
 $$
 Similarly, we have
$$
  \boxed{\forall x \in W_{\varepsilon}^u (K), \quad \widetilde{E}_{uo}^* (x) =  E_{uo}^* (x) \mbox{ and } \widetilde{E}_u^* (x) = E_u^* (x).}
 $$
These two last statement will be crucial in the proof of the compactness Proposition \ref{prop compactness}.

 Finally, in order to extend continuously the neutral direction, we define for all $x \in \Ncal =\Ncal_s \cap \Ncal_u$,
 $$ \widetilde{E}_o^* (x):= \{ \xi \in T_x^* M, \: \xi \left( \widetilde{E}_s(x) + \widetilde{E}_u(x)\right) = 0 \}.$$
 
  \begin{remark}\label{remark: upper triangular block matrices}
 It is important to note that for every $x \in W_{\varepsilon}^{s} (K) \cap \Ncal$, we have $\widetilde{E}_o^* (x) \subseteq \widetilde{E}_{so}^* (x)$. Similarly, we have for every $x \in W_{\varepsilon}^{u} (K) \cap \Ncal$, the inclusion $\widetilde{E}_o^* (x) \subseteq \widetilde{E}_{uo}^* (x)$.
 \end{remark}
 
  \begin{figure}[t]
\centering
\def\svgscale{0.5}
 \executeiffilenewer{neighborhoods_extension_distribution.svg}{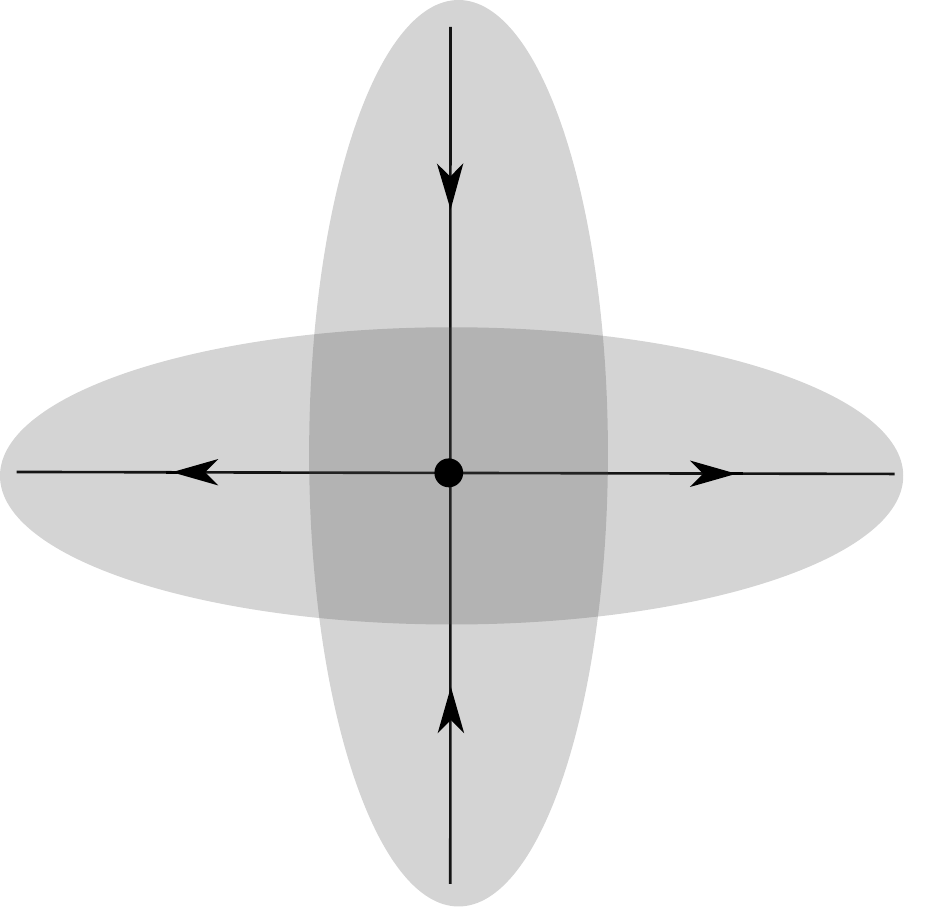}%
 {inkscape -z -D --file=neighborhoods_extension_distribution.svg %
 --export-pdf=neighborhoods_extension_distribution.pdf --export-latex}%
\begingroup%
  \makeatletter%
  \providecommand\color[2][]{%
    \errmessage{(Inkscape) Color is used for the text in Inkscape, but the package 'color.sty' is not loaded}%
    \renewcommand\color[2][]{}%
  }%
  \providecommand\transparent[1]{%
    \errmessage{(Inkscape) Transparency is used (non-zero) for the text in Inkscape, but the package 'transparent.sty' is not loaded}%
    \renewcommand\transparent[1]{}%
  }%
  \providecommand\rotatebox[2]{#2}%
  \newcommand*\fsize{\dimexpr\f@size pt\relax}%
  \newcommand*\lineheight[1]{\fontsize{\fsize}{#1\fsize}\selectfont}%
  \ifx\svgwidth\undefined%
    \setlength{\unitlength}{271.82766652bp}%
    \ifx\svgscale\undefined%
      \relax%
    \else%
      \setlength{\unitlength}{\unitlength * \real{\svgscale}}%
    \fi%
  \else%
    \setlength{\unitlength}{\svgwidth}%
  \fi%
  \global\let\svgwidth\undefined%
  \global\let\svgscale\undefined%
  \makeatother%
  \begin{picture}(1,0.96040629)%
    \lineheight{1}%
    \setlength\tabcolsep{0pt}%
    \put(0,0){\includegraphics[width=\unitlength,page=1]{neighborhoods_extension_distribution.pdf}}%
    \put(0.51070278,0.74800619){\makebox(0,0)[lt]{\lineheight{1.25}\smash{\begin{tabular}[t]{l}$\mathcal{N}_s$\end{tabular}}}}%
    \put(0.70716772,0.35933134){\makebox(0,0)[lt]{\lineheight{1.25}\smash{\begin{tabular}[t]{l}$\mathcal{N}_u$\end{tabular}}}}%
    \put(0.51389744,0.52696848){\makebox(0,0)[lt]{\lineheight{1.25}\smash{\begin{tabular}[t]{l}$\mathcal{N}'$\end{tabular}}}}%
    \put(0.39239639,0.38192418){\makebox(0,0)[lt]{\lineheight{1.25}\smash{\begin{tabular}[t]{l}$K$\end{tabular}}}}%
    \put(0.41715023,-0.02851194){\makebox(0,0)[lt]{\lineheight{1.25}\smash{\begin{tabular}[t]{l}$W_{\varepsilon}^s (K)$\end{tabular}}}}%
    \put(-0.25507689,0.44665648){\makebox(0,0)[lt]{\lineheight{1.25}\smash{\begin{tabular}[t]{l}$W_{\varepsilon}^u (K)$\end{tabular}}}}%
  \end{picture}%
\endgroup%

  \caption{Illustration of some neighborhoods used in the extension of distributions}
  \label{fig: neighborhoods extension of distributions} 
\end{figure}
 
 Up to considering smaller $\Ncal$, we can assume that 
    \begin{equation}\label{decomposition cotangent au voisinage d un compact basique hyperbolique}
   \widetilde{E}_{u}^*(x) \oplus \widetilde{E}_{s}^*(x) \oplus \widetilde{E}_{o}^*(x) = T_x^* M, \qquad \forall x \in \Ncal .
 \end{equation}
 Note that this decomposition of the cotangent space is \textbf{not invariant} by the flow $\Phi^t$ in general. Despite that, hyperbolic estimates as well as stability of good conical neighborhoods of these bundles should extend by continuity on a neighborhood of $K$.
 
 \subsubsection{Stability of conical neighborhoods near a basic set}
 
 We set for all $\delta > 0$ and all $x \in \Ncal$,
 \begin{equation}\label{definition voisinage conique tilde}
 \begin{split}
  \Ccal_{u}^{\delta} (x) &:= \left\{ \xi \in T_{x}^* M, \: \delta |\xi_u|^2 >  |\xi_s|^2 +  |\xi_o|^2 \right\}, \\ 
  \Ccal_{uo}^{\delta} (x) &:= \left\{ x \in T_{x}^* M, \: \delta (|\xi_u|^2 + |\xi_o|^2 ) > |\xi_s|^2 \right\},
  \end{split}
 \end{equation}
 where we used the decomposition (\ref{decomposition cotangent au voisinage d un compact basique hyperbolique}) on the fibers, i.e. $\xi = \xi_{u} + \xi_{s}+ \xi_{o} \in \widetilde{E}_{u}^*(x) \oplus \widetilde{E}_{s}^*(x) \oplus \widetilde{E}_{o}^*(x)$. If we replace $s$ by $u$ in  (\ref{definition voisinage conique tilde}), then we can define similarly the stable and weak-stable conical neighborhoods $\Ccal_{s}^{\delta}$ and $\Ccal_{so}^{\delta}$ on $\Ncal$. The main technical statement of this section is
 
\begin{lemma}\label{stabilite voisinage conique tilde pour Phi1} Let $K$ be a basic set and let $\Ncal$ be the open neighborhood appearing in (\ref{decomposition cotangent au voisinage d un compact basique hyperbolique}). There exists $\varepsilon_0 >0$ such that, for every unrevisited neighborhood $\Vcal \subset \Ncal \cap\varphi^{-1}(\Ncal)$ contained in an $\varepsilon_0$-neighborhood of $K$, the following hold:
\begin{enumerate}[label=(\roman*)]
 \item For every $0<\delta_0 \leq 1$, one can find $m_{\delta_0} \geq 0$ so that, for every $\delta_0 \leq \delta \leq 1$, for every $m \geq m_{\delta_0}$ and for every $x\in\mathcal{V}\cap\varphi^{m}(\Vcal)$,
  one has the inclusions 
  \begin{equation}\label{lemma stab vois conique tilde - Unstable inlusions}
	  \Phi^{1}\left( \Ccal_{u}^{\delta}(x)\right) \subseteq  \Ccal_{u}^{\delta'}(\varphi^{1}(x)), \qquad  
	\Phi^{1}\left( \Ccal_{uo}^{\delta}(x) \right) \subseteq \Ccal_{uo}^{\delta'}(\varphi^{1}(x))
 \end{equation}
 where $\delta' = e^{-\lambda/3} \delta$ and with $\lambda$ being the constant appearing in the definition (\ref{eq : def extended metric near basic set}). 
  \item Moreover, there exists $0 < \delta_1 \leq 1$ such that for every $\delta \leq \delta_1$, one can find $m_{\delta} \geq 0$ so that, for every $m \geq m_\delta$, for every $x \in \Vcal \cap\varphi^{m}(\Vcal)$ and for every $\xi \in \widetilde{C}^\delta_u(x)$, we have
  $$
 |\Phi^1 (x,\xi)|^2 \geq  e^{\lambda /3} |\xi|^2.
 $$
\end{enumerate}
\end{lemma}

   \begin{proof}
\textbf{Hyperbolic inequalities on the unstable manifold of $K$.} 
      Let us denote by $\pi_{\alpha}$ the \textbf{continuous} projector on $\widetilde{E}_{\alpha}^*$ for each $\alpha \in \{s,o,u\}$ and let us define 
   $$
	  \forall \alpha, \beta \in \{s,o,u\}, \quad A_{\alpha \beta}:= \pi_{\alpha} \circ \Phi^1 \circ \pi_{\beta} \mbox{ and } B_{\alpha \beta}:= \pi_{\alpha} \circ \Phi^{-1} \circ \pi_{\beta}.
   $$
   We will also use the notation $\xi_{\alpha}^1$ for $\pi_{\alpha} \circ \Phi^1 (\xi)$. For every $x \in \Ncal \cap \varphi^{-1}(\Ncal) $, we have $\xi = \xi_u + \xi_s + \xi_o \in \widetilde{E}_u^* (x) \oplus \widetilde{E}_s^* (x) \oplus \widetilde{E}_o^* (x) = T_x^* M$ and $\Phi_x^1 (\xi) = \xi_u^1 + \xi_s^1 + \xi_o^1 \in \widetilde{E}_u^* (\varphi^1 (x)) \oplus \widetilde{E}_s^* (\varphi^1 (x)) \oplus \widetilde{E}_o^* (\varphi^1 (x))$. The dynamics of $\Phi^1$ and $\Phi^{-1}$ can be encode within the two following matrices of linear morphisms:
 $$ \begin{pmatrix}\xi_u^1 \\ \xi_o^1 \\ \xi_s^1 \end{pmatrix} 
 = \begin{pmatrix}
A_{uu} & A_{ou} & A_{su} \\
A_{uo} & A_{oo} & A_{so} \\
A_{us} & A_{os} & A_{ss}
\end{pmatrix}
\begin{pmatrix}\xi_u \\ \xi_o \\ \xi_s \end{pmatrix}
\: \: \mbox{and} \: \: \begin{pmatrix}\xi_u \\ \xi_o \\ \xi_s \end{pmatrix} 
 = \begin{pmatrix}
B_{uu} & B_{ou} & B_{su} \\
B_{uo} & B_{oo} & B_{so} \\
B_{us} & B_{os} & B_{ss}
\end{pmatrix}
\begin{pmatrix}\xi_u^1 \\ \xi_o^1 \\ \xi_s^1 \end{pmatrix}.$$
To be more preciese, we should have written $A_{\alpha \beta} (x)$ and $B_{\alpha \beta} (\varphi^1(x))$ to indicate that $\xi \in T_x^* M$. In the particular case where $x \in W^u (K) \cap \Ncal \cap \varphi^{-1}(\Ncal)$, both matrices are upper triangular block matrices thanks to our definition of $\widetilde{E}_u^*$, $\widetilde{E}_o^*$ and $\widetilde{E}_s^*$, i.e.
$$
A_{uo} = A_{us} = A_{os} = B_{uo} = B_{us} = B_{os} = 0 \mbox{ on }W^u (K) \cap \Ncal \cap \varphi^{-1}(\Ncal). 
$$
 However, on the whole open set $\Ncal \cap \varphi^{-1}(\Ncal)$ this is not the case anymore\footnote{It follows from the fact that in the general case (where $\widetilde{E}_s^*$ and $\widetilde{E}_u$ are both non trivial) any of the extended distribution $\widetilde{E}_s^*$, $\widetilde{E}_u^*$ or $\widetilde{E}_o^*$ is invariant by $\Phi^t$ on $\Ncal$, except on the stable and unstable manifold $W^{s/u} (K) \cap \Ncal$.}. Now, we extend the hyperbolic estimates (\ref{hyperbolic estimates for the adapted metric}) on the unstable manifold of $K$. For every $\varepsilon \geq 0$, we define the set $\Ncal(\varepsilon)$ of points $y \in \Ncal \cap \varphi^{-1}(\Ncal)$ on which we have for all $\xi \in T_{y}^* M$,
$$  |A_{uu} (\xi_u)|  \geq  (e^{\lambda/2 }-\varepsilon ) |\xi_u|, \quad |B_{ss} (\xi_s^1)|  \: \geq  (e^{\lambda/2 } - \varepsilon) |\xi_s^1|, $$
  \begin{equation}\label{eq perturbe: lemme stabilite vois conique tilde - estimées matrices}
  |A_{oo} (\xi_o)| \geq  (1 - \varepsilon) |\xi_o|, \quad |B_{oo} (\xi_o^1)| \geq  (1 - \varepsilon ) |\xi_o^1|
 \end{equation}
 $$\begin{matrix}
	\| A_{ou} \| \leq \varepsilon , \: & \| A_{su} \| \leq \varepsilon , \: & \| A_{so} \| \leq \varepsilon , \: \\
	\| B_{ou} \| \leq \varepsilon, \: & \| B_{su} \| \leq  \varepsilon , \: & \| B_{so} \| \leq \varepsilon, \\
	A_{uo} = A_{us} = &A_{os} = B_{uo} = &B_{us} = B_{os} = 0.
 \end{matrix}$$
 where $\| . \|$ denotes the operator norm defined by:
 $$ \forall \alpha, \beta \in \{u,o,s\}, \quad \| A_{\alpha \beta} \|:=  \sup_{\xi_{\alpha} \in \widetilde{E}_{\alpha}^* (y) \setminus \{0 \}} \frac{|A_{\alpha \beta} (\xi_{\alpha})|}{ |\xi_{\alpha}|} .  $$
 Note that the map $A_{\alpha \beta}$ and the norm $\|.\|$ depend on the point $y$, but we will not preciese the point $y$ if everything is clear. For every $\varepsilon > 0$ and for every unrevisited neighborhood $\Vcal$ sufficiently close to $K$, we have $\Vcal \cap W^u (K) \subset \Ncal (\varepsilon)$. Let first check that the inclusions (\ref{lemma stab vois conique tilde - Unstable inlusions}) hold on the unstable manifold of $K$ for some $\delta' \leq \delta e^{-\lambda/2}$ as soon as $\varepsilon$ is sufficiently small. The result on the whole neighborhood will then follow by continuity.
 \begin{itemize}
 \item \textbf{First inclusion, on unstable cones $\Ccal_u^{\delta}$.} Let us fix $0 < \delta \leq 1$, $x \in W^u (K) \cap  \Ncal \cap \varphi^{-1} (\Ncal)$ and $\xi \in \Ccal_u^{\delta} (x)$. Our goal is to find a parameter $\delta' \leq \delta e^{-\lambda/2} $ such that $\xi^1:= \Phi_x^1 (\xi) \in \widetilde {\Ccal}_u^{\delta'} (\varphi^1 (x))$. The first step consists in computing $|\xi_u^1|^2$ in order to find a lower bound which essentially depends on $|\xi_u|$. Precisely, we have
 \begin{align*}
	|\xi_u^1|^2 &= |A_{uu} \xi_u + A_{su} \xi_s  + A_{ou} \xi_o |^2 \\
		    &\geq  |A_{uu} \xi_u|^2 +  \underbrace{|A_{su} \xi_s + A_{ou} \xi_o |^2}_{\geq 0} + 2 \underbrace{\w{A_{uu} \xi_u, A_{su} \xi_s + A_{ou} \xi_o}}_{I},
 \end{align*}
 where the bracket and the norm are the one induced by the extended adapted metric on the fiber. By Cauchy-Schwarz inequality and thanks to (\ref{eq perturbe: lemme stabilite vois conique tilde - estimées matrices}),we get
 \begin{align*}
	|I| &\geq - |A_{uu} \xi_u||A_{su} \xi_s | - |A_{uu} \xi_u||A_{ou} \xi_o | \geq - \varepsilon |A_{uu} \xi_u|^2 - \frac{\varepsilon }{2}(|\xi_s|^2 + |\xi_o|^2) 
 \end{align*}
 and therefore, for $\varepsilon < \frac{1}{2}$ and thanks to our assumption $\delta \leq 1$, we obtain 
 \begin{equation}\label{eq: proof perturb vois conique tilde - lower bound unstable}
      \begin{split}
	  \delta|\xi_u^1|^2 &\geq \delta (1 -2 \varepsilon) |A_{uu} \xi_u|^2 - \delta \varepsilon (|\xi_s|^2 + |\xi_o|^2) \\
		    &\geq (1 -2 \varepsilon )(e^{\lambda/2} -\varepsilon )^2 \delta | \xi_u|^2 - \delta \varepsilon (|\xi_s|^2 + |\xi_o|^2) \\
		    &> (1 -2 \varepsilon) (e^{\lambda/2} -\varepsilon)^2  (| \xi_s|^2 + | \xi_o|^2) - \varepsilon  (| \xi_s|^2 + | \xi_o|^2) \\
		    &\qquad =: C_1 (\varepsilon) (|\xi_s|^2 + |\xi_o|^2).
      \end{split}
 \end{equation}
 The second inequality is obtained thanks to the estimate on $A_{uu}$ in (\ref{eq perturbe: lemme stabilite vois conique tilde - estimées matrices}) and the third one follows from our choice of $\xi \in \Ccal_u^{\delta} (x)$. Note that we implicitely show the following estimate which will imply the exponential estimate (as we will see later on):
 \begin{equation}\label{inequality intermediate exponential bound}
	|\xi_u^1|^2 \geq \left( (1 -2 \varepsilon)(e^{\lambda/2} -\varepsilon )^2 - \delta \varepsilon  \right) |\xi_u|^2 \geq C_1 (\varepsilon) |\xi_u|^2,
 \end{equation}
 since $\delta \leq 1$. Now, let us do a similar computation for the matrix $B$. This time, we aim to find a lower bound for $|\xi_s|^2 + |\xi_o|^2$:
  \begin{align*}
	|\xi_s|^2 + |\xi_o|^2 &= | B_{ss} \xi_s^1 |^2 +  |B_{so} \xi_s^1  + B_{oo} \xi_o^1 |^2 \\
		    &\geq |B_{ss} \xi_s^1|^2 + \underbrace{|B_{so} \xi_s^1|^2}_{\geq 0} + |B_{oo} \xi_o^1|^2 + 2 \underbrace{\w{ B_{so} \xi_s^1, B_{oo} \xi_o^1}}_{J}. 
 \end{align*}
 Again by Cauchy-Schwarz inequality on $J$ and using again estimates (\ref{eq perturbe: lemme stabilite vois conique tilde - estimées matrices}), we deduce
 \begin{align*}
	J \geq - |B_{so} \xi_s^1||B_{oo} \xi_o^1 | \geq - \frac{\varepsilon}{2} (|\xi_s^1|^2 + |B_{oo} \xi_o^1|^2  )
 \end{align*}
 Then, it gives 
 \begin{equation}\label{eq: proof stability of conical neigh. - inequality 2 unstable}
    \begin{split}
	|\xi_s|^2 + |\xi_o|^2 &\geq   |B_{ss} \xi_s^1|^2  + (1 - \varepsilon) |B_{oo} \xi_o^1 |^2 - \varepsilon|\xi_s^1|^2 \\
	&\geq  (e^{\lambda/2} -\varepsilon)^2 |\xi_s^1|^2  + (1 - \varepsilon)(1 - \varepsilon)^2 |\xi_o^1|^2  - \varepsilon |\xi_s^1|^2 \\
	&\geq \left[  (1 - \varepsilon)^2 - \varepsilon \right] (|\xi_s^1|^2 + |\xi_o^1|^2) =: C_2 (\varepsilon)  (|\xi_s^1|^2 + |\xi_o^1|^2).
    \end{split}
\end{equation}
 Putting together equations (\ref{eq: proof perturb vois conique tilde - lower bound unstable}) and (\ref{eq: proof stability of conical neigh. - inequality 2 unstable}), we deduce
 \begin{align*}
	\delta|\xi_u^1|^2 > C_1 (\varepsilon)  C_2 (\varepsilon) (|\xi_s^1|^2 + |\xi_o^1|^2)  
 \end{align*}
 Finally, we obtain $\xi^1 = \Phi_x^1 (\xi) \in \Ccal_u^{\delta'}(\varphi^1 (x))$ for 
 $$
	 \boxed{\delta'(\varepsilon) = \frac{\delta }{C_1 (\varepsilon) C_2 (\varepsilon)}, }
 $$
 where the polynomial functions $C_1$ and $C_2$ (with respect to the variable $\varepsilon$) satisfy $C_1 (\varepsilon) \leq e^{\lambda}$ for every $0 \leq \varepsilon < \frac{1}{2}$ and 
 $$ C_1 (\varepsilon) \underset{\varepsilon \to 0}{\longrightarrow} e^{\lambda} \mbox{ and } C_2 (\varepsilon) \underset{\varepsilon\to 0}{\longrightarrow} 1 .$$

 If we fix $\varepsilon$ sufficiently small so that $\varepsilon$ be sufficiently small so that 
 \begin{equation}\label{eq : proof big lemma - estimates C1 C2}
  C_1 (\varepsilon) C_2 (\varepsilon ) \geq e^{\lambda/2}, \quad C_1 (\varepsilon) \geq e^{\lambda/2},
 \end{equation}
then we get for every $0 \leq \delta \leq 1$ and for every $x \in \Ncal (\varepsilon)$ the inclusion
$$\Phi^{1}\left( \Ccal_{u}^{\delta}(x)\right) \subseteq  \Ccal_{u}^{\delta'}(\varphi^{1}(x))$$
for $\delta'(\varepsilon) \leq \delta e^{-\lambda/2}$, and we deduce from (\ref{inequality intermediate exponential bound}) the bound 
\begin{equation}\label{inequality intermediate exponential bound 2}
 |\xi_u^{1}|^2 \geq e^{\lambda/2} |\xi_u|^2
\end{equation}

 \item \textbf{Second inclusion, on weak unstable cones $\Ccal_{uo}^{\delta}$}.  Let us fix $x \in W^u (K) \cap \Ncal \cap \varphi^{-1}(\Ncal) $ and $\xi \in \Ccal_{uo}^{\delta} (x)$. By definition of the conical neighborhood, we now only have $\delta (|\xi_u|^2 + |\xi_o|^2) \geq |\xi_s|^2 $. The idea of the proof is exactly the same as for the conical neighborhood $\Ccal_{u}^{\delta}$. However, we won't have an exponential estimate because of the neutral direction $\widetilde{E}_o^*$. Let us check the computations for this case. First, we have
 \begin{align*}
	(|\xi_u^1|^2 + |\xi_o^1|^2) &= |A_{uu} \xi_u + A_{su} \xi_s  + A_{ou} \xi_o |^2 + |A_{so} \xi_s  + A_{oo} \xi_o |^2 \\
		    &\geq  |A_{uu} \xi_u|^2 +  \underbrace{|A_{su} \xi_s + A_{ou} \xi_o |^2}_{\geq 0} + 2 \underbrace{\w{A_{uu} \xi_u, A_{su} \xi_s + A_{ou} \xi_o}}_{I_1} \\
		    & \qquad +  \underbrace{| A_{so} \xi_s |^2}_{\geq 0} + | A_{oo} \xi_o |^2 + 2 \underbrace{\w{A_{so} \xi_s, A_{oo} \xi_o}}_{I_2},
 \end{align*}
 Using Cauchy-Schwarz inequality on each term $I_1$ and $I_2$, we deduce
 \begin{align*}
	|I_1| &\geq - \varepsilon |A_{uu} \xi_u|^2 - \frac{\varepsilon}{2}|\xi_o|^2 -\frac{\varepsilon}{2}|\xi_s|^2 \\ 
	|I_2| &\geq - \frac{\varepsilon }{2} |A_{oo} \xi_o|^2 -\frac{\varepsilon }{2}|\xi_s|^2 .
 \end{align*}
 With $\varepsilon < \frac{1}{2}$ and thanks to the estimates (\ref{eq perturbe: lemme stabilite vois conique tilde - estimées matrices}),
 \begin{align*}
  	  |\xi_u^1|^2 + |\xi_o^1|^2  &\geq  (1 - 2\varepsilon ) |A_{uu} \xi_u|^2 +   (1 - \varepsilon) |A_{oo} \xi_o|^2 - \varepsilon |\xi_o|^2 -  2\varepsilon |\xi_s|^2 \\
  	  &\geq (1 - 2\varepsilon) (e^{\lambda/2} - \varepsilon)|\xi_u|^2 +   (1 - \varepsilon)^2 |\xi_o|^2  - \varepsilon |\xi_o|^2 -  2\varepsilon|\xi_s|^2 \\
  	  &\geq \left[ (1 - 2\varepsilon)(1 - \varepsilon) - \varepsilon  \right] (|\xi_u|^2 + |\xi_o|^2) - 2\varepsilon  |\xi_s|^2 .
 \end{align*}
 Now, multiplying by $\delta \leq 1$ on both side of previous inequalities and using $\xi \in \Ccal_{uo}^{\delta} (x)$, we get 
\begin{equation}\label{eq: lemme vois conique tilde - C3 2nd inclusion}
       \begin{split}
	 \delta  (|\xi_u^1|^2 + |\xi_o^1|^2) &\geq  \left[(1 - 2\varepsilon)(1 - \varepsilon) - 3\varepsilon \right] |\xi_s|^2 =: C_3 (\varepsilon)  |\xi_s|^2 .
      \end{split}
\end{equation}
  It remains to find a lower bound for $|\xi_s|$. This case is much simpler since we have
  \begin{align}\label{eq: lemme vois conique tilde - 3rd step, 2nd inclusion}
	|\xi_s|^2 = |B_{ss} \xi_s^1|^2 \geq (e^{\lambda/2} -\varepsilon)^2 |\xi_s^1|^2 =: C_4 (\varepsilon) |\xi_s^1|^2.
 \end{align}

Putting together equations (\ref{eq: lemme vois conique tilde - C3 2nd inclusion}) and (\ref{eq: lemme vois conique tilde - 3rd step, 2nd inclusion}), we get
 \begin{align*}
	 \delta(|\xi_u^1|^2 + |\xi_o^1|^2) \geq C_3 (\varepsilon)  C_4 (\varepsilon) |\xi_s^1|^2 .
\end{align*}

 Therefore, we have $\xi^1 = \Phi_x^1 (\xi) \in \Ccal_u^{\delta'}(\varphi^1 (x))$ for 
 $$
	\boxed{ \delta'(\varepsilon) = \frac{\delta }{C_3 (\varepsilon) C_4 (\varepsilon)}. }
 $$
 where the polynomial functions $C_3$ and $C_4$ (w.r.t the variable $\varepsilon$) satisfy $|C_3 (\varepsilon)| \leq 1$ for every $0 \leq \varepsilon < \frac{1}{4}$ and $ C_3 (\varepsilon) \underset{\varepsilon \to 0}{\longrightarrow} 1$ and $C_4 (\varepsilon) \underset{\varepsilon \to 0}{\longrightarrow} e^{\lambda}$. For $\varepsilon$ sufficiently small so that 
 \begin{equation}\label{eq : proof big lemma - estimates C3 C4}
   C_3 (\varepsilon)  C_4 (\varepsilon) \geq e^{\lambda/2},
 \end{equation}
 we get for every $0 \leq \delta \leq 1$ and for every $x \in \Ncal (\varepsilon)$ the inclusion
$$\Phi^{1}\left( \Ccal_{uo}^{\delta}(x)\right) \subseteq  \Ccal_{uo}^{\delta'}(\varphi^{1}(x)),$$
 with $\delta' (\varepsilon) \leq \delta e^{-\lambda/2}$. It ends the proof of the inclusions along the unstable manifolds.
\end{itemize}

Now, fix a value of $\varepsilon > 0$ sufficiently small so that (\ref{eq : proof big lemma - estimates C1 C2}) and (\ref{eq : proof big lemma - estimates C3 C4}) are verified. There exists $\varepsilon_0 > 0$ such that $\Vcal \cap W^u (K) \subset \Ncal (\varepsilon)$ hold for every unrevisited neighborhood $\Vcal$ contained in a $\varepsilon_0$-neighborhood of $K$. Let $\Vcal$ be an unrevisited neighborhood contained in a $\varepsilon_0$-neighborhood of $K$. Thanks to our choice of $\varepsilon$, the inclusions (\ref{lemma stab vois conique tilde - Unstable inlusions}) are verified on $\Vcal \cap W^u (K)$ and we would like to extend them to $\Vcal \cap \varphi^m( \Vcal)$ as stated in (\ref{lemma stab vois conique tilde - Unstable inlusions}). 

Recall that $\Vcal \cap \varphi^m (\Vcal)$ converges uniformly to $\overline{\Vcal} \cap W^u (K)$ as $m \to + \infty$ in the sense of Lemma \ref{lemme: cv nonrevisite vers variete stable et instable}. The remaining of the proof consists in extending the inclusions (\ref{lemma stab vois conique tilde - Unstable inlusions}) by continuity on $\Vcal \cap \varphi^m (\Vcal)$ for $m$ large enough. Fix $0 < \delta_0 \leq 1$. Let us define for every $x \in \varphi^1(\Ncal) \cap \varphi^{-1} (\Ncal)$ and every $\xi \in S_x^* M$ such that $(\xi_s, \xi_o) \neq 0$ and $\xi_u^1 \neq 0$ the following contraction rate
$$
 \Gamma (x,\xi) = \left( \frac{|\xi_s|^2 + |\xi_o|^2}{|\xi_u|^2}\right)^{-1} \left( \frac{|\xi_s^1|^2 + |\xi_o^1|^2}{|\xi_u^1|^2} \right).
$$
The map $\Gamma$ is continuous by definition. Moreover, for every $x \in W^u (K) \cap \overline{\Vcal}$ and for every $\xi \in \Ccal_u^1 (x) \setminus \Ccal_u^{\delta_0} (x)$ there exists $\delta > 0$ such that 
$\delta_0 \leq \delta < 1$, 
$$\frac{|\xi_s|^2 + |\xi_o|^2}{|\xi_u|^2} = \delta,$$
and our previous analysis on stability of conical neighborhoods allows to obtain
\begin{equation}\label{eq : big lemma - upper bound gamma on the unstable manifold}
 \Gamma (x,\xi) \leq e^{-\lambda /2}.
\end{equation}
Let us prove by contradiction that there exists $m_{\delta_0} \in \N$ such that for every $x \in \Vcal \cap \varphi^{m_{\delta_0}} (\Vcal)$ and for every $\xi \in \overline{\Ccal_u^1 (x)} \setminus \Ccal_u^{\delta_0} (x)$,
$$
 \Gamma (x,\xi) \leq e^{-\lambda/3}.
$$
The conclusion will then follow as $\Vcal \cap \varphi^m(\Vcal)$ is a decreasing sequence of neighborhoods. By contradiction, assume that for every $m \in \N$ there exist $x_m \in \Vcal \cap \varphi^m (\Vcal)$ and $\xi_m \in \overline{\Ccal_u^1 (x_m)} \setminus \Ccal_u^{\delta_0} (x_m)$ with $(x_m,\xi_m) \in S^* M$ such that $\Gamma (x_m, \xi_m) > e^{-\lambda /3}$. By compactness of $S^* M$, we can extract a subsequence $(x_{m_k},\xi_{m_k})_{k}$ which converges to an element $(x_{\infty},\xi_{\infty}) \in S^*M$. Since $\Vcal \cap \varphi^m (\Vcal)$ converges uniformly to $\overline{\Vcal} \cap W^u (K)$ as $m \to + \infty$ in the sense of Lemma \ref{lemme: cv nonrevisite vers variete stable et instable}, we must have $x_{\infty} \in W^u (K) \cap \overline{\Vcal}$. By continuity, we also get $\xi_{\infty} \in \overline{\Ccal_u^1 (x_{\infty})} \setminus \Ccal_u^{\delta_0} (x_{\infty})$. Moreover, we deduce from (\ref{eq : big lemma - upper bound gamma on the unstable manifold}) the inequality $\Gamma (x_{\infty},\xi_{\infty}) \leq e^{-\lambda /2}$. However, by construction of the sequence $(x_m,\xi_m)_m$ and by continuity of $\Gamma$ on the set 
$$\bigcup_{x \in \overline{\Vcal}} S_x^* M \cap \left( \overline{\Ccal_u^1 (x)} \setminus  \Ccal_u^{\delta_0} (x) \right) \subset \bigcup_{x \in \varphi^1 (\Ncal) \cap \varphi^{-1}(\Ncal)} \left\{ \xi \in T_x^* M, \: \xi_u^1 \neq 0 \mbox{ and } (\xi_s, \xi_o) \neq 0 \right\}$$ for $m$ sufficiently large, we must have $\Gamma (x_{\infty},\xi_{\infty}) \geq e^{-\lambda /3} > e^{-\lambda /2} $ which gives the expected contradiction. So, we have proved that for every $0 < \delta_0 \leq 1$, there exists $m_{\delta_0} \in \N$ such that for every $x \in \Vcal \cap \varphi^{m_{\delta_0}} (\Vcal)$, we have the inclusions  
\begin{equation}\label{eq : proof big lemma - stability of conical neigh N1}
 \forall \delta_0 \leq \delta \leq 1, \quad \Phi^1 \left( \Ccal_u^{\delta} (x) \setminus \Ccal_u^{\delta_0} (x) \right) \subseteq \Ccal_u^{\delta'} (\varphi^1 (x)),
\end{equation}
with $\delta ' = \delta e^{-\lambda/3}$. With a similar argument using the continuity of the map 
$$\Lambda (x,\xi) =   \frac{|\xi_s^1|^2 + |\xi_o^1|^2}{|\xi_u^1|^2},$$
we can assume, up to choosing a larger constant $m_{\delta_0}$, that 
\begin{equation}\label{eq : proof big lemma - stability of conical neigh N2}
 \Phi^1 \left( \Ccal_u^{\delta_0} (x)\right) \subseteq \Ccal_u^{\delta'_0} (\varphi^1 (x))
\end{equation}
is verified on $\Vcal \cap \varphi^{m_{\delta_0}} (\Vcal)$ for $\delta'_0 = \delta_0 e^{-\lambda/3}$. Finally, putting (\ref{eq : proof big lemma - stability of conical neigh N1}) and (\ref{eq : proof big lemma - stability of conical neigh N2}) together gives the expected inclusions for the unstable cones. For the weak unstable cones, we can proceed similarly using the following maps 
$$\widetilde{\Gamma} (x,\xi) = \left( \frac{|\xi_s|^2}{|\xi_u|^2 + |\xi_o|^2}\right)^{-1} \left( \frac{|\xi_s^1|^2}{|\xi_u^1|^2 + |\xi_o^1|^2} \right) \mbox{ and } \widetilde{\Lambda} (x,\xi) =   \frac{|\xi_s^1|^2}{|\xi_u^1|^2 + |\xi_o^1|^2}.$$
Up to choosing a larger $m_{\delta_0}$, it ends the proof of $(i)$.

\textbf{Exponential estimate (proof of $(ii)$).} Applying $(i)$ to $0 < \delta_0 = \delta \leq 1$, we get the existence of $m_{\delta} \in \N$ such that the inclusions (\ref{lemma stab vois conique tilde - Unstable inlusions}) are verified on $\Vcal \cap \varphi^{m_{\delta}} (\Vcal)$. Extending (\ref{inequality intermediate exponential bound 2}) by continuity, we can find $m_0 \in \N$ such that for every $x \in \Vcal \cap \varphi^{m_0} (\Vcal)$ and for every $\xi \in \Ccal_{u}^{\delta} (x)$, we have
$$ |\xi_u^1|^2 \geq e^{2\lambda / 5 } |\xi_u|^2  . $$
 Up to considering a larger $m_{\delta}$, we can assume that $m_{\delta} \geq m_0$. The exponential estimate follows by equivalence of the norms $|\xi |_u:= |\xi_u|$ and $|.|$ on the conical neighborhoods $\Ccal_u^{\delta}(x)$ and $\Ccal_u^{\delta}(\varphi^1(x))$. Indeed, we have $\xi^1 \in \Ccal_u^{\delta} (\varphi^1 (x))$ thanks to our previous analysis and therefore
 $$
	\boxed{|\Phi^1 (x,\xi)|^2 = |\xi^1|^2 \geq \frac{e^{2/5\lambda}}{1 + \delta} |\xi|^2 \geq e^{\lambda /3} |\xi|^2.}
 $$
for every $\delta \leq \delta_1:= \min (1, e^{2/5\lambda} - 1)$ and every $x \in \Vcal \cap \varphi^{m_{\delta}} (\Vcal)$.

\end{proof}

The following corollary states in a quantitative manner that, if the trajectory of a point $(x,\xi)$ stays for a long time near a basic set, then the fiber part of $\Phi^t (x,\xi)$ gets attracted to the $\widetilde{E}_u^*$ or $\widetilde{E}_{uo}^*$ distribution: 

\begin{corollary}\label{corollary big lemma} Let $K$ be a basic set and let $\Ncal$ be the open neighborhood appearing in (63). There exists $\varepsilon_0 >0$ such that, for every unrevisited neighborhood $\Vcal\subset \varphi^1(\Ncal) \cap\varphi^{-1}(\Ncal)$ contained in an $\varepsilon_0$-neighborhood of $K$, the following hold:
\begin{itemize}
 \item For every $0<\delta' \leq \delta\leq 1$, one can find $m_{0}\geq 0$ such that, for every $m \geq m_{0}$ and for every $x \in \varphi^{-2m}(\Vcal)\cap \Vcal$, one has 
 $$
	\boxed{ \xi \notin \Ccal_{so/s}^{\delta} (x) \Longrightarrow \Phi^{2 m} (x,\xi) \in \Ccal_{u/uo}^{\delta'} (\varphi^{2m} (x)) .}
  $$ 
 \item For every $0<\delta \leq 1$, there exists $C > 0$, $m_1 \in \N$ such that, for every $m \geq m_1$, for every $x \in \varphi^{-2m}(\Vcal)\cap \Vcal$ and for every $\xi \in T_x^* M$ such that $ \xi \notin \Ccal_{so}^{\delta} (x) $, we have the inequality
 \begin{equation}\label{eq : corollary big lemma - statement exponential estimates}
     |\Phi^{2 m} (\xi)| \geq C e^{m \lambda /3} |\xi|,
\end{equation}
with $\lambda$ being the constant appearing in the definition (\ref{eq : def extended metric near basic set}).
\end{itemize} 
\end{corollary}

The idea of the proof is the following. Thanks to Lemma \ref{stabilite voisinage conique tilde pour Phi1}, we know that the conical neighborhoods are $\Phi^1$-stable near the unstable manifold of the basic set, so our goal is to use the different unrevisited neighborhoods represented in figure \ref{fig: unrevisited neighborhoods near a basic set} to show that we can iterate the lemma.

  \begin{proof}[Proof of the corollary]Fix $0 < \delta' \leq \delta \leq 1$ and let $\Vcal$ be some unrevisited $\varepsilon_0$-neighborhood of $K$ contained in $\varphi^1(\Ncal)  \cap \varphi^{-1} (\Ncal)$ so that we can apply Lemma \ref{stabilite voisinage conique tilde pour Phi1} for the flows $\varphi^t$ and $\varphi^{-t}$ (with $\varepsilon_0$ given by the lemma). Doing so for the flows $\varphi^t$ and $\varphi^{-t}$, we obtain the existence of an integer $m_{\delta'}$ as in the statement of the lemma (we take $\delta_0 := \delta'$ and we assume that $m_{\delta'}$ is the same integer obtained for $\varphi^t$ and $\varphi^{-t}$). Now, fix $m_0 \geq m_{\delta'}$ which will be chosen sufficiently large later on. For all $m \geq m_{0}$ and all $x \in \Vcal \cap \varphi^{-2m} (\Vcal)$, we have the inclusions
  $$ \forall k \in [\![ 0, m_0 ]\!], \quad \varphi^{k} (x) \in \varphi^k (\Vcal) \cap \varphi^{-2m + k} (\Vcal) \subset \Vcal \cap \varphi^{- m_{0}} (\Vcal) $$
  and 
  $$ \forall k \in [\![m_0, 2m ]\!], \quad \varphi^{k} (x) \in \varphi^k (\Vcal) \cap \varphi^{-2m + k} (\Vcal) \subset \varphi^{m_{0}} (\Vcal) \cap \Vcal ,$$
  see figure \ref{fig: unrevisited neighborhoods near a basic set}. The remaining of the proof consists in interating $2m$ times the lemma for $(x,\xi) \in T^* M$ such that $x \in \Vcal \cap \varphi^{-2m} (\Vcal)$ with $m \geq m_{0}$ and $\xi \notin \Ccal_{so/s}^{\delta} (x)$: $m_0$ times for the backward flow $\varphi^{-t}$ when the orbit of $x$ belongs to $\Vcal \cap \varphi^{- m_0} (\Vcal)$ and $2m - m_0$ times for $\varphi^t$ when the orbit of $x$ goes in $ \varphi^{m_{0}} (\Vcal) \cap \Vcal$. For $m$ sufficiently large, we will get the result. Let us consider $(x,\xi)$ as above. Two analysis are needed: first for $0 \leq k \leq m_0$ and then for $m_0 \leq k \leq 2m$.
  \begin{itemize}
   \item First, since $\varphi^k (x) \in  \Vcal \cap \varphi^{- m_{0}} (\Vcal) \subset \Vcal \cap \varphi^{- m_{\delta'}} (\Vcal)$ for all $0 \leq k \leq m_0-1$, a straight application of the lemma for the flow $\varphi^{-t}$ instead of $\varphi^t$ gives us 
  $$
     \xi \notin \overline{\Ccal_{so/s}^{\delta} (x)} \Longrightarrow \Phi^{m_0} (x,\xi) \notin \overline{\Ccal_{so/s}^{\delta_{m_0}} (\varphi^{m_0}(x))}
  $$
  where $\delta_{m_0} = \min (e^{\lambda m_0 /3} \delta, 1)$. So, by choosing $m_0 \in \N$ large enough so that 
  $$e^{-m_0 \lambda /3} \leq \delta' ,$$ 
  we get in particular that $\delta_{m_0} = 1$, 
  $$ \Phi^{m_0} (x,\xi) \notin \overline{\Ccal_{so/s}^{1} (\varphi^{m_0} (x))}  \mbox{ and thus } \Phi^{m_0} (x,\xi) \in \Ccal_{u/uo}^{1} (\varphi^{m_0} (x)).$$
   \item For every $m_0 \leq k \leq 2m-1$, we can apply the lemma\footnote{Note that we use here the uniformity in $\delta \in [\delta', 1]$ stated in Lemma \ref{stabilite voisinage conique tilde pour Phi1}.} for the flow $\varphi^t$ and it gives
   $$\varphi^k (x) \in \Vcal \cap \varphi^{m_{\delta'}} (\Vcal), \: \Phi^k (x,\xi) \in \Ccal_{u/uo}^{\mu^{k-m_0}} (\varphi^m (x)) \Longrightarrow  \Phi^{k+1}(x,\xi) \in \Ccal_{u/uo}^{\mu^{k-m_0+1}} (\varphi^{k+1} (x)), $$
with $\mu = e^{-\lambda/3}$. Finally, we obtain
$ \Phi^{2m}(x,\xi) \in \Ccal_{u/uo}^{\mu^{2m - m_0}} (\varphi^{2m} (x)).$ Thanks to our choice of $m_0$, we have $\mu^{2m - m_0} \leq \mu^{m_0} = e^{-m_0 \lambda /3} \leq \delta'$. It gives the first point.
  \end{itemize}
   To prove the exponential estimate (\ref{eq : corollary big lemma - statement exponential estimates}), we need to introduce first the constant $0 < \delta_1 \leq 1$ given by the point $(ii)$ of Lemma \ref{stabilite voisinage conique tilde pour Phi1}. From the above analysis, there exists an integer $m_1 \geq m_0$ such that, for every $m \geq m_1$ and for every $ m_1 \leq k \leq 2m$, we have $ \Phi^k (x,\xi) \in \Ccal_{u/uo}^{\delta_1} (\varphi^k (x))$. The result is a direct application of the point $(ii)$ of Lemma \ref{stabilite voisinage conique tilde pour Phi1} and the constant $C$ is obtained by continuity of $\Phi^{m_1}$.
 \end{proof}

 \subsection{Proof of proposition \ref{prop compactness}: compactness.}
 \label{subsection : proof compactness result}
 
We are now in position to prove the compactness of $\Sigma_{u/uo}$. The compactness of $\Sigma_{s/so}$ can be proved similarly if we exchange $\varphi^1$ by $\varphi^{-1}$. Let us assume the indices to be ordered in the sense of \S \ref{subsubsection total order relation}. 
 Let $((x_m, \xi_m))_{m \in \N}$ be a sequence of elements in $\Sigma_{u/uo}$. Up to extraction of a subsequence, we can assume that there exists an integer $1 \leq j \leq N$ such that every point $x_m$ belongs to $W^s (K_j)$. Up to another extraction, we can suppose that $((x_m,\xi_m))_m$ has a limit in $S^*M$ that we denote by $(x_{\infty},\xi_{\infty}) \in S^*M$. Our goal will be to prove that $(x_{\infty}, \xi_{\infty})$ belongs to $\Sigma_{u/uo}$. First of all, we know from our assumption on $x_m$ that the limit $x_{\infty}$ must lie in the closure of the set $ W^s (K_j)$, i.e. $x_{\infty} \in \overline{W^s (K_j)}$. Since $\overline{W^s (K_j)} = \cup_{j_0, K_{j_0} \leq K_j} W^s (K_{j_0})$, we can find some integer $j_0 \leq j$ such that $x_{\infty} \in W^s (K_{j_0})$. Now, let us assume by contradiction that $(x_{\infty},\xi_{\infty}) \notin \Sigma_{u/uo}$.
 
 To obtain a contradiction, we split the analysis in three steps. First, when $j \neq j_0$, we will construct by induction a family of integers $j_0 < j_1 < \cdots < j_{\ell} = j$ such that for each $1 \leq k \leq \ell$ one can find an element $(x_{\infty}^{(k)}, \xi_{\infty}^{(k)})$ with $x_{\infty}^{(k)} \in W^s (K_{j_k}) \cap W^u (K_{j_{k-1}})$ and a sequence $(x_m^{(k)},\xi_m^{(k)}) = \widetilde{\Phi}^{\tau_{m,k}} (x_m,\xi_m)$, for some parameter $ \tau_{m,k} \geq 0$ verifying $\tau_{m,k +1} - \tau_{m,k} \underset{m\to+\infty}{\longrightarrow} + \infty$, which converges to $(x_{\infty}^{(k)},\xi_{\infty}^{(k)})$ as $m$ tends to infinity. In a second part, we will apply our previous analysis near basic sets from Corollary \ref{corollary big lemma} to prove that if $(x_{\infty}^{(k)},\xi_{\infty}^{(k)})$ does \textbf{not} belong to $\Sigma_{u/uo}$, then we have $(x_{\infty}^{(k+1)}, \xi_{\infty}^{(k+1)}) \in \Sigma_{so/s}$ which is disjoint to $\Sigma_{u/uo}$ thanks to the transversality assumption (\ref{Transversality assumption}). Finally, we will obtain the expected contradiction.
 
 First, let us see how the first two steps yield the contradiction. At the end of the induction, we will get $(x_{\infty}^{(\ell)}, \xi_{\infty}^{(\ell)}) \in \Sigma_{so/s}$ with $x_{\infty}^{(\ell)} \in W^s (K_j)$ as well as $x_{m}^{(\ell)} \in W^s (K_j)$ for every $m \in \N$. However , we also have the convergence 
 $$ \Sigma_{u/uo} \ni (x_m^{(\ell)},\xi_m^{(\ell)}) = \widetilde{\Phi}^{\tau_{m,\ell}} \underbrace{(x_m,\xi_m)}_{\in \Sigma_{u/uo}} \underset{m \to +\infty}{\longrightarrow} (x_{\infty}^{(\ell)}, \xi_{\infty}^{(\ell)}) \in \Sigma_{so/s},$$
 which is in contradiction with the continuity of $\Sigma_{u/uo}$ on $W_{\varepsilon}^s (K_j)$ for $\varepsilon \ll 1$. Indeed, if $x_{\infty}^{(\ell)} \in W^s (K_j)$ and if $\varepsilon > 0$ is a small parameter (that will be fixed later on), then there exists $T \geq 0$ (which depends on $\varepsilon$) such that $\varphi^T (x_{\infty}^{(\ell)}) \in W_{\varepsilon}^s (K_j)$. Since $x_m^{(\ell)} \in W^s (K_j)$ and the sequence $(x_m^{(\ell)})_{m \in \N}$ converges to $x_{\infty}^{(\ell)}$ as $m$ tends to $+\infty$, we can find an integer $m_0 \geq 0$ such that $\varphi^T (x_m) \in W_{\varepsilon}^s (K_j)$ for every $m \geq m_0$. Now, if we choose $\varepsilon$ small enough in order to have the continuity of $E_{so}^*$ and $E_s^*$ on $W_{\varepsilon}^s (K_j)$, then we find that $(x_{\infty}^{(\ell)}, \xi_{\infty}^{(\ell)}) \in \Sigma_{u/uo}$. Therefore, we get the contradiction
 $$(x_{\infty}^{(\ell)}, \xi_{\infty}^{(\ell)}) \in \Sigma_{u/uo} \cap \Sigma_{so/s},$$
 which is empty from the transversality assumption (\ref{Transversality assumption}).
 
 \textbf{Step 1: construction of the sequence by induction.} 
  	\begin{figure}[t]
\centering
\def\svgscale{0.6}
 \executeiffilenewer{compactness.svg}{compactness.pdf}%
 {inkscape -z -D --file=compactness.svg %
 --export-pdf=compactness.pdf --export-latex}%
 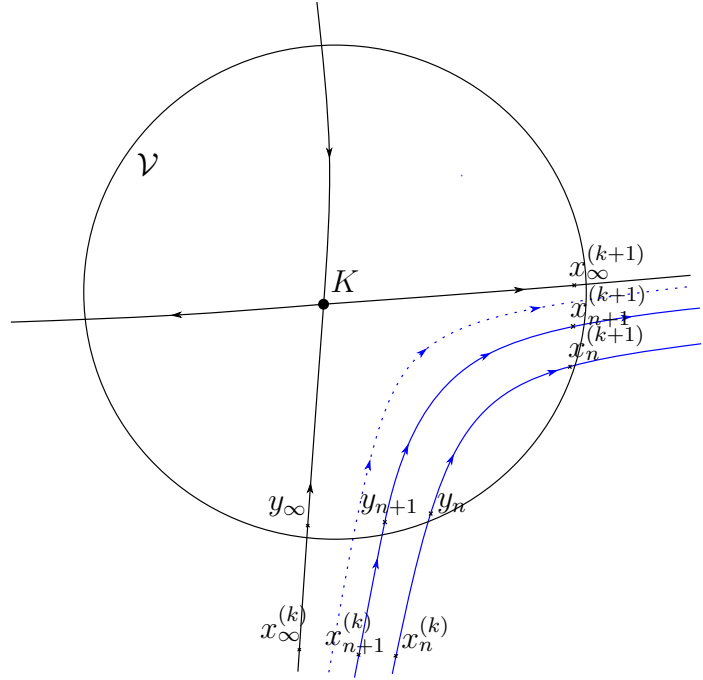%

  \caption{Illustration of some element used in the proof}
\end{figure}
	Since $j_0$ has already been defined, we just have to set $(x_{\infty}^{(0)}, \xi_{\infty}^{(0)}):= (x_{\infty}, \xi_{\infty})$ to end the base case. Now, let us assume the integers $j_1 < j_2 < \cdots < j_k (<j)$, the elements of the unitary cotangent bundle $(x_{\infty}^{(1)}, \xi_{\infty}^{(1)}),  \cdots, (x_{\infty}^{(k)}, \xi_{\infty}^{(k)})$ and the constants $\tau_{m,1}, \cdots, \tau_{m,k}$ to be constructed as in the above discussion. Our goal is to define a sequence $(x_{m}^{(k+1)}, \xi_{m}^{(k+1)}):= \widetilde{\Phi}^{\tau_{m,k+1}} (x_m,\xi_m)$, the constants $\tau_{m,k+1} \geq 0$ such that $\tau_{m,k+1} - \tau_{m,k}$ tends to $+\infty$ as $m$ goes to $+\infty$ and to exhibit a new accumulation point $(x_{\infty}^{(k+1)}, \xi_{\infty}^{(k+1)})$ with $x_{\infty}^{(k+1)} \in W^u(K_{j_k})\cap W^s( K_{j_{k+1}})$.

	\textbf{Definition of $\tau_{m,k+1}$:}	let us consider an unrevisited neighborhood $\Vcal$ of $K_{j_k}$ so that $\overline{\Vcal} \cap \Omega = K_{j_k}$ and sufficiently small to be in the range of application of Corollary \ref{corollary big lemma}. Recall that $x_{\infty}^{(k)}$ belongs to $W^s (K_{j_k})$, so there exists $T > 0$ such that $\varphi^T (x_{\infty}^{(k)}) \in W^s (K_{j_k}) \cap \Vcal$. Since every $x_{1}^{(k)}, x_{2}^{(k)}, \cdots, x_{m}^{(k)}, \cdots$ are in $W^s (K_{j})$ and the sequence converges to $x_{\infty}^{(k)}$ as $m \rightarrow + \infty$, there exists an integer $m_0 \geq 0$ such that for every $m \geq m_0$ we have $\varphi^T (x_m^{(k)}) \in  W^s (K_{j}) \cap \Vcal$.
	To lighten notations, let us denote by $y_m$ the point $\varphi^T (x_{m}^{(k)})$ and by $\eta_m$ the cotangent vector $\kappa ( ( D \varphi^{T} (x_m^{(k)})^{-1} )^{\top} (\xi_m^{(k)}) ) \in S_{y_m}^* M$ for every $m \in \N \cup \{ \infty \}$. 
	For every $m \geq m_0$, we can define the exit time of the unrevisited set $\Vcal$ for the point $y_m$ as follows:
	$$ \tau_m:= \inf \{ p \in \N, \: \varphi^p (y_m) \notin \Vcal \} -1 .$$ 
	It is well defined by definition of $y_m$ and it depends on $k$. Note that $\tau_m$ is finite since $y_m \notin W^s (K_{j_k})$ thanks to the unrevisited neighborhood's property \textbf{(P2)} of $\Vcal$. Also, the convergence $y_m \underset{m \to +\infty}{\longrightarrow} y_{\infty} \in W^s (K_{j_k})$ implies that $\tau_m$ goes to $+\infty$ when $m$ tends to $+\infty$. Now, we define $y_m^{(1)} = \varphi^{\tau_m} (y_m)$ and $\eta_m^{(1)} = \kappa ( ( D\varphi^{\tau_m} (y_m)^{-1} )^{ \top} (\eta_m) )$. Up to extraction, we can assume that $(y_m^{(1)}, \eta_m^{(1)})$ converges to an element $(y_{\infty}^{(1)},\eta_{\infty}^{(1)}) \in S^*M$ by compactness of $S^* M$. Let us define $(x_m^{(k+1)}, \xi_m^{(k+1)}):= (y_m^{(1)}, \eta_m^{(1)}) = \widetilde{\Phi}^{\tau_{m,k} + \tau_m + T} (x_m, \xi_m)$ for every $m \in \N \cup \{ \infty \}$ and therefore 
	$$\tau_{m, k+ 1}:= \tau_{m,k} + \tau_m + T \geq 0.$$
	Thanks to Lemmas \ref{lemme: cv nonrevisite vers variete stable et instable} and \ref{lemma: clos. of stable man on unrevisited neigh.}, the point $y_{\infty}^{(1)}$ belongs to $\overline{W^u (K_{j_k})\cap \Vcal} = W^u (K_{j_k})\cap \overline{\Vcal}$. Indeed, it follows from the convergence $y_m^{(1)} \underset{m \to +\infty}{\rightarrow} y_{\infty}^{(1)}$ together with the fact that $y_m^{(1)} \in \varphi^{\tau_m} (\Vcal) \cap \Vcal$ and $\tau_m \underset{m \to +\infty}{\longrightarrow} +\infty$. Moreover, we have $\varphi^{1} ( y_{\infty}^{(1)}) \notin \Vcal$ and thus $y_{\infty}^{(1)} \notin K_{j_k}$ because $\varphi^1 (y_m^{(1)}) \notin \Vcal$ by definition. It ends the induction step. Note that the algorithm stops once we have reached $K_j$, i.e. we have defined $j_0 < j_1 < \cdots < j_{\ell} = j$. 
	
  \textbf{Step 2: the local analysis near a basic set.} For every $0 \leq k < \ell$, we will prove the implication
	\begin{equation}\label{eq: compactness - proof implication}
	      \boxed{(x_{\infty}^{(k)}, \xi_{\infty}^{(k)}) \notin \Sigma_{u/uo} \mbox{ and } \tau_{m,k +1} - \tau_{m,k}  \underset{m\to+\infty}{\longrightarrow} + \infty \Longrightarrow (x_{\infty}^{(k+1)}, \xi_{\infty}^{(k+1)}) \in \Sigma_{so/s}.}
	\end{equation}
  Let us fix $0 \leq k < \ell$. To simplify, we will use the same notations as the one used in the previous induction. By definition, we have $(x_\infty^{(k+1)},\xi_\infty^{(k+1)}) \in \Sigma_{so/s}$ if and only if $\eta_{\infty}^{(1)} \in E_{u/uo}^* (y_{\infty}^{(1)})$. Since $y_{\infty} \in W^s (K_{j_k}) \cap \overline{\Vcal}$, since $x_{\infty}^{(k+1)} = y_{\infty}^{(1)} \in W^u (K_{j_k}) \cap \overline{\Vcal} \setminus \varphi^{-1} (\Vcal)$ and since the conical neighborhoods $\Ccal_{*}^{\delta}$ are well-defined on $\overline{\Vcal}$, we have by construction 
  $$ E_{so/s}^* (y_{\infty}) = \widetilde{E}_{so/s}^* (y_{\infty}) \quad \mbox{and} \quad E_{u/uo}^* (y_{\infty}^{(1)}) = \widetilde{E}_{u/uo}^* (y_{\infty}^{(1)}). $$
  Assume that $\eta_{\infty} \notin E_{so/s}^* (y_{\infty})$ and let us prove that $\eta_{\infty}^{(1)}$ belongs to $E_{u/uo}^* (y_{\infty}^{(1)})$ or equivalently: that for any $\delta' > 0$, we have 
  \begin{equation}\label{eq: compactness - phase at the exit}
   \eta_{\infty}^{(1)} \in \Ccal_{u/uo}^{\delta'} (y_{\infty}^{(1)}).
  \end{equation}
  Fix $\delta' > 0$. By hypothesis, we can find a small constant $\delta > 0$ such that 
	$$
		\eta_{\infty} \notin \Ccal_{so/s}^{\delta} (y_{\infty}).
	$$
	Since $\tau_{m,k +1} - \tau_{m,k}  \underset{m\to+\infty}{\longrightarrow} + \infty$, we can apply Corollary \ref{corollary big lemma} which gives (\ref{eq: compactness - phase at the exit}) and thus $(x_{\infty}^{(k+1)}, \xi_{\infty}^{(k+1)}) \in \Sigma_{so/s}$.
	
	\textbf{Conclusion.} We started with a point $(x_{\infty}^{(0)}, \xi_{\infty}^{(0)}) \notin \Sigma_{u/uo}$. Thanks to (\ref{eq: compactness - proof implication}), we have $(x_{\infty}^{(1)}, \xi_{\infty}^{(1)}) \in \Sigma_{so/s}$ which is disjoint to $\Sigma_{u/uo}$ according to the transversality assumption. Therefore, we can iterate (\ref{eq: compactness - proof implication}) by induction to finally get $(x_{\infty}^{(\ell)}, \xi_{\infty}^{(\ell)}) \in \Sigma_{so/s}$ and this concludes the proof after application of the third step to end up the induction. \hfill $\square$
 
 \vspace{0,5cm}
 Let us state a corollary which will be used to construct the map $f$ which appears in the definition of the escape function - see Proposition \ref{prop escape function}.
  
 \begin{corollary}\label{Corollary compactness proposition}
  For every $\varepsilon > 0$, there exists $\varepsilon' > 0$ such that, for every point $x$ which is $\varepsilon'$-close to $K$,
  $$
        d_{S^* M} \left(\kappa \left( E_{u/uo}^* (x) \right), \bigcup_{z \in K} \kappa\left( E_{u/uo}^* (z) \right) \right) < \varepsilon.
  $$
  where the distance should be understood as a distance between compact subset of $S^* M$ associated with the geodesic distance for the Sasaki metric on $S^*M$.
 \end{corollary}

 \begin{proof}
  By contradiction, assume that there exists $\varepsilon > 0$ such that for every $m \in \N^*$, there exist $x_m$ and $\xi_m \in \kappa \left(E_{u/uo}^* (x_m) \right)$ which verify
  \begin{equation}\label{eq : corollary compactness}
   d_g (x_m, K) \leq 1/m \quad \mbox{and} \quad d_{S^* M} \left( (x_m,\xi_m ), \bigcup_{z \in K} \kappa\left( E_{u/uo}^* (z) \right) \right) \geq \varepsilon.
  \end{equation}
  By compactness of $S^* M$, we can extract a converging subsequence $((x_{m_k},\xi_{m_k} ))_k$ which converges to an element $(x_{\infty},\xi_{\infty}) \in S^* M$ as $k \to + \infty$. Taking the limit in the inequalities (\ref{eq : corollary compactness}) implies that $x_{\infty} \in K$ and $\xi_{\infty} \notin \kappa \left( E_{u/uo}^* (x_{\infty} ) \right)$. However, the set $\Sigma_{s/so} = \kappa  \left( E_{u/uo}^* \right)$ is a compact set thanks to the compactness Proposition \ref{prop compactness}. Therefore, we must have $(x_{\infty},\xi_{\infty}) \in \Sigma_{s/so}$ or equivalently $\xi_{\infty} \in \kappa \left( E_{u/uo}^* (x_{\infty}) \right)$ which gives the contradiction. 
 \end{proof}

\subsection{The compact sets are attracting and repelling sets for the Hamiltonian flow}
\label{subsection : proof Lemma attracting and repelling sets for the Hamiltonian flow}

 \begin{figure}[!t]
\centering
\def\svgscale{1.1}
 \executeiffilenewer{cvdualdistrib.svg}{cvdualdistrib.pdf}%
 {inkscape -z -D --file=cvdualdistrib.svg %
 --export-pdf=cvdualdistrib.pdf --export-latex}%
 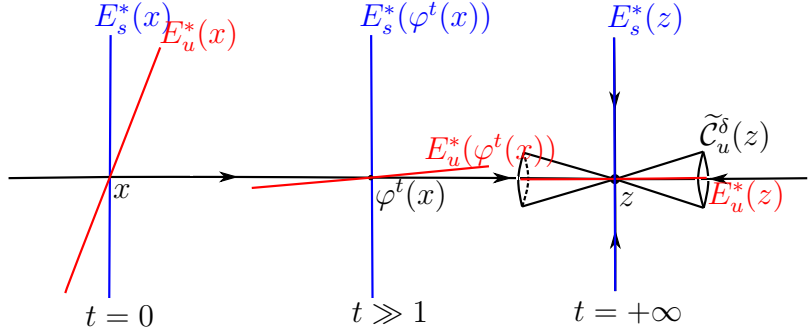%

  \caption{Convergence of the dual unstable distribution. The point $z$ denotes an hyperbolic fixed point and the point $x$ belongs to the stable manifold of $z$.}
  \label{fig: cv of dual distribution for Morse-Smale gradient flows}
\end{figure}

Now that we have proved the compactness of $\Sigma_{uo/u}$ and $\Sigma_{s/so}$, it remains to show that $(\Sigma_{s/so}, \Sigma_{uo/u})$ defines an attractor-repeller couple. Equivalently, we have to prove Lemmas \ref{lemma attractor and repeller on the phase space} and \ref{lemma invariant neighborhoods}. In the upcoming argument, we generalize the convergence presented in figure \ref{fig: cv of dual distribution for Morse-Smale gradient flows} to Axiom A flows satisfying the transversality assumption (\ref{Transversality assumption - first}). The proof is very similar to the one given in \cite[lemma 2.10 assertion 4, p.13]{dyatlov-guillarmou_2016} except we authorize our phase point $\xi$ to have a neutral component, i.e. $\xi (V)$ can be non zero.

\begin{proof}[Proof of Lemma \ref{lemma attractor and repeller on the phase space}] Let us take $(x,\xi) \in S^* M$ such that $(x,\xi) \notin  \Sigma_{uo/u}$ and let us prove that $\widetilde{\Phi}^t (x, \xi)$ tends to $\Sigma_{s/so}$ as $t \to +\infty$. The case $t\rightarrow-\infty$ is obtained by applying the result to the vector field $-V$. From the spectral decomposition of $M$ (Lemma \ref{lemma: spectral decomposition of the manifold}), we have $M = \sqcup_{i = 1}^N W^s (K_i)$ and consequently there exists a unique integer $1 \leq i \leq N$ such that $x \in W^s (K_i)$. Let $\Vcal$ be an unrevisited neighborhood of $K_i$ sufficiently small to be in the range of application of Corollary \ref{corollary big lemma}. Also, there exists $T > 0$ such that $\varphi^t (x) \in W^s (K_i) \cap \Vcal$ for every $t \geq T$. Now, let us recall that for every $y \in W^s (K_i) \cap \Vcal$ we have $$\widetilde{E}_s^* (y) = E_s^* (y) \quad \mbox{and} \quad \widetilde{E}_{so}^* (y) = E_{so}^* (y)$$
by construction. Since $\Sigma_{uo/u}$ is $\widetilde{\Phi}^t$-invariant, the condition $(x,\xi) \notin  \Sigma_{uo/u} $ implies that $\widetilde{\Phi}^t (x,\xi) \notin  \Sigma_{uo/u} = \kappa(E_{so/s}^*)$ for every $t \geq 0$. Therefore, there exists $\delta > 0$ such that $ \widetilde{\Phi}^T (x,\xi) \notin \Ccal_{so/s}^{\delta}$ and Corollary \ref{corollary big lemma} implies that  
 $$d_{S^* M} \left( \widetilde{\Phi}^t (x,\xi) , \bigcup_{z \in K_i} \kappa \left( \Ccal_{u/uo}^{\delta'} (z) \right) \right) \underset{t \to +\infty}{\longrightarrow} 0 , \qquad \forall \delta' > 0 ,$$
 or equivalently that
 $$ d_{S^* M} \left( \widetilde{\Phi}^t (x,\xi), \bigcup_{z \in K_i } \kappa (E_{u/uo}^* (z)) \right) \underset{t \to +\infty}{\longrightarrow} 0 .$$
 It ends the proof.
\end{proof}

 \subsection{Proof of Lemma \ref{lemma invariant neighborhoods}: invariant neighborhoods for the Hamiltonian flow}\label{subsection : Lemma invariant neighborhoods for the Hamiltonian flow}
 
Now, we need to find invariant neighborhoods of $\Sigma_{s/so}$ and $\Sigma_{uo/u}$, i.e. we have to prove Lemma \ref{lemma invariant neighborhoods}. The idea of the proof follows from two remarks:

\begin{itemize}
 \item In the case where $K$ is an attractor, we can construct a $\widetilde{\Phi}^1$-stable neighborhood of $\{(x,\xi)\in\Sigma_{s/so}: x\in K\}$ as follows: if $\Vcal$ is an unrevisited neighborhood of $K$ small in the sense that the conical neighborhood $\Ccal_*^{\delta}$ are well-defined on $\Vcal$, then for all $0 < \delta \leq 1$ the sets
 $$\bigcup_{x \in \Vcal} \kappa \left( \Ccal_{u}^{\delta}(x) \right) \quad \mbox{and} \quad \bigcup_{x \in \Vcal} \kappa \left( \Ccal_{uo}^{\delta}(x) \right)  $$ 
 are $\widetilde{\Phi}^1$-stable neighborhoods of $\{(x,\xi)\in\Sigma_{s}: x\in K\}$ and $\{(x,\xi)\in\Sigma_{so}: x\in K\}$.
 \item In the gerenal case, thanks to Lemma \ref{stabilite voisinage conique tilde pour Phi1} we can note that if $\Vcal$ is an unrevisited neighborhood of $K$ (to be in the setting of this lemma), then the set
  $$ \bigcup_{x \in \varphi^m (\Vcal)\cap \Vcal} \kappa \left( \Ccal_{u/uo}^{\delta}(x) \right)$$
  defines an unrevisited neighborhood of $\bigcup_{x \in W^u (K) \cap \Vcal } \kappa\left( E_{u/uo}^* (x) \right)$ for every $0 < \delta \leq 1$ and every $m \gg 1$.
\end{itemize}

Our strategy will be to construct invariant neighborhoods $\Sigma_{s/so}$ and $\Sigma_{uo/uo}$ by induction similarly to what we did when proving Lemma \ref{lemma: equivalence filtration and unrevisited} for the flow $\varphi^t$ on $M$. Indeed, in this lemma, we saw that the existence of unrevisited neighborhoods is deeply related to the existence of filtrations. In the upcoming proof, we contruct a filtration\footnote{Even if the definition of filtration was given for the flow $\varphi^t$, it can be adapted for the Hamiltonian without too much effort, see Shub \cite{shub}.} of open set for the Hamiltonian flow $\widetilde{\Phi}^t$ starting from the filtration on the base. 
 
\begin{proof}[Proof of Lemma \ref{lemma invariant neighborhoods}] 
Let us only treat the case of $\Sigma_{s/so}$, as the other cases can again be treating similarly by reversing the sense of time. Recall that, thanks to Lemma \ref{lemma: equivalence filtration and unrevisited}, we are given a filtration $\Ocal_j^+$ of the manifold $M$ which is arbitrarly close to the unstable manifolds. Using a total order relation on the indices of the basic sets $K_j$, we will proceed by induction to construct a filtration for the diffeomorphism $\widetilde{\Phi}^1$
\begin{equation}\label{eq: proof lemma invariant neighborhood HR}
	\Ucal_j^{s/so} \: \: \varepsilon\mbox{-close to } \bigcup_{x \in \bigcup_{k > N-j} W^u (K_k)} \kappa \left( E_{u/uo}^* (x) \right).
\end{equation}
For $j = 1$, we define $\Ucal_1^{s/so}$ as a neighborhood of the attractor $\bigcup_{x \in K_N} \kappa \left( E_{u/uo}^* (x) \right)$ by fixing
$$
	\Ucal_1^{s/so}:= \bigcup_{x \in \varphi^m (\Vcal_N) \cap \Vcal_N} \kappa \left(\Ccal_{u/uo}^{\delta} (x) \right),
$$
where $\Vcal_N$ denotes an unrevisited neighborhood of the attractor $K_N$ sufficiently small to be in the range of application of Corollary \ref{corollary big lemma} and where $m$ is an integer. By choosing $\delta$ small enough and then $m$ sufficiently large, i.e $m \geq m_{\delta}$, then the conical neighborhoods and $\Ucal_1^{s/so}$ are $\widetilde{\Phi}^1$-stable and we can assume that
$$
	d_{S^* M} \left( \Ucal_1^{s/so}, \bigcup_{x \in K_N} \Sigma_{s/so}(x) \right) < \varepsilon .
$$
Note that this construction only uses the fact that $K_N$ is an attractor.

Now, let us deal with the induction step and assume that we can find for every $\varepsilon > 0$ some open sets $\Ucal_{j}^{s/so}$ satisfying (\ref{eq: proof lemma invariant neighborhood HR}) for any $j < i$. Fix $\varepsilon > 0$. We want to construct $\Ucal_{i}^{s/so}$. 

 First, consider an unrevisited neighborhood $\Vcal$ of $K:= K_{N-i+1}$ small enough so that we can apply Corollary \ref{corollary big lemma}. As in the proof of Lemma \ref{lemma: equivalence filtration and unrevisited}, we define the annulus
 $$\Acal (m):= \left( \Vcal \cap \varphi^m (\Vcal) \right) \setminus \left( \varphi^{-1}(\Vcal) \cap \varphi^m (\Vcal) \right).$$ 
 We recall that it satisfies the following properties (\ref{eq : lemma invariant neighborhoods on the base - convergence of annulus}), (\ref{eq : lemma filtration on the base - annulus on unstable manifold}): $\Acal (m) \neq \emptyset$,  $\Acal(m) \cap W^u (K) = \Acal(0) \cap W^u (K)$ for every integer $m \geq 0$ and $\Vcal \cap \varphi^m (\Vcal)$ (resp. $\Acal(m)$) converges uniformly to $\overline{\Vcal} \cap W^u (K)$ (resp. $\overline{\Acal (0)} \cap W^u(K)$) in the sense of Lemma \ref{lemme: cv nonrevisite vers variete stable et instable}. By the compactness Proposition \ref{prop compactness} and thanks to Lemma \ref{lemma attractor and repeller on the phase space}, there exists an integer $ m_0 \geq 0$ such that: 
 \begin{equation}\label{eq : lemma invariant neighborhood - uniform cv induction}
    \forall (x,\xi) \in \bigcup_{x \in \Acal (0) \cap W^u (K)} \kappa \left( E_{u/uo}^* (x) \right), \quad \widetilde{\Phi}^{m_0} (x,\xi) \in \Ucal_{i-1}^{s/so}.
 \end{equation}
 Indeed, the fact that (\ref{eq : lemma invariant neighborhood - uniform cv induction}) holds ponctually for each $(x,\xi)$ follows from and a direct application of Lemmas \ref{lemma: attractor and repeller for the flow on the basis} and \ref{lemma attractor and repeller on the phase space}. Furthermore, the constant $m_0$ can be chosen uniformly using the fact that $\Sigma_{uo/u}$ and $\Sigma_{s/so}$ are disjoint \textbf{compact} sets. By continuity and since $\Ucal_{i-1}^{s/so}$ is an open set, we can extend (\ref{eq : lemma invariant neighborhood - uniform cv induction}) on small conical neighborhoods: there exists $\delta_0 > 0$ such that
 $$
	\forall (x,\xi) \in \bigcup_{x \in \Acal (0) \cap W^u (K)} \kappa ( \Ccal_{u/uo}^{\delta_0}(x)),  \quad \widetilde{\Phi}^{m_0} (x,\xi) \in \Ucal_{i-1}^{s/so}.
 $$
 Let us define for all $m \geq 0$ and all $0 < \delta \leq 1$ the set
 $$
  \Wcal (m,\delta):= \bigcup_{x \in \Vcal \cap \varphi^m (\Vcal)} \kappa (\Ccal_{u/uo}^{\delta}(x)).
 $$
 
According to Lemmas \ref{lemme: cv nonrevisite vers variete stable et instable} and \ref{stabilite voisinage conique tilde pour Phi1}, for every $\delta > 0$ there exists an integer $m(\delta) > 0$ such that for all $m \geq m (\delta)$ the open set $\Wcal(m,\delta)$ is an arbitrarily small (as $m \rightarrow + \infty$ and $\delta \to 0$) unrevisited neighborhood of $\bigcup_{x \in \Vcal\cap W^u (K)} \kappa \left( E_{u/uo}^* (x) \right)$. Actually, we have something better. For every $m \geq m(\delta)$, the conical neighborhood $\Ccal_{u/uo}^{\delta}$ is $\Phi^1$-stable on $\Vcal \cap \varphi^m (\Vcal)$ in the sense that for all $\ell \in \N$,
\begin{equation}\label{eq : lemma invariant neighborhoods - stability}
 (x,\xi) \in \Wcal (m,\delta), \:  \varphi^{\ell} (x) \in \Vcal \cap \varphi^{m} (\Vcal) \Longrightarrow \: \forall k \in [\![0, \ell ]\!], \: \widetilde{\Phi}^k (x,\xi) \in \Wcal (m,\delta).
\end{equation}

From (\ref{eq : lemma invariant neighborhoods on the base - convergence of annulus}) as mentionned before, $\Acal (m)$ converges to $\overline{\Acal(0)} \cap W^u (K)$ as $m \to + \infty$, so for every $\delta \leq \delta_0$, there exists an integer $\widetilde{m}(\delta) \geq m(\delta)$ such that for all $m \geq \widetilde{m}(\delta)$,
\begin{equation}\label{eq: invariant neighborhood for hamilt flow - induction}
	\forall (x,\xi) \in \bigcup_{x \in \Acal (m)} \kappa (\Ccal_{u/uo}^{\delta}(x)), \qquad  \widetilde{\Phi}^{ m_0} (x,\xi) \in \Ucal_{i-1}^{s/so}.
\end{equation}  
 Therefore, we define for all $\delta \leq \delta_0$ and all $m \geq \widetilde{m}(\delta)$ the set
 $$ \Ucal_i^{s/so} (m,\delta):= \Ucal_{i-1}^{s/so} \cup \bigcup_{k = 0}^{m_0 - 1} \widetilde{\Phi}^k ( \Wcal (m,\delta )).$$
 It is $\widetilde{\Phi}^1$-stable from (\ref{eq: invariant neighborhood for hamilt flow - induction}). Indeed, let us consider $(y,\eta) \in  \widetilde{\Phi}^{k} (\Wcal(m,\delta))$ for some $0 \leq k \leq m_0 - 1$. There exists $(x,\xi) \in \Wcal(m,\delta)$ such that $(y,\eta) = \widetilde{\Phi}^k (x,\xi)$. We have two cases to deal with.
	\begin{itemize}
	 \item \textbf{1\ts{st} case:} $x \in  (\varphi^{m} (\Vcal) \cap \Vcal ) \setminus (\varphi^{m} (\Vcal) \cap \varphi^{-1} (\Vcal))  = \Acal (m)$. Since $x \in \varphi^{m} (\Vcal) \cap \Vcal$ we get $\varphi^{p}(x) \in \varphi^{p}(\varphi^{m} (\Vcal) \cap \Vcal)$ for every $p \in [\![0, m_0 -1 ]\!]$ and the statement (\ref{eq: lemme filtration cv uniform R(n)}) together with (\ref{eq : lemma invariant neighborhoods - stability}) gives us $\widetilde{\Phi}^{m_0}(x,\xi) \in \Ucal_{i-1}^+$. Therefore, we deduce that $\widetilde{\Phi}^{p} (x,\xi) \in \Ucal_i^{s/so} (m,\delta)$ for all $p \in \N$.
	 \item \textbf{2\ts{nd} case:} $x \in \varphi^{m} (\Vcal) \cap \varphi^{-1} (\Vcal)$. Then, we have
	 $$\varphi^1 (x) \in  \varphi^{m + 1} (\Vcal) \cap \Vcal \subseteq  \varphi^{m} (\Vcal) \cap \Vcal$$
	 and thanks to (\ref{eq : lemma invariant neighborhoods - stability}) we get that $\widetilde{\Phi}^1 (x,\xi) \in \Wcal (m,\delta) \subset \Ucal_i^{s/so} (m, \delta)$. Fix $d \geq 0$ such that $\varphi^d (x) \in \Acal (\widetilde{m} (\delta))$. A direct application of (\ref{eq : lemma invariant neighborhoods - stability}) gives that $\widetilde{\Phi}^d (x,\xi) \in \Wcal (m, \delta)$ and the first case applied to $\widetilde{\Phi}^d (x,\xi)$ implies that $\widetilde{\Phi}^{p} (x,\xi) \in \Ucal_i^{s/so} (m, \delta)$ for all $p \in \N$.
	 \end{itemize}
	 
 Finally, choosing $\delta$ small enough and then $m$ large enough, we can assume $\Ucal_i^{s/so} (m,\delta)$ to be a $\varepsilon$-neighborhood of $\bigcup_{x \in \bigcup_{j > N-i} W^u (K_j)} \kappa \left( E_{u/uo}^* (x) \right)$. It ends the induction and the proof.
 
 Note that at the end of the induction (which finishes), we constructed an arbitrarily small neighborhood $\Ucal_N^{s/so}$ of $\Sigma_{s/so}$.
 \end{proof}
 
 \begin{remark}
  We can note that the set $\Ucal_N^{s/so}$  constructed in the previous proof defines an arbitrarily small neighborhood of $\Sigma_{s/so}$ which is stable by $\widetilde{\Phi}^1$.
 \end{remark}

\subsection{Proof of Proposition \ref{prop energy function on the unitary cotangent bundle}: energy function for the Hamiltonian flow}\label{subsection : Proof energy functions for the Hamiltonian flow}

 Thanks to Proposition \ref{prop compactness}, Lemmas \ref{lemma attractor and repeller on the phase space} and \ref{lemma invariant neighborhoods}, we deduce that $(\Sigma_{uo},\Sigma_{s})$ and $(\Sigma_{u},\Sigma_{so})$ define attractor-repeller couples. Therefore, a straight application of Lemma \ref{lemma: energy function for attractor-repeller model} gives the result.

 \section{Construction of the escape function: proof of Proposition \ref{prop escape function}.}\label{section : Proof of escape function}

  Let us begin with some candidate for the escape function that depends on two auxiliary functions. We will see how the properties of the maps $E$ and $f$ will be related to the ones of the escape function. First of all, let us assume that the maps $E$ and $f$ have been defined properly and recall the expression of the escape function:
 $$
	G_m (x,\xi) = m(x,\xi) \log \sqrt{1 + f(x,\xi)^2},
 $$
 with $m(x, \xi) = E \left( x, \frac{\xi}{|\xi|} \right) \chi (|\xi|^2)$. To lighten notations, we will use the japanese bracket $\w{r}:=  \sqrt{1 + r^2}$ and the shortcut $\tilde{\xi}$ for $\frac{\xi}{|\xi|}$. Now, let us compute $X_H (G_m)$, for $\|\xi\|\geq 1$
   $$
	X_H \left( G_m (x,\xi) \right) =   X_H (E (x, \tilde{\xi}))\log \w{f(x,\xi)} +  E (x, \tilde{\xi}) X_H \left( \log \w{f(x,\xi)} \right)
   $$
  Our goal is to make sure that this quantity is nonpositive and is negative outside a conical neighborhood of $E_o^*$.
  
  \textbf{Definition of $E$.} Fix $\varepsilon >0$. We define $E \in \Ccal^{\infty} (S^*M)$ as follows:
  $$
	\boxed{E(x,\xi):= -E(x) + s + (u - n_0) E_+ (x,\xi) + (n_0 - s) E_- (x,\xi)}
  $$
  where $E \in \Ccal^{\infty} (M)$ denotes the energy function on the basis given by Proposition \ref{prop energy function for Axiom A flows} for some parameter $\varepsilon' > 0$ (which will be fixed later on) and $\lambda_1 = 0, \cdots, \lambda_j = \frac{n_0 (j-1)}{4(N-1)}, \cdots, \lambda_N = \frac{n_0}{4}$. Thus, the map $E(x)$ has value in the interval $[0, \frac{n_0}{4}]$ and there exists a family $\Ncal_i$ of $\varepsilon'$-neighborhoods of $K_i$ and a constant $\eta_0 > 0$ (which only depends on $\varepsilon'$) such that\footnote{We make the assumption that $N \geq 2$. The case $N=1$ corresponds to the Anosov case.} 
  $$
  - \Lie_V (E) < - \min_{1 < j \leq N}(\lambda_j - \lambda_{j-1}) \eta_0 \leq - \frac{n_0}{4(N-1)} \eta_0, \quad \mbox{on } M \setminus \bigcup_{i = 1}^N \Ncal_i . 
  $$
  Moreover, from Proposition \ref{prop energy function on the unitary cotangent bundle}, there exist smooth energy functions $E_{\pm} \in \Ccal^{\infty}(S^*M,[0,1])$, small neighborhoods $\Wcal^{s/so}$ of $\Sigma_{s/so}$, $\Wcal^{uo/u}$ of $\Sigma_{uo/u}$ and a constant $\eta > 0$ such that:
	\begin{align*}
	 \begin{split}
		 \Lie_{\widetilde{X}_H} E_{+} &\geq 0 \mbox{ on } S^* M \: \mbox{and } \Lie_{\widetilde{X}_H} E_{+} > \eta \mbox{ on } S^* M \setminus  \bigsqcup_{i = 1}^N (\Wcal^{uo} \cup \Wcal^{s}) , \\ 
		  \Lie_{\widetilde{X}_H} E_{-} &\geq 0 \mbox{ on } S^* M \: \mbox{and } \Lie_{\widetilde{X}_H} E_{-} > \eta \mbox{ on } S^* M \setminus  \bigsqcup_{i = 1}^N (\Wcal^{u} \cup \Wcal^{so}) .
	 \end{split}
	\end{align*}  
	
\begin{center}
  \begin{figure}[t]
  \begin{minipage}[c]{.46\linewidth}
\centering
\def\svgscale{0.8}
 \executeiffilenewer{Hamiltonianfirst.svg}{Hamiltonianfirst.pdf}%
 {inkscape -z -D --file=Hamiltonianfirst.svg %
 --export-pdf=Hamiltonianfirst.pdf --export-latex}%
\begingroup%
  \makeatletter%
  \providecommand\color[2][]{%
    \errmessage{(Inkscape) Color is used for the text in Inkscape, but the package 'color.sty' is not loaded}%
    \renewcommand\color[2][]{}%
  }%
  \providecommand\transparent[1]{%
    \errmessage{(Inkscape) Transparency is used (non-zero) for the text in Inkscape, but the package 'transparent.sty' is not loaded}%
    \renewcommand\transparent[1]{}%
  }%
  \providecommand\rotatebox[2]{#2}%
  \newcommand*\fsize{\dimexpr\f@size pt\relax}%
  \newcommand*\lineheight[1]{\fontsize{\fsize}{#1\fsize}\selectfont}%
  \ifx\svgwidth\undefined%
    \setlength{\unitlength}{192.05055297bp}%
    \ifx\svgscale\undefined%
      \relax%
    \else%
      \setlength{\unitlength}{\unitlength * \real{\svgscale}}%
    \fi%
  \else%
    \setlength{\unitlength}{\svgwidth}%
  \fi%
  \global\let\svgwidth\undefined%
  \global\let\svgscale\undefined%
  \makeatother%
  \begin{picture}(1,0.89269416)%
    \lineheight{1}%
    \setlength\tabcolsep{0pt}%
    \put(0,0){\includegraphics[width=\unitlength,page=1]{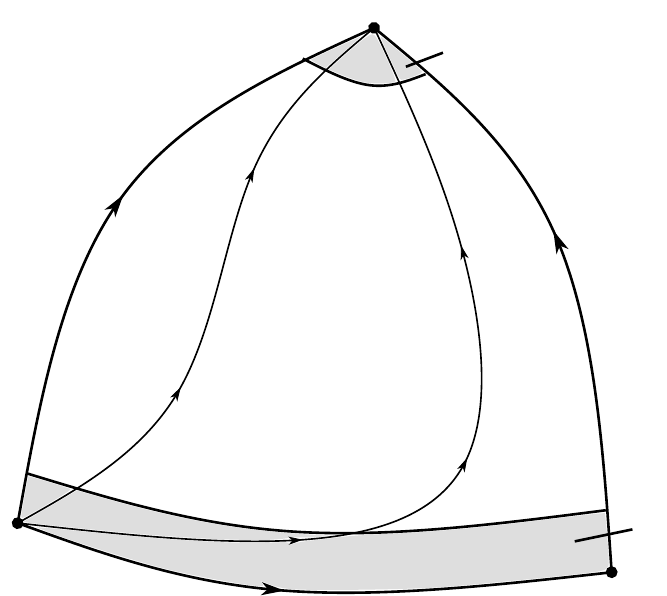}}%
    \put(0.66468589,0.79352199){\makebox(0,0)[lt]{\lineheight{1.25}\smash{\begin{tabular}[t]{l}$\mathcal{W}^s$\end{tabular}}}}%
    \put(0.94754162,0.0793803){\makebox(0,0)[lt]{\lineheight{1.25}\smash{\begin{tabular}[t]{l}$\mathcal{W}^{uo}$\end{tabular}}}}%
    \put(0.08811445,0.76299782){\makebox(0,0)[lt]{\lineheight{1.25}\smash{\begin{tabular}[t]{l}$S^* M$\end{tabular}}}}%
    \put(0.51294832,0.87239973){\makebox(0,0)[lt]{\lineheight{1.25}\smash{\begin{tabular}[t]{l}$\kappa (E_u^*)$\end{tabular}}}}%
    \put(0.8822652,-0.02350962){\makebox(0,0)[lt]{\lineheight{1.25}\smash{\begin{tabular}[t]{l}$\kappa (E_o^*)$\end{tabular}}}}%
    \put(-0.09233957,0.02565685){\makebox(0,0)[lt]{\lineheight{1.25}\smash{\begin{tabular}[t]{l}$\kappa (E_s^*)$\end{tabular}}}}%
  \end{picture}%
\endgroup%

  \end{minipage}
  \begin{minipage}[c]{.46\linewidth}
   \centering
\def\svgscale{0.8}
 \executeiffilenewer{Hamiltoniansecond.svg}{Hamiltoniansecond.pdf}%
 {inkscape -z -D --file=Hamiltoniansecond.svg %
 --export-pdf=Hamiltoniansecond.pdf --export-latex}%
\begingroup%
  \makeatletter%
  \providecommand\color[2][]{%
    \errmessage{(Inkscape) Color is used for the text in Inkscape, but the package 'color.sty' is not loaded}%
    \renewcommand\color[2][]{}%
  }%
  \providecommand\transparent[1]{%
    \errmessage{(Inkscape) Transparency is used (non-zero) for the text in Inkscape, but the package 'transparent.sty' is not loaded}%
    \renewcommand\transparent[1]{}%
  }%
  \providecommand\rotatebox[2]{#2}%
  \newcommand*\fsize{\dimexpr\f@size pt\relax}%
  \newcommand*\lineheight[1]{\fontsize{\fsize}{#1\fsize}\selectfont}%
  \ifx\svgwidth\undefined%
    \setlength{\unitlength}{215.32969557bp}%
    \ifx\svgscale\undefined%
      \relax%
    \else%
      \setlength{\unitlength}{\unitlength * \real{\svgscale}}%
    \fi%
  \else%
    \setlength{\unitlength}{\svgwidth}%
  \fi%
  \global\let\svgwidth\undefined%
  \global\let\svgscale\undefined%
  \makeatother%
  \begin{picture}(1,0.80197986)%
    \lineheight{1}%
    \setlength\tabcolsep{0pt}%
    \put(0,0){\includegraphics[width=\unitlength,page=1]{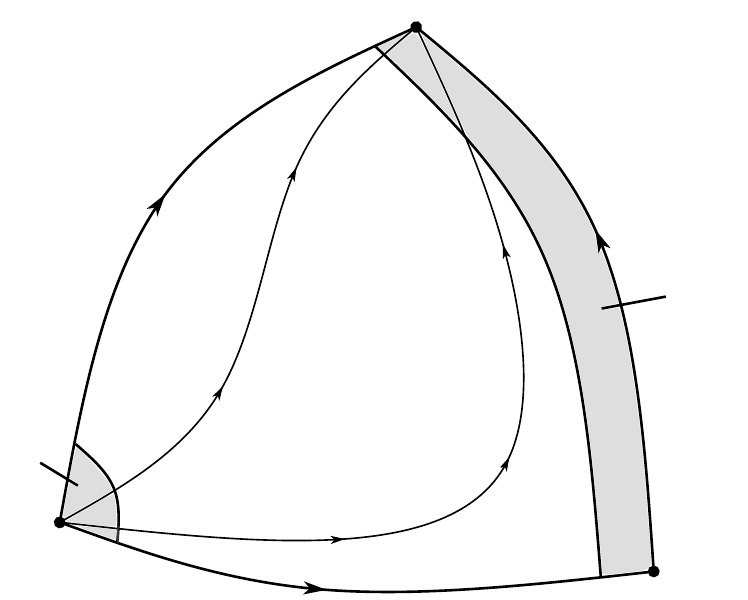}}%
    \put(-0.02735483,0.16832987){\makebox(0,0)[lt]{\lineheight{1.25}\smash{\begin{tabular}[t]{l}$\mathcal{W}^{u}$\end{tabular}}}}%
    \put(0.8875112,0.3913118){\makebox(0,0)[lt]{\lineheight{1.25}\smash{\begin{tabular}[t]{l}$\mathcal{W}^{so}$\end{tabular}}}}%
    \put(0.51856771,0.78760367){\makebox(0,0)[lt]{\lineheight{1.25}\smash{\begin{tabular}[t]{l}$\kappa (E_u^*)$\end{tabular}}}}%
    \put(0.84805511,-0.01236207){\makebox(0,0)[lt]{\lineheight{1.25}\smash{\begin{tabular}[t]{l}$\kappa (E_o^*)$\end{tabular}}}}%
    \put(-0.02583906,0.0318666){\makebox(0,0)[lt]{\lineheight{1.25}\smash{\begin{tabular}[t]{l}$\kappa (E_s^*)$\end{tabular}}}}%
  \end{picture}%
\endgroup%

  \end{minipage}
  \caption{Illustration of some elements used in the proof. Inspired from the picture of \cite[Fig. 6]{FS}.}
\end{figure}
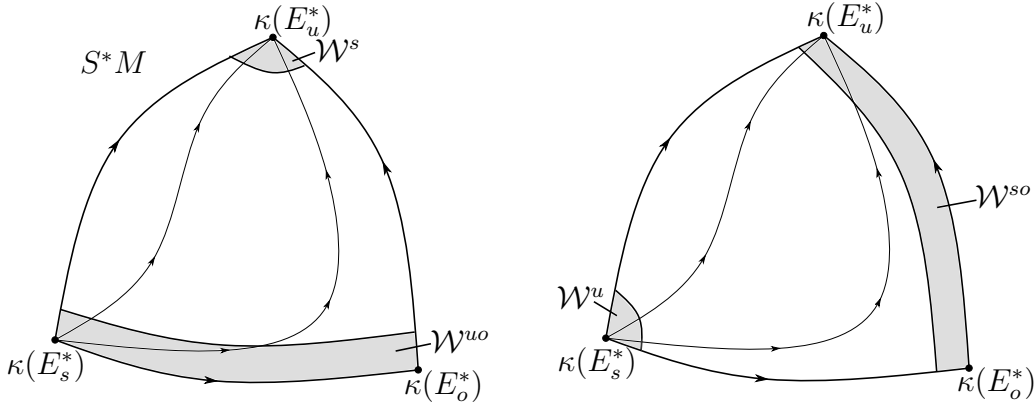
\end{center}

 We also have the estimates $E_{+} \geq 1 - \varepsilon$ on $\Wcal^{s}$, $E_{+} \leq \varepsilon$ on $\Wcal^{uo}$, $E_{-} \geq 1 - \varepsilon$ on $\Wcal^{so}$ and $E_{-} \leq \varepsilon$ on $\Wcal^{u}$. In order to use these estimates together, let us introduce new open sets in $S^*M$:
 $$
  \Ncal^{s}:= \Wcal^{s} \cap \Wcal^{so}, \quad \Ncal^{o}:= \Wcal^{so} \cap \Wcal^{uo}, \quad \Ncal^{u}:= \Wcal^{uo} \cap \Wcal^{u}.
 $$
   Let us introduce another notation which will be useful: for any open set $U$ in $S^*M$, we denote by $C(U)$ the conical neighborhood\footnote{defined by $C(U) = \left\{ (x,\lambda \xi) \in T^* M, \: \lambda \in \R^*, \: (x,\xi) \in U \right\}$.} of $U$ in $T^*M \setminus 0_M$. In particular, the sets $C(\Ncal^s)$, $C(\Ncal^o)$ and $C(\Ncal^u)$ define conical neighborhoods of $\cup_{x \in M} E_u^* (x)$, $\cup_{x \in M} E_o^* (x)$ and $\cup_{x \in M} E_s^* (x)$ respectively (outside the null section).
    \begin{itemize}
  \item On $\Ncal^s$, we have 
  \begin{align*}
	E(x,\xi) &\leq  \underbrace{-E (x)}_{\leq 0} + s + (u - n_0) (1 - \varepsilon) +  (n_0 - s) (1 - \varepsilon) \\
		&\leq s \varepsilon + u (1 - \varepsilon) \leq \frac{u}{2},
  \end{align*}
  for $\varepsilon$ small enough.
  \item On $\Ncal^u$, we obtain
  \begin{align*}
	E(x,\xi) &\geq  \underbrace{-E (x)}_{\geq - s/4} + s + (u - n_0) \varepsilon +  (n_0 - s)  \varepsilon \\
		&\geq s (1 - \varepsilon  - \frac{1}{4} ) + u \varepsilon \geq \frac{s}{4},
  \end{align*}
  for $\varepsilon$ small enough.
  \item On $\Ncal^o$, we obtain (again for $\varepsilon$ small enough) 
  \begin{align*}
	 \frac{n_0}{2} \leq  - \frac{n_0}{4}  + n_0 (1 - \varepsilon) + u \varepsilon  \leq E(x,\xi) \leq  n_0 ( 1 - \varepsilon) + s \varepsilon \leq n_0 .
  \end{align*}
  \end{itemize} 
  
  We now explain how to construct the function $f$ following the strategy of \cite{FS,dyatlov-zworski}. We want to construct a function $f \in \Ccal^{\infty} (T^*M)$ which is a homogeneous polynomial of degree $1$ for $|\xi| \geq 1$, so that $\Lie_{X_H} (f)$ is a bounded function. Since we want $\Lie_{X_H}(G_m)$ to be negative everywhere, we need to make sure that our definition of $f$ gives the right sign in the term $E . \Lie_{X_H}(\log \w{f})$. We will also choose $f$ such that $\Lie_{X_H} (f)$ vanishes near $E_o^*$. Moreover, we set $f(x,\xi) = \chi (|\xi|^2) |\xi| \tilde{f} ( x,\tilde{\xi} )$ with $\chi$ as before and with $\tilde{f} \in \Ccal^{\infty}(S^* M)$ that is to be define. We can already check that $\Lie_{X_H} (\log \w{f})$ is a bounded function on $T^* M$ for $|\xi| \geq 1$. Take an arbitrary basic set $K_i$.
  \begin{itemize}
   \item  On $\bigcup_{z \in K_i} \left\{(z,\xi) \in \Ncal^s \right\}$, we define the map $\tilde{f}$ by
   $$ \tilde{f}(x,\xi):= \int_0^T |\Phi^t (x,\xi)| dt .$$
  Thanks to the hyperbolicity of $\varphi^t$ on the basic set $K_i$, we have for $T$ sufficiently large
  $$ \Lie_{X_H} (\tilde{f})(x,\xi) = |\Phi^T (x,\xi)| - |\xi| \geq \tilde{C} e^{\tilde{\lambda} T} |\xi| - |\xi| \geq 2 |\xi|.$$
  Also, $\tilde{f}$ is a homogeneous polynomial of degree $1$, so we can find a constant $c > 0$ such that $0 < c^{-1} |\xi| < \tilde{f} (x,\xi) < c |\xi|$ on $\Ncal^s$. For $|\xi| \geq 1$, we get $\Lie_{X_H} (f) = \chi (|\xi|^2) \Lie_{X_H} (\tilde{f}) = \Lie_{X_H} (\tilde{f}) \geq \frac{2}{c}. c|\xi| \geq \frac{2}{c} \tilde{f} = \frac{2}{c} f$. Therefore, there exists a universal constant $\gamma > 0$ such that
  \begin{equation}\label{eq : proof escape function - estimate 1 on f}
   \Lie_{X_H} (\log \w{f} ) = \frac{f}{\w{f}^2} \Lie_{X_H} (f) \geq \frac{2}{c} \frac{f^2}{\w{f}^2} \geq \gamma > 0 .
  \end{equation}
 \item On $\bigcup_{z \in K_i} \{(z,\xi) \in \Ncal^u \}$, we define similarly
 $$ \tilde{f}(x,\xi):= \int_0^T |\Phi^{-t} (x,\xi)| dt .$$
 With the same remark, up to reducing the constant $\gamma$, we obtain
 \begin{equation}\label{eq : proof escape function - estimate 2 on f}
  \Lie_{X_H} (\log \w{f} )  \leq -\gamma < 0,
 \end{equation}
 as soon as $|\xi| \geq 1$.
 \item We extend the bound (\ref{eq : proof escape function - estimate 1 on f}), (\ref{eq : proof escape function - estimate 2 on f}) by continuity on a neighborhood of $K_i$. Precisely, there exists $\varepsilon_0 > 0$ such that for every point $(x,\xi)\in S^* M$ which is $\varepsilon_0$-close to $\bigcup_{z \in K_i} \{(z,\xi) \in \Ncal^s \}$ (resp. $\bigcup_{z \in K_i} \{(z,\xi) \in \Ncal^u \}$), we have
 \begin{equation}\label{eq : proof escape function - estimate on f near basic set}
   \Lie_{X_H} (\log \w{f} )  \geq \frac{\gamma}{2} > 0 \qquad \left(  \mbox{resp. } \Lie_{X_H} (\log \w{f} )  \leq -\frac{\gamma}{2} < 0  \right) .
 \end{equation}
 If we choose $\varepsilon$ sufficiently small then we can assume that $\varepsilon < \varepsilon_0$. Now, from Corollary \ref{Corollary compactness proposition}, there exists $\varepsilon' > 0$ (\textbf{now $\varepsilon'$ is fixed}) such that for every $x$ such that $d_g (x,K_i) < \varepsilon'$ we have
 $$ d_{S^* M} \left( \kappa \left( E_{u/s}^* (x) \right), \bigcup_{z \in K_i} \kappa \left( E_{u/s}^* (z) \right) \right) < \varepsilon_0 - \varepsilon.$$
 Thus, we deduce
 $$ d_{S^* M} \left( \bigcup_{ d_g (x,K_i) < \varepsilon'}\left\{ (x,\xi) \in \Ncal^{s/u} \right\} , \bigcup_{z \in K_i} \kappa \left( E_{u/s}^* (z) \right) \right) < \varepsilon_0$$
 and that (\ref{eq : proof escape function - estimate on f near basic set}) hold on $ \bigcup_{ d_g (x,K_i) < \varepsilon'}\left\{ (x,\xi) \in \Ncal^{s/u} \right\}$. In particular, since $\Ncal_i$ is an $\varepsilon'$-neighborhood of $K_i$, the inequalities (\ref{eq : proof escape function - estimate on f near basic set}) are verified on $\bigcup_{x \in \Ncal_i} \left\{ (x,\xi) \in \Ncal^{s/u} \right\}$.
 \item On $\bigcup_{z \in \Ncal_i} \left\{ (z,\xi) \in  \Ncal^o \right\}$, we fix $\tilde{f} (x,\xi) = |\xi (V (x))|$. This ensures that 
 $$\Lie_{X_H} (\log \w{f}) = 0.$$
 \item On $ \bigcup_{z \in \cup_i \Ncal_i} \left\{ (z,\xi) \in S^* M \setminus \left(\Ncal^s \sqcup \Ncal^o \sqcup \Ncal^u \right) \right\}$ and on $\bigcup_{x \notin \cup_{i} \Ncal_i} S_x^* M$, we let $\tilde{f}$ take arbitrary positive values on $S^*M$.
 \end{itemize}
 
 It now remains to show that $G_m$ has the expected decaying properties (namely points (2) and (3)) of Proposition \ref{prop escape function}.
  
 \textbf{Decaying estimates.} To compute the derivative of the map $E$ along the flow $\widetilde{\Phi}^t$ we will need at some point the next relation:
 $$
	\Lie_{\widetilde{X}_H} E = - \Lie_V (E) + (u - n_0) \Lie_{\widetilde{X}_H} E_+ + (n_0 -s) \Lie_{\widetilde{X}_H} E_- 
 $$
 and we can already see that it is nonpositive everywhere. For $|\xi| \geq 1$, we can estimate the quantity $\Lie_{X_H} G_m$ in different directions.
 \begin{itemize}
  \item On $\widetilde{\Ncal}^s := C \left( \bigcup_{x \in \cup_i \Ncal_i} \left\{ (x,\xi) \in  \Ncal^s \right\} \right)$, we get
  \begin{align*}
   \Lie_{X_H} G_m &= \underbrace{\Lie_{\widetilde{X}_H} (E)}_{\leq 0} \underbrace{\log \w{f}}_{> 0} + \underbrace{E (x,\tilde{\xi})}_{\leq u /2} \underbrace{\Lie_{X_H} (\log \w{f})}_{\geq \gamma/2} \\
    &\leq - \frac{\gamma}{4} |u|.
  \end{align*}
  \item On $\widetilde{\Ncal}^u := C \left( \bigcup_{x \in \cup_i \Ncal_i} \left\{ (x,\xi) \in  \Ncal^u \right\} \right)$, we get
  \begin{align*}
  \Lie_{X_H} G_m &= \underbrace{\Lie_{\widetilde{X}_H} (E)}_{\leq 0} \underbrace{\log \w{f}}_{> 0} + \underbrace{E (x,\tilde{\xi})}_{\geq s /4} \underbrace{\Lie_{X_H} (\log \w{f})}_{\leq - \gamma/2} \\
    &\leq - \frac{\gamma}{8} s .
  \end{align*}
  \item On $\widetilde{\Ncal}^o := C \left(\bigcup_{z \in \Ncal_i} \left\{ (z,\xi) \in  \Ncal^o \right\}\right)$, we obtain $\Lie_{X_H} G_m = \Lie_{\widetilde{X}_H} (E) \log \w{f} \leq 0$.
  \item On $C \left( \bigcup_{z \in \cup_i \Ncal_i} \left\{ (z,\xi) \in S^* M \setminus \left(\Ncal^s \sqcup \Ncal^o \sqcup \Ncal^u \right) \right\} \right)$. Or equivalently, for $x \in \bigcup_{1 \leq i \leq N} \Ncal_i$ with $(x,\xi) \notin C (\Wcal^{so}\cup \Wcal^{u})$ or $(x,\xi) \notin C (\Wcal^{s}\cup \Wcal^{uo})$. Since $E$ and $X_H (\log \w{f})$ are bounded and $f$ is $1$-homogeneous w.r.t $\xi$, there exist constants $C_1, C_2 > 0$ such that for all $|\xi| \geq 1$,
  $$ \Lie_{X_H} (G_m) =  \Lie_{\widetilde{X}_H} (E) \log \w{ |\xi| C_1} + C_2 .$$
  Moreover, according to the construction of $E$, one can find a positive constant $\eta > 0$ such that $ \Lie_{\widetilde{X}_H} (E) (x,\tilde{\xi}) < - \eta < 0$. Therefore, there exists a positive radius $R > 0$ such that for every $(x,\xi) \in T^*M \setminus \left( C(\Ncal^s) \cup C(\Ncal^o) \cup C(\Ncal^u) \right)$ with $x \in \bigcup_{1 \leq i \leq N} \Ncal_i$ and $|\xi| \geq R$, we have
  \begin{equation}\label{eq: proof of escape function - bound on derivative of the escape function}
	\Lie_{X_H} (G_m) \leq - \frac{\gamma}{8} \min(s, |u|) .
  \end{equation}
  Therefore, we define $C_m:= \frac{\gamma}{8}  \min(s, |u|) $.
  \item Outside a small neighborhood of the nonwandering set, i.e. for $x \in M\setminus \left( \cup_{i = 1}^N \Ncal_i \right)$. This time we have $ \Lie_{\widetilde{X}_H} (E) \leq - \frac{n_0}{4 (N-1)}\eta_0 < 0$. So, using a similar argument than the previous point, we deduce that equation (\ref{eq: proof of escape function - bound on derivative of the escape function}) still holds far away from the null section. \hfill $\square$
 \end{itemize}

 \begin{appendix}
 
 \section{Hyperbolic sets}\label{Appendix: hyperbolic set}

In this appendix, we recall the definition of a hyperbolic set.

\begin{definition}A $\varphi^t$-invariant compact set $K$ is said to be \textbf{hyperbolic} for the flow $\varphi^t$ on $M$ if 
\begin{itemize}
 \item For each $x \in K$, we have the following decomposition
      \begin{equation}\label{decompositionTM}
       T_x M = E_u (x) \oplus E_s(x) \oplus E_0 (x)
      \end{equation}
      where $E_0 (x) = \R V(x)$ and $E_s$ (resp. $E_u$) is called the stable (resp. unstable) distribution.
 \item The decomposition (\ref{decompositionTM}) is invariant by the flow $\varphi^t$: 
 $$
  \forall x \in K, \quad (D_x \varphi^t )(E_u(x)) = E_u(\varphi^t (x)) \quad \mbox{and} \quad (D_x \varphi^t )(E_s(x)) = E_s(\varphi^t (x)) 
 $$
 \item There are constants $C > 0$ and $\lambda > 0$ such that for every $x \in K$ and for every $\forall t \geq 0$, the following inequalities are verified 
 \begin{equation}\label{eq: hyperbolic inegality on an hyperbolic set}
  \begin{split}
   | D_x \varphi^t (v_s) |_g &\leq C e^{-\lambda t} |v_s|_g , \quad \quad \forall v_s \in E_s(x), \\
   | D_x \varphi^{-t} (v_u) |_g &\leq C e^{-\lambda t} |v_u|_g , \quad \quad \forall v_u \in E_u(x).
  \end{split}
 \end{equation}
\end{itemize} 
\end{definition}

When $K$ is reduced to a singleton $\{z\}$, we must have $V(z) = 0$ and thus $E_o (x) = \{0\}$. In that case, we say that $z$ is a hyperbolic fixed point. When the whole manifold is a hyperbolic compact set on which the vector field $V$ never vanishes, the flow is said to be Anosov. It was first introduced by D. Anosov in \cite{anosov} and this formal definition of hyperbolicity was motivated by the properties of the geodesic flow on negatively curved manifolds. Another famous example of Anosov flow is given by the suspension of an Anosov diffeomorphism. The notion of hyperbolicity was later extended by Smale who defined the notion of Axiom A flows which are at the heart of this article. There are many other examples of hyperbolic sets and we refer the reader to \cite{katok-hasselblatt} and \cite{palis-demelo} for a comprehensive study of hyperbolic dynamics. We also refer to \cite{dyatlov} for the case of flows.

\begin{remark}\begin{itemize}[align=left, leftmargin=*, noitemsep]
       \item  The definition of hyperbolicity does not depend on the continuous metric $g$ on $M$. Indeed, if $g'$ denotes another smooth metric then, by compactness of $M$, $g'$ is equivalent to $g$ and (\ref{eq: hyperbolic inegality on an hyperbolic set}) still holds for some constant $C'$ instead of $C$.
       \item The distributions $E_u$ and $E_s$ are only Hölder-continuous in general. Let $d_u (x)$ and $d_s (x)$ be the dimensions of $E_u (x)$ and $E_s (x)$ at any point $x \in K$. The maps $d_s$ and $d_u$ do not depend on $x \in K$ and $E_u$ (resp. $E_s$) define a Hölder-continuous section of the Grassmann bundle $G_{d_u, n}$ (resp. $G_{d_u, n}$) of vector subspaces of dimension $d_u$ (resp. $d_s$).  
      \end{itemize}      
\end{remark}
  
\section{Proof of De Rham's theorem}\label{Appendix: De Rham}

As we shall see, a short proof consists in using the fact that eigenvectors of Hodge-De Rham Laplace operators $\Delta^{(k)} = (d + d^*)^2$ acting on $L^2 (M ; \Lambda^k T^* M \otimes \C)$ are smooth by elliptic regularity. Let us only prove the first point. A proof of the second point can be found in \cite[p. 355]{schwartz}. If we denote by $\pi_0$ the spectral projector onto the eigeinspace of the eigenvalue $0$, then we have for all $0 \leq k \leq n$ and all $u \in \Hil_k^m$
 \begin{equation}\label{eq : proof of De Rham's thm}
    \begin{split}
     u &= \pi_0 (u) + (\Id - \pi_0) (u) \\
        &= \pi_0 (u) + \Delta \circ \Delta^{-1} \circ (\Id - \pi_0) (u) \\
        &= \pi_0 (u) + d \circ R (u) + R \circ d (u).
    \end{split}
 \end{equation}
 Here, $R$ denotes the homotopic operator defined by $R: u \in \Hil_k^m \mapsto d^* \circ \Delta^{-1} \circ (\Id - \pi_0) (u) \in d^* \left(\Hil_{k}^{m+2} \right) \subset \Hil_{k-1}^{m+1}$ where we used that $\Delta^{-1} \circ (\Id - \pi_0)$ is a pseudodifferential operator of order $-2$. The last equality is obtained using commutativity of $d$ with $\Delta^{-1}$ and $\pi_0$ (which follows from the commutativity of $d$ and $\Delta$). Now, fix an element $u \in \Hil_k^m$ such that $du = 0$. We deduce from (\ref{eq : proof of De Rham's thm}) the identity
 $$ u = \pi_0 (u) + d \circ R (u).$$
 By setting $\omega := \pi_0 (u) \in \Omega^k (M; \C)$ which is indeed smooth by smoothness of the eigenvectors of $\Delta^{(k)}$, we obtain the result. 
\end{appendix}
\begin{small}

\bibliographystyle{plain}
\bibliography{biblio}
\end{small}

\end{document}